\documentclass[preprint,12pt]{elsarticle}
\usepackage{amssymb}
\usepackage{amsfonts}
\usepackage{mathrsfs}
\usepackage{amsmath}
\usepackage{bm}
\usepackage{verbatim}
\usepackage{subfigure}
\usepackage{color}
\usepackage{tabularx}
\usepackage{booktabs} 
\usepackage{multirow} 
\usepackage[ruled,linesnumbered]{algorithm2e}
\usepackage{adjustbox} % 确保导入 adjustbox 包
\usepackage{graphicx} % 确保导入 graphicx 包
\usepackage{threeparttable} % 确保导入 threeparttable 包
\usepackage{amsthm}
\usepackage{longtable}
\usepackage{tikz}

\numberwithin{equation}{section}

\newtheorem{theorem}{Theorem}[section]

\newtheorem{lemma}[theorem]{Lemma}

\theoremstyle{definition}
\newtheorem{definition}[theorem]{Definition}

\newtheorem{remark}[theorem]{Remark}

\graphicspath{{figures/}}

\hfuzz=\maxdimen
\vfuzz=\maxdimen
\journal{}

\makeatletter
\def\ps@pprintTitle{%
	\let\@oddhead\@empty
	\let\@evenhead\@empty
	\let\@oddfoot\@empty
	\let\@evenfoot\@empty
}
\makeatother

\begin{document}

\begin{frontmatter}



\title{A Variational Nonlocal Phase-Field Model for Dynamic Fracture in Elastic Solids}

\author[inst1]{Qing Cheng}
\ead{qingcheng@tongji.edu.cn}

\author[inst1]{Yuqi Sun}
\ead{yuqisun@tongji.edu.cn}

\author[inst1]{Xuejun Xu}
\ead{xuxj@tongji.edu.cn}

%\author[inst2]{ Haitao Yu}
%\ead{yuhaitao@tongji.edu.cn}

%\cortext[cor2]{Co-corresponding author}

\affiliation[inst1]{organization={School of Mathematical Sciences, Key Laboratory of Intelligent Computing and Applications (Ministry of Education)}, 
addressline={Tongji University}, 
city={Shanghai},
postcode={200092}, 
country={China}}

\begin{comment}
\affiliation[inst2]{organization={State Key Laboratory of Disaster Reduction in Civil Engineering}, 
addressline={Tongji University}, 
city={Shanghai},
postcode={200092}, 
country={China}}
\end{comment}

\begin{abstract}
		We develop a variational nonlocal phase-field model for dynamic fracture in elastic solids. The proposed formulation is distinguished by three main features. First, the model is formulated through nonlocal kinematics and kernel-dependent function spaces, allowing weaker regularity requirements while recovering the classical local theory as the nonlocal interaction domain vanishes. Second, a nonlocal crack-surface functional is introduced as an integral counterpart of the Ambrosio--Tortorelli regularization, so that the characteristic length of the diffusive crack is implicitly determined by the nonlocal interaction domain rather than by a prescribed length scale. Third, the degraded nonlocal elastic energy and the nonlocal crack-surface functional are combined into a variationally consistent dynamic fracture system, consisting of a nonlocal momentum balance and an irreversible nonlocal gradient-flow evolution law for the phase field. The coupled system is solved using two temporal discretization strategies: a structure-preserving scalar auxiliary-variable scheme and a staggered alternating scheme, both combined with finite element discretization in space. Numerical examples involving Mode-I fracture, dynamic crack branching, Kalthoff--Winkler-type shear fracture, and fragmentation show that the proposed model captures complex crack initiation, propagation, branching, and interaction without explicit crack tracking. Quantitatively, the predicted crack-tip velocities remain below $0.6c_R$ in the dynamic branching and shear-loading tests, and the shear-loading benchmark gives an inclined crack path of approximately $48^\circ$, consistent with the characteristic Kalthoff--Winkler fracture pattern. 
\end{abstract}

\begin{comment}
\begin{graphicalabstract}
\includegraphics{grabs}
\end{graphicalabstract}
\end{comment}

\begin{comment}
\begin{highlights}
	\item A kernel-based variational nonlocal phase-field model is developed for dynamic fracture in elastic solids.
	\item A nonlocal crack-surface functional implicitly embeds the diffusive crack length scale and leads to an irreversible energy-dissipative formulation.
	\item Structure-preserving scalar auxiliary-variable and staggered alternating schemes are developed for complex dynamic fracture simulations.
\end{highlights}
\end{comment}


\end{frontmatter}



\section*{Nomenclature}

\begingroup
\begin{longtable}{@{}p{0.25\textwidth}p{0.68\textwidth}@{}}
	\toprule
	\textbf{Symbol} & \textbf{Description} \\
	\midrule
	\endfirsthead
	
	\toprule
	\textbf{Symbol} & \textbf{Description} \\
	\midrule
	\endhead
	
	\bottomrule
	\endfoot
	
	\multicolumn{2}{@{}l}{\textbf{Domains, boundaries, and geometric quantities}} \\
	\midrule
	$\Omega$ & Material domain in $\mathbb{R}^d$. \\
	$B_\delta(\mathbf{x})$ & Nonlocal interaction domain defined as the open ball centered at $\mathbf{x}$ with radius $\delta$. \\
	
	\midrule
	\multicolumn{2}{@{}l}{\textbf{Fields and kinematic quantities}} \\
	\midrule
	$\mathbf{u}$ & Displacement field. \\
	$\psi$ & Phase-field variable, with $\psi=0$ for intact material and $\psi=1$ for the cracked region. \\
	$\boldsymbol{\varepsilon}_\delta$ & Nonlocal infinitesimal strain tensor. \\
	$\boldsymbol{\sigma}_\delta$ & Nonlocal Cauchy stress tensor. \\
	
	\midrule
	\multicolumn{2}{@{}l}{\textbf{Nonlocal kernels and operators}} \\
	\midrule
	$\mathcal{G}_\delta(\mathbf{u})$ & Nonlocal displacement gradient. \\
	$\mathcal{D}_\delta(\boldsymbol{\sigma}_\delta)$ & Nonlocal divergence operator acting on the stress tensor. \\
	$\mathcal{N}_\delta(\boldsymbol{\sigma}_\delta)$ & Nonlocal boundary traction operator. \\
	$\mathcal{L}_\delta\psi$ & Nonlocal diffusion-type operator for the phase-field variable. \\
	
	\midrule
	\multicolumn{2}{@{}l}{\textbf{Crack-surface functional and phase-field evolution}} \\
	\midrule
	$\gamma_\delta(\psi,\mathcal{L}_\delta\psi)$ & Nonlocal crack-surface density functional. \\
	$\mathcal{A}_\delta(\psi,\mathcal{L}_\delta\psi)$ & Nonlocal crack-surface functional. \\
	$\mathcal{Y}(\psi,D)$ & Nonlocal crack-surface driving force. \\
	
	\midrule
	\multicolumn{2}{@{}l}{\textbf{Function spaces}} \\
	\midrule
	$\mathcal{V}_\omega(\Omega)$ & Kernel-dependent nonlocal function space for displacement fields. \\
	$\mathcal{M}_\omega(\Omega)$ & Kernel-dependent nonlocal function space for the phase-field variable. \\
	
\end{longtable}
\endgroup

\section{Introduction}
\label{sec:introduction}
Fracture of materials is a longstanding and fundamental problem in solid mechanics and continues to pose significant challenges for computational modeling. A central difficulty arises from the inherently multiscale and discontinuous nature of crack evolution, which involves damage nucleation, crack initiation, propagation, coalescence, branching, and strong topology changes. 

	Existing fracture models may be broadly viewed from two complementary paradigms. The first paradigm represents cracks at the discrete or algorithmic level. Cohesive element methods \cite{wells2001new} introduce traction--separation laws along potential crack surfaces, but they usually require predefined or adaptively inserted cohesive interfaces and may introduce artificial compliance and mesh dependence \cite{foulk2010examination,nguyen2014open}. Enrichment-based methods, such as XFEM \cite{moes1999finite,sukumar2000extended,daux2000arbitrary} and meshfree approaches \cite{rabczuk2004cracking,rabczuk2007three}, represent discontinuities by enriching the approximation space or relaxing mesh conformity. Although these methods are powerful for prescribed or evolving crack paths, robustly handling crack branching, merging, and complex three-dimensional topological changes remains algorithmically demanding \cite{wu2019computational}.

	The second paradigm describes fracture at the continuous field-equation level. Local damage models \cite{lemaitre2012course} introduce internal variables to describe material degradation, but they usually lack intrinsic length scales and may lead to mesh-dependent strain localization \cite{geers1998strain,peerlings2002localisation}. Nonlocal damage models and Peridynamic-type formulations regularize the problem by replacing local differential operators with spatial integral operators, thereby allowing finite-range interactions and discontinuous displacement fields \cite{peerlings2002localisation,bavzant2002nonlocal,silling2000reformulation,silling2007peridynamic}. However, many such nonlocal fracture formulations are not derived from a variational energy principle, and their associated crack-surface energy structure, variational consistency, and connection with the local sharp-crack limit remain insufficiently clarified.

	Phase-field fracture models \cite{bourdin2000numerical,bourdin2008variational,miehe2010phase,borden2012phase} also belong to the continuous-level paradigm. They represent cracks as diffusive damage zones governed by a variational energy principle and naturally avoid explicit crack tracking \cite{Wu2020}. Nevertheless, phase-field fracture models require a prescribed length scale to control the width of the diffusive crack. The relationship between this macroscopic length parameter and microscopic fracture processes, such as microcrack nucleation, void growth, and bond breakage, remains largely unresolved \cite{budarapu2019multiscale}.

	These observations motivate the development of a variational nonlocal phase-field formulation at the continuous field-equation level. The aim is to combine the finite-range interaction structure of nonlocal continuum models with the energetic crack-surface regularization of phase-field fracture, while avoiding explicit crack tracking at the discrete level. In this work, we develop such a model for dynamic fracture in elastic solids. The main contributions are summarized as follows:

\begin{itemize}
	
	\item
		We formulate a variational nonlocal phase-field fracture model based on nonlocal kinematics and kernel-dependent function spaces for the displacement and phase-field variables. This setting allows weaker regularity requirements than classical local continuum models while retaining consistency with the local theory in the vanishing nonlocal interaction domain limit.
	
	\item
		We introduce a nonlocal crack-surface functional as an integral counterpart of the Ambrosio--Tortorelli regularization. In this formulation, the characteristic length of the diffusive crack is implicitly determined by the nonlocal interaction domain rather than prescribed as an independent internal length scale. Combining this functional with the degraded nonlocal elastic energy yields a variationally consistent system consisting of a nonlocal momentum balance and an irreversible nonlocal gradient-flow evolution law for the phase field.
	
	\item We develop two numerical strategies for the coupled system: a structure-preserving
	scalar auxiliary variable (SAV) scheme and a staggered alternating scheme combining Newmark-type time integration and implicit phase-field evolution. Both schemes are implemented with finite element spatial discretization. Numerical examples, including Mode-I fracture, dynamic crack branching, Kalthoff--Winkler-type shear fracture, and fragmentation, demonstrate the capability of the proposed model to capture complex dynamic crack patterns without explicit crack tracking.
	
\end{itemize}

	The remainder of the paper is organized as follows. Section~\ref{sec:phase_field_model} presents the proposed nonlocal phase-field formulation, including the nonlocal kinematics, the nonlocal crack-surface functional, and the resulting variational governing equations. Section~\ref{sec:numerical_scheme} develops the SAV scheme and staggered alternating schemes and the corresponding finite element discretization. Section~\ref{sec:numerical_eg} provides dynamic fracture simulations to assess the performance of the proposed model. Finally, concluding remarks are provided at the end of the paper.

\section{Nonlocal phase-field formulation for crack propagation of solids}
\label{sec:phase_field_model}
\noindent This section presents the theoretical framework of the proposed nonlocal phase-field formulation for crack propagation in elastic solids. Section~\ref{subsec:nonlocal_kinematics} introduces key phenomenological concepts of the proposed nonlocal continuum mechanics. By applying the variational principle, the governing equation for displacement field is derived. In Section~\ref{subsec:nonlocal_geometric_functional}, a nonlocal geometric functional for diffusive cracks is proposed, where the implicit crack length scale is determined by the nonlocal characteristic length. Building on these foundations, a phase-field system of equations for diffusive crack propagation is developed in Section~\ref{subsec:nonlocal_gover_eq}. For brevity, we may omit the explicit dependence on the time variable $t$ whenever it is clear from the context.

\subsection{Governing equations of nonlocal continua}
\label{subsec:nonlocal_kinematics}
\noindent In this framework of nonlocal continuum mechanics, the motion of a continuum is
influenced not only by local fields but also by the neighborhood of each material
point. In this paper, the continuum $\Omega\subset\mathbb{R}^d$ ($d=2,3$) is considered as a bounded Lipschitz domain, whose boundary $\partial\Omega$ is a $(d-1)$-dimensional Lipschitz manifold. Its associated \emph{nonlocal boundary} is defined by
\begin{equation}
	\Gamma_\delta = \big\{ \mathbf{x} \in \Omega\,\big|\, d(\mathbf{x}, \partial\Omega) \leq \delta \big\},
\end{equation}
where $d(\mathbf{x}, \partial\Omega) = \inf\{  |\mathbf{x} - \mathbf{y}|\, |\, {\mathbf{y} \in \partial\Omega} \}$ denotes the Euclidean distance from $\mathbf{x}$ to $\partial\Omega$.

Figure~\ref{fig:ref-cur-config} schematically illustrates the motion of the continuum $\Omega$ from its reference configuration $\Omega_0$ to the current configuration $\Omega_t$ over the time interval $[0,T]$. Following the idea of nonlocal operators\cite{gunzburger2010nonlocal,du2013nonlocal}, we define the \emph{nonlocal deformation
	gradient tensor} $\mathcal{F}_\delta$ by
\begin{equation}
	\label{eq:F}
	\mathcal{F}_\delta(\bm{\phi})(\mathbf{x},t)=
	\int_{B_\delta(\mathbf{x})}
	\big(\boldsymbol{\phi}(\mathbf{y},t)-\boldsymbol{\phi}(\mathbf{x},t)\big)
	\otimes \boldsymbol{\omega}(\mathbf{x},\mathbf{y}) \, \mathrm{d}v_\mathbf{y},
	\quad \forall(\mathbf{x},t)\in \Omega_0\times[0,T],
\end{equation}
where $\boldsymbol{\phi}:\Omega_0\times[0,T]\to\mathbb{R}^d$ is the motion mapping, and
$B_\delta(\mathbf{x})$ denotes the open ball centered at $\mathbf{x}$ with radius $\delta$, which characterizes the nonlocal interaction length scale. $\boldsymbol{\omega}(\mathbf{x},\mathbf{y})=(\mathbf{y}-\mathbf{x})\omega_\delta(|\mathbf{y}-\mathbf{x}|)$ is an antisymmetric kernel function, which should satisfy the following assumptions:
\begin{equation}
	\label{assump:vector-kernel}
	\begin{aligned}
		\text{[\textbf{A1}]}\;&\omega_\delta:\mathbb{R}^d\rightarrow\mathbb{R}_{+} \,\text{with}\, \operatorname{supp}\,\omega_\delta =\overline{B_\delta(\mathbf{0})},\\
		[\textbf{A2}]\;& \int_{B_\delta(\mathbf{0})} \mathbf{x} \otimes\bm{\omega}(\mathbf{0},\mathbf{x})\, \mathrm{d}v_\mathbf{x}=\mathbf{I}_d,\\
	\end{aligned}
\end{equation}
where $\mathbf{I}_d\in\mathbb{R}^{d\times d}$ is the $d\times d$ identity matrix.

\begin{remark}
There exist many admissible choices for the scalar kernel function 
$\omega_\delta(|\mathbf{x}|)$ satisfying assumptions~[\textbf{A1}] and~[\textbf{A2}]. Two widely used classes of kernels in nonlocal models—such as nonlocal diffusion 
and Peridynamics~\cite{silling2000reformulation,silling2007peridynamic}—are the family 
of the Gaussian kernels and singular kernels with algebraic decay of order $d+2\beta$ ($\beta\in(0,1)$), and both types of kernels are supported in 
$B_\delta(\mathbf{0})$. Here, $\beta$ controls the strength of the near-field singularity of the algebraically decaying kernel. One representative Gaussian-type example is given by
\begin{equation}\label{eq:kernel}
	\begin{aligned}
		\omega_{\delta}(|\mathbf{x}|)
		&=c_\alpha\exp\left(-\frac{\alpha}{\delta^d}|\mathbf{x}|^d\right)\chi_{[0,\delta)}(|\mathbf{x}|),
	\end{aligned}
\end{equation}
where $\alpha\in(0,\infty)$ is the shape parameter controlling the decay rate of the Gaussian-type kernel, and $c_\alpha=\frac{d}{|\partial B_1(\mathbf{0})|\int_{0}^{\delta}\exp\left(-\frac{\alpha}{\delta^d}r^d\right)r^{d+1}\, \mathrm{d}r}$ is a normalization constant chosen such that $\omega_\delta$ satisfies assumption~[\textbf{A2}].
\end{remark}

\begin{figure}[htbp]
	\centering
	\begin{tikzpicture}
		
		\node[anchor=south west, inner sep=0] (img) 
		at (0,0) {\includegraphics[width=0.9\textwidth]{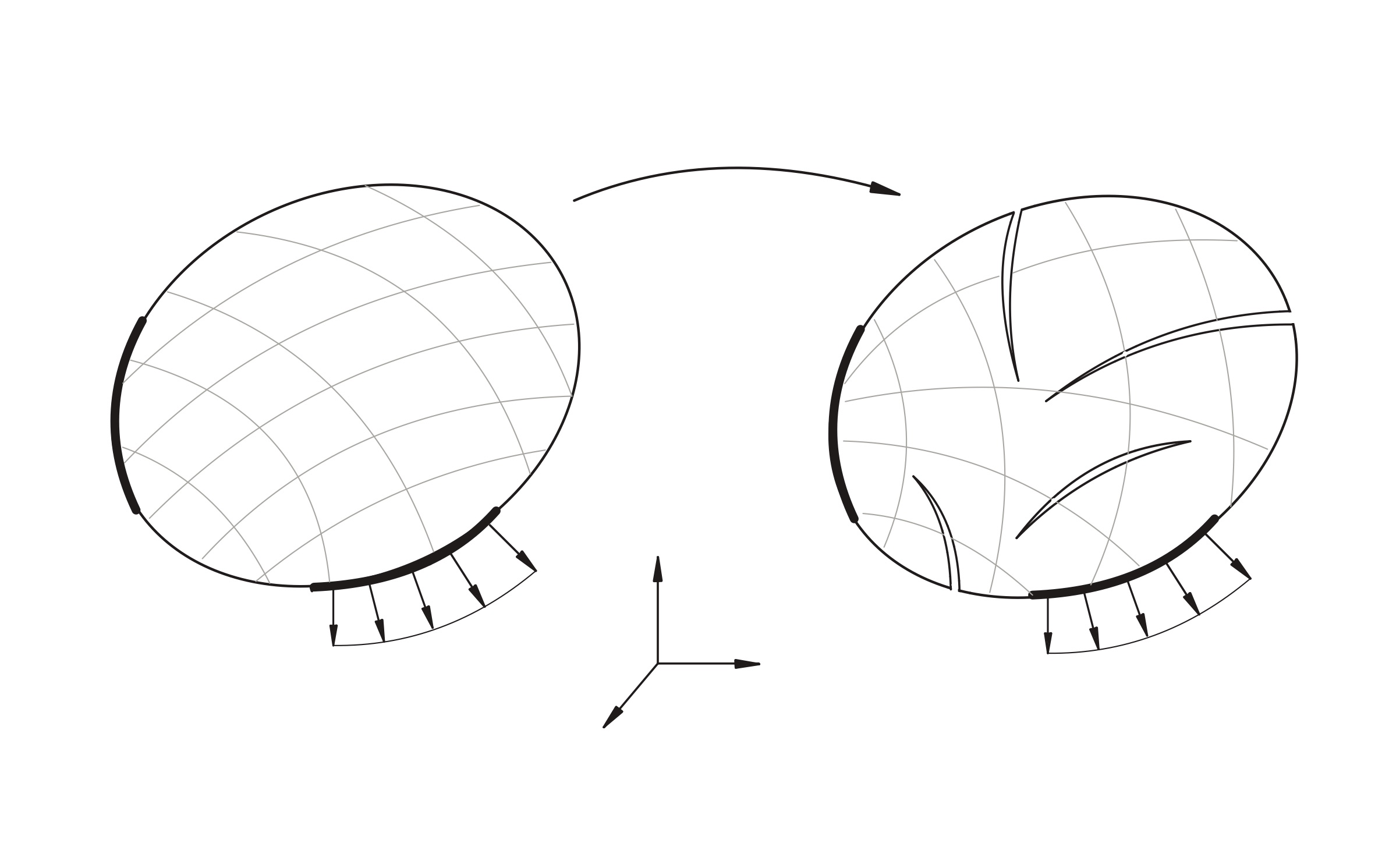}};
		\begin{scope}[x={(img.south east)}, y={(img.north west)}]
			\node at (0.35,0.7) {$\Omega_0$};
			\node at (0.86,0.7) {$\Omega_t$};
			\node at (0.52,0.85) {$\bm{\phi}(\cdot,t)$};
			\node at (0.35,0.24) {$\Gamma_{\delta,t}$};
			\node at (0.9,0.24) {$\bm{\phi}(\Gamma_{\delta,t},t)$};
			\node at (0.12,0.5) {$\Gamma_{\delta,u}$};
			\node at (0.67,0.5) {$\bm{\phi}(\Gamma_{\delta,u},t)$};
			% 坐标轴
			\node at (0.5,0.35) {$x_3$};
			\node at (0.55,0.2) {$x_1$};
			\node at (0.45,0.13) {$x_2$};
		\end{scope}
	\end{tikzpicture}
	\caption{Schematic of a nonlocal continuum illustrating the motion mapping from the referential configuration $\Omega_0$ to the current configuration $\Omega_t$ under discontinuous deformation.}
	\label{fig:ref-cur-config}
\end{figure}
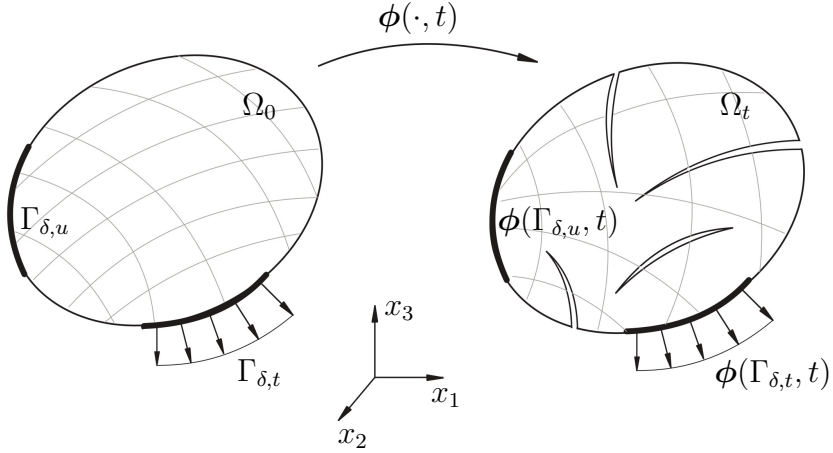

The tensor $\mathcal F_\delta$ provides an averaged nonlocal measure of
deformation and can be regarded as a nonlocal counterpart of the classical
deformation gradient $\nabla_{\mathbf x}\boldsymbol{\phi}$. Under sufficient
smoothness assumptions on the motion mapping $\boldsymbol{\phi}$, the
normalization condition in~\eqref{assump:vector-kernel} implies the local
consistency
\begin{equation}\label{eq:local_limit}
	\lim_{\delta\to0}\mathcal F_\delta(\boldsymbol{\phi})(\mathbf{x},t)
	=
	\nabla_{\mathbf x}\boldsymbol{\phi}(\mathbf{x},t).
\end{equation}

In this work, we focus on small-strain dynamic fracture. Therefore, the
distinction between the reference and current configurations is neglected, and
the kinematics is formulated directly in terms of the displacement field
$\mathbf u(\mathbf x,t)$. For $\boldsymbol{\phi}(\mathbf x,t)
=\mathbf x+\mathbf u(\mathbf x,t)$, the displacement contribution to
$\mathcal F_\delta$ motivates the following \emph{nonlocal displacement
	gradient}:
\begin{equation}
	\label{eq:Gu}
	\mathcal{G}_\delta(\mathbf{u})(\mathbf{x},t)
	=
	\int_{B_\delta(\mathbf{x})}
	\big(
	\mathbf{u}(\mathbf{y},t)
	-
	\mathbf{u}(\mathbf{x},t)
	\big)
	\otimes
	\boldsymbol{\omega}(\mathbf{x},\mathbf{y})\,
	\mathrm{d}v_\mathbf{y},
	\quad
	(\mathbf{x},t)\in\Omega\times(0,T).
\end{equation}
For sufficiently smooth displacement fields, one has
$\mathcal G_\delta(\mathbf u)\to\nabla\mathbf u$ as $\delta\to0$.

The \emph{nonlocal infinitesimal strain tensor} is then defined as the
symmetric part of the nonlocal displacement gradient:
\begin{equation}\label{eq:strain}
	\boldsymbol{\varepsilon}_\delta(\mathbf{u})
	=
	\frac{1}{2}
	\left(
	\mathcal{G}_\delta(\mathbf{u})
	+
	\mathcal{G}_\delta(\mathbf{u})^\top
	\right).
\end{equation}
Accordingly, $\boldsymbol{\varepsilon}_\delta(\mathbf u)$ recovers the
classical infinitesimal strain tensor in the local limit. For a homogeneous, isotropic, linearly elastic material, the linearized
\emph{nonlocal elastic strain energy density} is defined by
\begin{equation}
	\mathcal{W}_{\delta,e}(\mathbf{u})
	=
	\frac{\lambda}{2}
	\big(\operatorname{tr}\boldsymbol{\varepsilon}_\delta(\mathbf{u})\big)^2
	+
	\mu\,
	\boldsymbol{\varepsilon}_\delta(\mathbf{u})
	:
	\boldsymbol{\varepsilon}_\delta(\mathbf{u}),
\end{equation}
where $\lambda$ and $\mu$ are the Lamé constants, and $:$ denotes the double
contraction of second-order tensors. The corresponding \emph{nonlocal Cauchy
	stress} is given by
\begin{equation}\label{eq:stress}
	\bm{\sigma}_{\delta}(\mathbf{u})
	=
	\frac{\partial\mathcal{W}_{\delta,e}(\mathbf{u})}
	{\partial\boldsymbol{\varepsilon}_\delta}
	=
	\lambda
	\operatorname{tr}\boldsymbol{\varepsilon}_\delta(\mathbf{u})\,\mathbf I_d
	+
	2\mu\,\boldsymbol{\varepsilon}_\delta(\mathbf{u}).
\end{equation}

Before formulating the governing equations of nonlocal continuum mechanics, we introduce the functional space in which the displacement field $\mathbf{u}(\mathbf{x},t)$ is defined.

\begin{definition}
	Let $\Omega\subset \mathbb{R}^d$ be a bounded Lipschitz domain and let 
	$\omega_\delta: \mathbb{R}^d \to [0,\infty)$ be a kernel function satisfies assumptions [\textbf{A1}] and [\textbf{A2}] in Eq.\eqref{assump:vector-kernel}.
	Define the nonlocal seminorm
	\begin{equation}
		[\mathbf{u}]_{\mathcal{V}_\omega(\Omega)}^2
		=\int_{\Omega} \int_{\Omega} 
		|\mathbf{u}(\mathbf{y}) - \mathbf{u}(\mathbf{x})|^2 \, 
		\omega_\delta(|\mathbf{y}-\mathbf{x}|)\, \mathrm{d}v_\mathbf{y}\mathrm{d}v_\mathbf{x}, 
	\end{equation}
then the nonlocal function space is given by
	\begin{equation}
		\mathcal{V}_\omega(\Omega)
		= \big\{\, \mathbf{u}\in L^2(\Omega; \mathbb{R}^d) \; \big|\; [\mathbf{u}]_{\mathcal{V}_\omega(\Omega)} < \infty \, \big\},
	\end{equation}
	equipped with the norm $		\|\mathbf{u}\|_{\mathcal{V}_\omega(\Omega)}^2
	= \|\mathbf{u}\|_{L^2(\Omega)}^2 + [\mathbf{u}]_{\mathcal{V}_\omega(\Omega)}^2$.
\end{definition}

\begin{lemma}\label{lem:nonlocal-gradient-estimate}
	Let $\Omega\subset\mathbb{R}^d$ be a bounded Lipschitz domain and 
	$\omega_\delta$ satisfy assumptions {[\textbf{A1}]}–{[\textbf{A2}]} in~\eqref{assump:vector-kernel}.
	For any $\mathbf{u}\in\mathcal{V}_\omega(\Omega)$, the nonlocal gradient $\mathcal{G}_\delta(\mathbf{u})(\cdot,t)$ $ (t\in[0,T])$ in Eq.\eqref{eq:Gu} is well-defined as an element of $L^2(\Omega;\mathbb{R}^{d\times d})$ and there exists a constant 
	$C_\omega=\int_{B_{\delta}(\mathbf{0})}|\bm{\xi}|^2\omega_\delta(|\bm{\xi}|)
	\,\mathrm{d}v_{\bm{\xi}}$, depending only on the kernel function, such that
	\begin{equation}\label{eq:nonlocal-G-upper}
		\|\mathcal{G}_\delta(\mathbf{u})\|_{L^2(\Omega)}^2
		\le C_\omega\,[\mathbf{u}]_{\mathcal{V}_\omega(\Omega)}^2,
		\quad \forall\,\mathbf{u}\in\mathcal{V}_\omega(\Omega).
	\end{equation}
\end{lemma}

\begin{remark}
Lemma~\ref{lem:nonlocal-gradient-estimate} shows that the nonlocal displacement gradient $\mathcal{G}_\delta(\mathbf{u})$ is controlled by the nonlocal seminorm of the displacement field. This estimate ensures that the nonlocal strain and the corresponding elastic energy are well defined for displacement fields in $\mathcal{V}_\omega(\Omega)$, even when classical spatial derivatives may not exist. The proof is given in~\ref{app:proof-nonlocal-gradient-estimate}. 
\end{remark}

We now derive the nonlocal equation of motion from Hamilton's principle. Let $\Omega\subset\mathbb{R}^d$ be a bounded Lipschitz domain with the nonlocal boundary decomposition $\Gamma_{\delta} = \Gamma_{\delta,t} \cup \Gamma_{\delta,u}$, where $\Gamma_{\delta,t}\neq\emptyset$, $\Gamma_{\delta,u}\neq\emptyset$ and $ \text{int}(\Gamma_{\delta,t})\cap \text{int}(\Gamma_{\delta,u}) = \emptyset$. Assume that the displacement field $\mathbf{u}(\mathbf{x},t)$ satisfies
\begin{equation}
   \mathbf{u} \in L^2(0,T;\mathcal{V}_{\omega,\mathbf u_g}(\Omega)), 
	\quad 
	\dot{\mathbf{u}} \in H^1(0,T;L^2(\Omega;\mathbb{R}^d)),
\end{equation}
	where $\dot{\mathbf{u}}=\partial_t\mathbf{u}$ and $\mathcal{V}_{\omega,\mathbf u_g}(\Omega)$ is defined by 
	\begin{equation}
\mathcal{V}_{\omega,\mathbf u_g}(\Omega)=\{\mathbf{u}(\cdot,t)\in\mathcal{V}_{\omega}(\Omega):\mathbf{u}(\cdot,t) = \mathbf{u}_g(t)\, \text{on}\,  \Gamma_{\delta,u}\}.
	\end{equation}
	
The nonlocal Lagrangian functional is defined by
	\begin{equation}
		\mathcal{L}_\delta(\mathbf{u},\dot{\mathbf{u}})
		= \int_{\Omega}
		\left(
		\frac{1}{2}\rho\,|\dot{\mathbf{u}}|^2
		- \mathcal{W}_{\delta,e}(\mathbf{u})
		+ \mathbf{b}\cdot\mathbf{u}
		\right)\mathrm{d}v
		- \int_{\Gamma_{\delta,t}}\mathbf{t}\cdot\mathbf{u}\,\mathrm{d}v,
	\end{equation}
	where $\rho\in L^\infty(\Omega;\mathbb{R}_{+})$ denotes the mass density,
	$\mathbf{b}\in L^2(\Omega;\mathbb{R}^{d})$ the body force,
	and $\mathbf{t}\in L^2(\Gamma_{\delta,t};\mathbb{R}^{d})$ the prescribed nonlocal traction. Further, the corresponding action functional is given by
	\begin{equation}
		\mathcal{S}[\mathbf{u}]
		= \int_{0}^{T}\mathcal{L}_\delta(\mathbf{u},\dot{\mathbf{u}})\,\mathrm{d}t.
	\end{equation}
Then the principle of stationary action, $\delta\mathcal{S}[\mathbf{u}]= 0$, leads to the following Euler–Lagrange equation governing the motion of the system as
	\begin{equation}\label{eq:motion}
		\rho\,\partial_{t}^2\mathbf{u}
		= \mathcal{D}_{\delta}(\bm{\sigma}_\delta)
		+ \mathbf{b},
		\quad
		\forall(\mathbf{x},t)\in(\Omega\setminus\Gamma_{\delta})\times[0,T],
	\end{equation}
	with the nonlocal traction boundary condition
	\begin{equation}\label{eq:bc}
		\mathcal{N}_{\delta}(\bm{\sigma}_\delta)
		= \mathbf{t},
		\quad
		\forall(\mathbf{x},t)\in\Gamma_{\delta,t}\times[0,T].
	\end{equation}
	Here $\mathcal{D}_{\delta}$ and $\mathcal{N}_{\delta}$ denote the nonlocal divergence and nonlocal boundary operators, respectively, defined by
	\begin{equation}
		\begin{aligned}
			\mathcal{D}_{\delta}(\bm{\sigma}_\delta)(\mathbf{x},t)
			&= \int_{\Omega}
			\big(\bm{\sigma}_\delta(\mathbf{x},t)
			+ \bm{\sigma}_\delta(\mathbf{y},t)\big)
			\cdot\bm{\omega}(\mathbf{x},\mathbf{y})\,
			\mathrm{d}v_\mathbf{y},\ \forall(\mathbf{x},t)\in(\Omega\setminus\Gamma_{\delta})\times[0,T],\\
			\mathcal{N}_{\delta}(\bm{\sigma}_\delta)(\mathbf{x},t)
			&= -\int_{\Omega}
			\big(\bm{\sigma}_\delta(\mathbf{x},t)
			+ \bm{\sigma}_\delta(\mathbf{y},t)\big)
			\cdot\bm{\omega}(\mathbf{x},\mathbf{y})\,
			\mathrm{d}v_\mathbf{y},\ \forall(\mathbf{x},t)\in\Gamma_{\delta}\times[0,T],
		\end{aligned}
	\end{equation}
where $\bm{\sigma}_\delta(\mathbf{x},t)$ denotes $\bm{\sigma}_\delta(\mathbf{u})(\mathbf{x},t)$.

\begin{remark}
Eqs.~\eqref{eq:strain}, \eqref{eq:stress}, \eqref{eq:motion} and \eqref{eq:bc},
together with the initial conditions 
\begin{equation}
	\mathbf{u}(\mathbf{x},0)=\mathbf{u}_0(\mathbf{x}),\, 
	\dot{\mathbf{u}}(\mathbf{x},0)=\mathbf{v}_0(\mathbf{x}),\, \mathbf{x}\in\Omega,
\end{equation}
constitute the complete initial--nonlocal boundary value problem governing 
the nonlocal dynamic behavior of the isotropic elastic solids.
\end{remark}

Now, we demonstrate that the nonlocal continuum model proposed in this section converges to the classical isotropic homogeneous elastic mechanics as the nonlocal parameter $\delta$ approaches zero.

\subsection{Nonlocal geometric functional for diffusive cracks}\label{subsec:nonlocal_geometric_functional}
\noindent In this paper, the sharp crack is approximated by a \emph{diffuse} crack representation within a nonlocal phase-field framework. We introduce an order parameter $\psi(\cdot,t)\in L^2(\Omega; [0,1])$ $(t\in[0,T])$ to describe the diffusive crack topology, whose values characterize the degree of damage in the solid: $\psi=0$ corresponds to intact material, whereas $\psi=1$ denotes the crack region (see Figure~\ref{fig:diffusive_crack}). Define the \emph{nonlocal crack surface density functional} as
\begin{equation}
	\gamma_\delta(\psi, \mathcal{L}_\delta\psi)
	=\frac{1}{2\delta}\psi^2+\frac{\delta}{2}\psi\,\mathcal{L}_\delta\psi,
\end{equation}
where the nonlocal operator $\mathcal{L}_\delta\psi$ is given by
\begin{equation}
	\mathcal{L}_\delta\psi(\mathbf{x},t)
	=\int_{B_\delta(\mathbf{x})}\omega_\delta(|\mathbf{y}-\mathbf{x}|)
	\big(\psi(\mathbf{x},t)-\psi(\mathbf{y},t)\big)\,\mathrm{d}v_\mathbf{y},
	\quad \forall(\mathbf{x},t)\in\Omega\times[0,T].
\end{equation}
Accordingly, the total diffusive crack surface is characterized by the 
\emph{nonlocal crack surface functional} (see Figure~\ref{fig:diffusive_crack}):
\begin{equation}
	\begin{aligned}
		\mathcal{A}_\delta(\psi,\mathcal{L}_\delta\psi)
		&=\int_{\Omega}\gamma_\delta(\psi,\mathcal{L}_\delta\psi)\,\mathrm{d}v \\
		&=\int_{\Omega}\frac{1}{2\delta}\psi^2\,\mathrm{d}v
		+\frac{\delta}{4}\int_{\Omega}\int_{\Omega}\omega_\delta(|\mathbf{y}-\mathbf{x}|)
		\big(\psi(\mathbf{x},t)-\psi(\mathbf{y},t)\big)^2
		\,\mathrm{d}v_\mathbf{x}\mathrm{d}v_\mathbf{y}.
        \label{eq:nonlocal_crack_functional}
	\end{aligned}
\end{equation}

\begin{remark}
	The second term in Eq.~\eqref{eq:nonlocal_crack_functional},
	\begin{equation}
		\frac{\delta}{4}\int_{\Omega}\int_{\Omega}
		\omega_\delta(|\mathbf{y}-\mathbf{x}|)
		\big(\psi(\mathbf{x},t)-\psi(\mathbf{y},t)\big)^2
		\,\mathrm{d}v_\mathbf{x}\mathrm{d}v_\mathbf{y},
	\end{equation}
	can be interpreted as a nonlocal counterpart of the gradient regularization term
	in the classical Ambrosio--Tortorelli functional. Indeed, for a sufficiently smooth
	phase-field variable $\psi$, the difference
	$\psi(\mathbf{y},t)-\psi(\mathbf{x},t)$ approximates
	$\nabla\psi(\mathbf{x},t)\cdot(\mathbf{y}-\mathbf{x})$ when
	$|\mathbf{y}-\mathbf{x}|$ is small. Therefore, the above nonlocal quadratic form
	plays the same regularizing role as $\frac{\ell_c}{2}\int_{\Omega}|\nabla\psi|^2\,\mathrm{d}v$ in the classical Ambrosio--Tortorelli crack-surface functional \cite{bourdin2008variational}. The main
	difference is that, in the present formulation, the regularization is induced by
	the kernel function and the nonlocal interaction length scale$\delta$, rather than by an
	explicitly prescribed length scale $\ell_c$.
\end{remark}

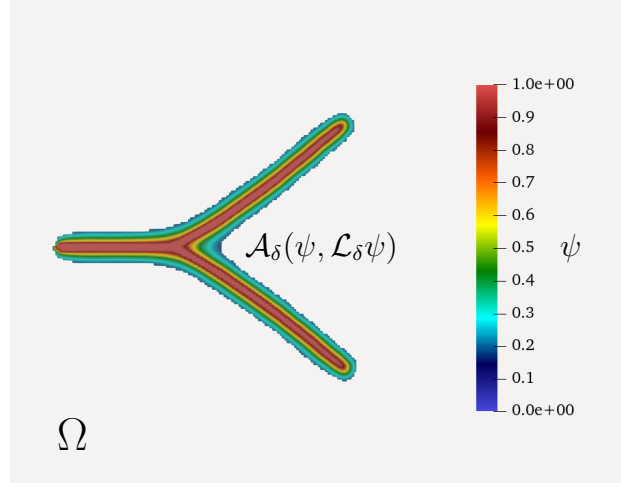
\begin{figure}[htbp]
	\centering
	\begin{tikzpicture}
\node[anchor=south west, inner sep=0] (img) 
    at (0,0) {\includegraphics[width=0.6\textwidth]{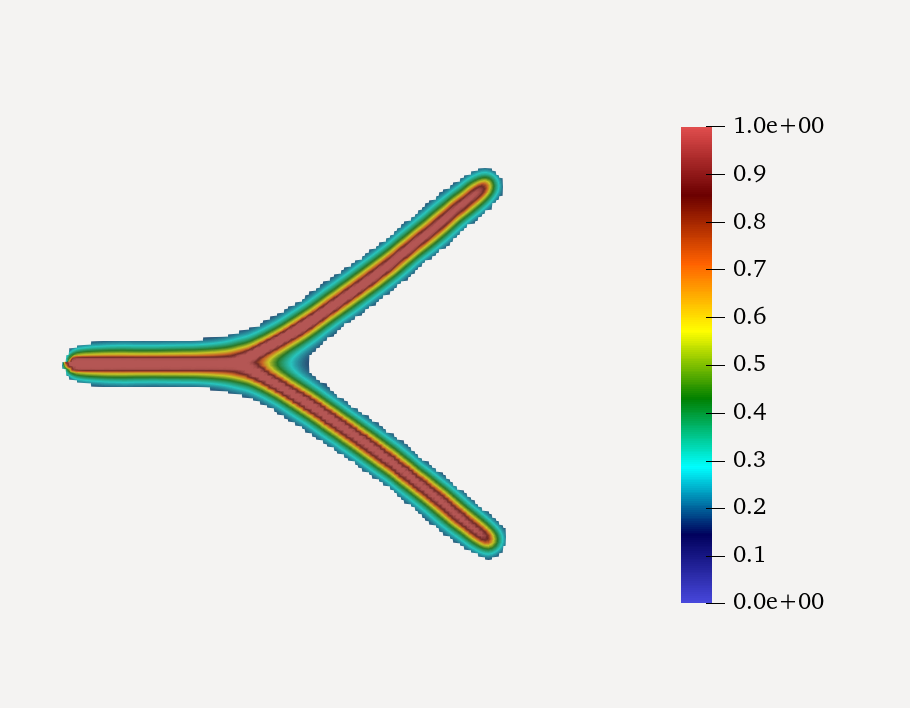}};
\begin{scope}[x={(img.south east)}, y={(img.north west)}]
\node at (0.1,0.1) {\Large $\Omega$};
\node at (0.9,0.485) {$\psi$};
\node at (0.5,0.485) {$\mathcal{A}_\delta(\psi,\mathcal{L}_\delta\psi)$};
\end{scope}
\end{tikzpicture}
  \caption{Schematic illustration of a diffusive crack characterized by the nonlocal crack-surface functional $\mathcal{A}_\delta(\psi,\mathcal{L}_\delta\psi)$.}
	\label{fig:diffusive_crack}
\end{figure}

\begin{definition}
Let $\Omega\subset \mathbb{R}^d$ be a bounded Lipschitz domain and let 
	$\omega_\delta: \mathbb{R}^d \to [0,\infty)$ be a kernel function satisfies assumption [\textbf{A1}] in Eq.\eqref{assump:vector-kernel}. For all $\psi, \phi\in L^2(\Omega)$, define the $\mathcal{M}_\omega(\Omega)$-inner product as\cite{zhou2010mathematical,du2013analysis} 
	
		\begin{equation}
		(\psi,\phi)_{\mathcal{M}_{\omega}(\Omega)}
		=\int_{\mathbb{R}^d}
		\overline{\hat{\phi}(\boldsymbol{\xi})}\,
		(1+m_\omega(\boldsymbol{\xi}))\,
		\hat{\psi}(\boldsymbol{\xi})\,\mathrm{d}v_{\boldsymbol{\xi}},
	\end{equation}
where $\psi, \phi$ are understood as their zero extension outside $\Omega$, $\hat{\psi}(\bm{\xi})$ and $\hat{\phi}(\bm{\xi})$ are the Fourier transform of $\psi$ and $\phi$, respectively, and $m_\omega(\bm{\xi})$ is defined by
	\begin{equation}
		m_\omega(\bm{\xi})=\int_{B_\delta(\mathbf{0})}\omega_\delta(|\mathbf{y}|)(1-\cos(2\pi\bm{\xi}\cdot\mathbf{y}))\ \mathrm{d}v_{\mathbf{y}}.
	\end{equation}
The function space $\mathcal{M}_\omega(\Omega)$ is then defined as
	\begin{equation}
	\mathcal{M}_{\omega}(\Omega)
	=\Big\{
	\psi\in L^2(\Omega)
	\ \Big|\
	\|\psi\|_{\mathcal{M}_{\omega}(\Omega)}<\infty
	\Big\}, 
\end{equation}
where $\|\psi\|_{\mathcal{M}_{\omega}(\Omega)}=\sqrt{(\psi,\psi)_{\mathcal{M}_{\omega}(\Omega)}}$.
\end{definition}

We next consider the first variation of the nonlocal crack-surface functional $\mathcal A_\delta(\psi,\mathcal L_\delta\psi)$. For each fixed time $t\in(0,T)$, assume that $\psi(\cdot,t)\in\mathcal M_\omega(\Omega)$. Taking the first variation of $\mathcal A_\delta$ with respect to $\psi$ gives the associated Euler--Lagrange condition \begin{equation}\label{eq:order_parameter} 
\psi+\delta^2\mathcal{L}_\delta\psi=0, \quad (\mathbf{x},t)\in\Omega\times(0,T). 
\end{equation}

\begin{remark}
In the local phase-field fracture model~\cite{Bourdin2000}, the order parameter 
$\psi$ describing the diffusive crack is governed by the elliptic equation:
\begin{equation}\label{eq:local_phase_field}
	\psi - \ell_c^{\,2}\Delta \psi = 0,
	\quad \forall\,(\mathbf{x},t)\in\Omega\times[0,T],
\end{equation}
where $\ell_c$ denotes the intrinsic length scale that controls the width of the 
diffusive crack. 
In the one-dimensional setting, the solution to 
Eq.~\eqref{eq:local_phase_field} takes the form 
$\psi(x)=\exp\!\left(-\ell_c^{-1}|x|\right)$, indicating that the length 
scale $\ell_c$ enters the solution explicitly. 

In contrast, in the proposed nonlocal model, the effective length scale 
associated with the diffusive crack is no longer imposed explicitly; 
instead, it is implicitly encoded in the nonlocal operator $\mathcal{L}_\delta$. 
This implicit dependence can be illustrated through the following 
one-dimensional example.
\end{remark}

Consider the following one-dimensional illustrative example:
\begin{equation}\label{eq:psi_1d}
	\left\{
	\begin{aligned}
		&\psi(x) 
		+ \delta^{2}\int_{x-\delta}^{\,x+\delta}
		\omega_\delta(|y-x|)\big(\psi(x)-\psi(y)\big)\,\mathrm{d}y = 0,
		\quad \forall\, x\in\mathbb{R},\\[4pt]
		&\psi(0)=1,\quad \lim_{|x|\to\infty}\psi(x)=0,
	\end{aligned}
	\right.
\end{equation}
where the kernel function $\omega_\delta$ is taken as the Gaussian kernel defined in 
Eq.~\eqref{eq:kernel} with $\alpha=(d+1)^2$. The solution to this one-dimensional problem takes the exponential form
\begin{equation}
	\psi(x)=\exp(\lambda_\delta |x|),
\end{equation}
where $\lambda_\delta<0$ is the negative root of the nonlinear algebraic equation
\begin{equation}\label{eq:nonlinear_eq}
	1+2\delta^{2}\int_{0}^{\delta}
	\omega_\delta(r)\bigl(1-\cosh(\lambda_\delta r)\bigr)\,\mathrm{d}r=0.
\end{equation}

\begin{figure}[h]
	\centering	
	\subfigure[]{
		\includegraphics[width=0.475\textwidth]{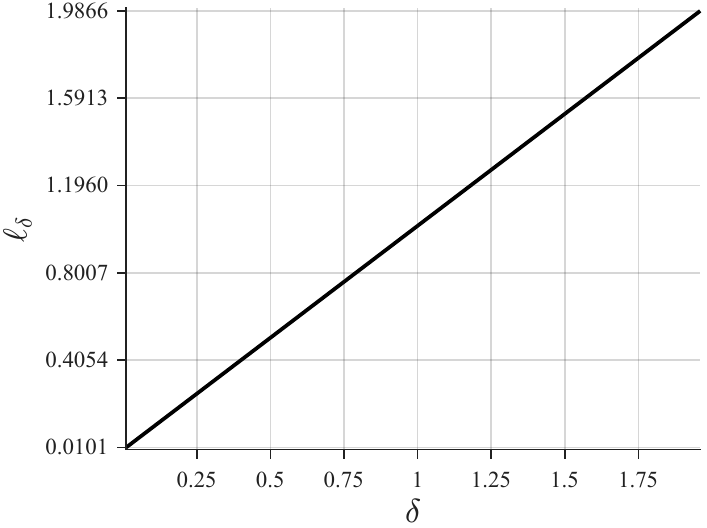}
		\label{fig:delta_vs_lc}
	}
	\subfigure[]{
		\includegraphics[width=0.475\textwidth]{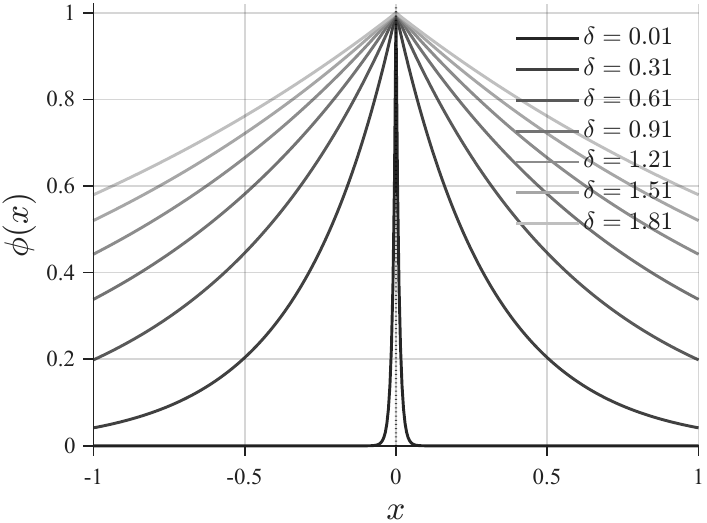}
		\label{fig:phi_profiles}
	}
	\caption{(a) Relationship between the nonlocal characteristic length $\delta$ and the diffusive-crack length scale $\ell_\delta$. (b) Profiles of the phase-field function $\psi(x)=\exp(\lambda_\delta |x|)$ for selected values of $\delta$.}
	\label{fig:delta_vs_lc_phi_profiles}
\end{figure}

By solving Eq.~\eqref{eq:nonlinear_eq}, one obtains the relationship between the characteristic length scale of the diffusive crack, defined by $\ell_\delta=|\lambda_\delta|^{-1}$, and the nonlocal characteristic length $\delta$, as shown in Fig.~\ref{fig:delta_vs_lc_phi_profiles}(a). It can be observed that $\ell_\delta$ decreases as $\delta$ decreases and satisfies $\ell_\delta\to0$ as $\delta\to0$. The phase-field profiles in Fig.~\ref{fig:delta_vs_lc_phi_profiles}(b) further illustrate that the diffusive crack becomes increasingly localized for smaller $\delta$, indicating that it converges to a sharp crack in the local limit.

Equation~\eqref{eq:order_parameter} characterizes the geometric configuration of the diffusive crack, but alone does not describe how it initiates or propagates. Therefore, it is necessary to introduce an
appropriate constitutive law that governs the evolution of the diffusive crack.
Following the approach proposed in~\cite{Miehe2015a}, we define the 
\emph{nonlocal crack surface driving functional} as
\begin{equation}
	\mathcal{S}_d(\psi,\dot{\psi},\mathcal{Y})
	= \int_{\Omega}\frac{1}{\delta}\big(\mathcal{Y}(\psi,D)-\bar{\eta}\dot{\psi}\big)\dot{\psi}\,\mathrm{d}v,
	\quad (\mathbf{x},t)\in\Omega\times[0,T],
\end{equation}
where $\dot{\psi}=\partial_t\psi$, $\bar{\eta}\in\mathbb{R}_{\ge 0}$ is a viscous regularization parameter characterizing the rate-dependent viscous resistance to crack propagation, $\mathcal{Y}(\psi,D)$ denotes the 
\emph{nonlocal crack surface driving force}, and $D$ is an \emph{associated driving state function}
depending on the nonlocal elastic energy $\mathcal{W}_{\delta,e}(\mathbf{u})$.  
The details of $\mathcal{Y}(\psi,D)$ and $D$ will be discussed in 
Section~\ref{subsec:nonlocal_gover_eq}.

The constitutive law governing the evolution of the order parameter $\psi$ is given by
\begin{equation}
	\frac{\mathrm{d}}{\mathrm{d}t}\mathcal{A}_\delta(\psi,\mathcal{L}_\delta\psi)
	= \mathcal{S}_d(\psi,\dot{\psi},\mathcal{Y}),
\end{equation}
where $\frac{\mathrm{d}}{\mathrm{d}t}\mathcal{A}_\delta(\psi,\mathcal{L}_\delta\psi)
= \int_{\Omega}\frac{1}{\delta}\big(\psi+\delta^{2}\mathcal{L}_\delta\psi\big)
\,\dot{\psi}\,\mathrm{d}v$. 

Consequently, the time evolution equation for $\psi$ is obtained as
\begin{equation}
	\bar{\eta}\dot{\psi}=\mathcal{Y}(\psi,D)-(\psi + \delta^{2}\mathcal{L}_\delta\psi),
	\quad \forall\,(\mathbf{x},t)\in\Omega\times[0,T].
\end{equation}

\begin{remark}
	To ensure the physical irreversibility of crack growth, the phase-field 
	variable must satisfy the constraint
	\begin{equation}
	\dot{\psi}(\mathbf{x},t)\ge 0,
	\quad \forall\,(\mathbf{x},t)\in\Omega\times[0,T].
	\end{equation}
	This condition prevents the healing of cracks and guarantees that the 
	degraded stiffness does not recover once damage has occurred.
\end{remark}

\subsection{Constitutive choices and Governing equations of nonlocal phase-field fracture model}\label{subsec:nonlocal_gover_eq}
\noindent We now introduce a unified nonlocal phase-field formulation for diffusive crack propagation in solids. We first present the constitutive description and then introduce the governing equations for the coupled nonlocal phase-field system.

Inspired by the spectral tension-compression decomposition of elastic strain energy introduced by Miehe et al.~\cite{Miehe2010}, we extend this concept to the nonlocal setting and define the \emph{degraded nonlocal spectral tension--compression energy} as
\begin{equation}
\mathcal{W}_{\delta,de}(\mathbf{u},\psi)
= g(\psi)\,\mathcal{W}_{\delta,e}^{+}(\mathbf{u})
+ \mathcal{W}_{\delta,e}^{-}(\mathbf{u}),
\label{eq:energy_model_II}
\end{equation}
where the tensile and compressive parts of the energy, $\mathcal{W}_{\delta,e}^{\pm}$, are defined by
\begin{equation}
\mathcal{W}_{\delta,e}^{\pm}
= \frac{\lambda}{2}\,\langle \operatorname{tr}\boldsymbol{\varepsilon}_\delta \rangle_{\pm}^{2}
+ \mu \sum_{a=1}^{d} \langle \varepsilon_{\delta,a} \rangle_{\pm}^{2}.
\end{equation}
Here, $\operatorname{tr}\boldsymbol{\varepsilon}_\delta$ denotes the trace of the nonlocal strain tensor $\boldsymbol{\varepsilon}_\delta(\mathbf{u})$, 
$\{\varepsilon_{\delta,a}\}_{a=1}^{d}$ are its principal values, and the Macaulay brackets are defined as $\langle x \rangle_{\pm} = \frac{x \pm |x|}{2}.$

The corresponding degraded nonlocal Cauchy stress $\boldsymbol{\sigma}_{\delta}^{\mathrm{d}}$ is obtained as
\begin{equation}
\boldsymbol{\sigma}_\delta^{\mathrm{d}}(\mathbf{u},\psi)
= \frac{\partial \mathcal{W}_{\delta,de}(\mathbf{u},\psi)}{\partial \boldsymbol{\varepsilon}_{\delta}}
= g(\psi)\,\boldsymbol{\sigma}_\delta^{+}
+ \boldsymbol{\sigma}_\delta^{-},
\end{equation}
where the tensile and compressive stress components are given by
\begin{equation}
\boldsymbol{\sigma}_\delta^{\pm}
= \lambda \langle \operatorname{tr}\boldsymbol{\varepsilon}_\delta \rangle_{\pm}\mathbf{I}
+ 2\mu \sum_{a=1}^{d} \langle \varepsilon_{\delta,a} \rangle_{\pm}\,
\mathbf{n}_a \otimes \mathbf{n}_a,
\end{equation}
and $\{\mathbf{n}_a\}_{a=1}^{d}$ are the orthonormal eigenvectors associated with the principal strains $\{\varepsilon_{\delta,a}\}_{a=1}^{d}$.

The nonlocal crack surface driving force \(\mathcal{Y}(\psi, D)\) is expressed as
\begin{equation}
	\mathcal{Y}(\psi, D) = (1 - \psi)D,
\end{equation}
where the associated driving state function $D$ is chosen according to the fracture regime:
\begin{equation}
	D=
	\begin{cases}
		\displaystyle \frac{2\mathcal{W}_{\delta,e}^{+}(\mathbf{u}(\mathbf{x},t))}{G_c/\delta},
		& \text{for rate-dependent fracture }(\bar{\eta}>0),\\[8pt]
		\displaystyle \max_{s\in[0,t]}\frac{2\mathcal{W}_{\delta,e}^{+}(\mathbf{u}(\mathbf{x},s))}{G_c/\delta},
		& \text{for rate-independent fracture }(\bar{\eta}=0).
	\end{cases}
\end{equation}
Here, $G_c$ is the critical energy release rate of solid materials. There are several widely used associated driving state functions in the phase-field fracture literature; interested readers can refer to \cite{Miehe2015a}.

\begin{definition}\label{def:coupled_system}
Suppose the displacement field $\mathbf{u} : \Omega\times[0,T]\to\mathbb{R}^d$ and the phase-field variable $\psi : \Omega\times[0,T]\to\mathbb{R}_{\ge0}$ satisfy the regularity assumptions $\mathbf{u} \in L^2(0,T;\mathcal{V}_{\omega,\mathbf u_g}(\Omega))$, $\dot{\mathbf{u}} \in H^1(0,T;L^2(\Omega;\mathbb{R}^d))$ and $\psi \in H^1\big(0,T;\mathcal{M}_\omega(\Omega)\big)$. Then the evolution of diffusive cracks is described by the following coupled
nonlocal system:
	\begin{itemize}
		\item[(i)] \textbf{Balance of linear momentum:}
		\begin{equation}\label{eq:def_motion}
			\rho\,\partial_t^2\mathbf{u}
			= \mathcal{D}_\delta\big(\bm{\sigma}_{\delta}^{\mathrm{d}}\big)
			+ \mathbf{b},
			\quad
			\text{in }(\Omega\setminus\Gamma_\delta)\times[0,T].
		\end{equation}
		\item[(ii)] \textbf{Phase-field evolution equation:}
		
\begin{equation}\label{eq:def_phase_system}
	\left\{
	\begin{aligned}
	&\bar{\eta}\dot{\psi}=
	\mathcal{Y}(\psi,D) -(\psi + \delta^{2}\mathcal{L}_\delta\psi),\quad \text{in }\Omega\times[0,T], \\
	&\dot{\psi}\ge0,\quad \text{in }\Omega\times[0,T].
		\end{aligned}
		\right.
\end{equation}

\item[(iii)] \textbf{Nonlocal traction and displacement boundary conditions:}
		\begin{equation}
			\mathcal{N}_\delta\big(\bm{\sigma}_{\delta}^{\mathrm{d}}\big)
			= \mathbf{t},\,\text{on }\Gamma_{\delta,t}\times[0,T],
			\quad
			\mathbf{u} = \mathbf{u}_g,
			\,\text{on }\Gamma_{\delta,u}\times[0,T].
		\end{equation}
		\item[(iv)] \textbf{Initial conditions:}
		\begin{equation}
				\left\{\begin{aligned}
					&\mathbf{u}(\mathbf{x},0)=\mathbf{u}_{0}(\mathbf{x}),\,
		\dot{\mathbf{u}}(\mathbf{x},0)=\mathbf{v}_{0}(\mathbf{x}),\quad\forall\mathbf{x}\in\Omega,\\
		&\psi(\mathbf{x},0)=\psi_{0}(\mathbf{x}),
		\quad \forall\mathbf{x}\in\Omega.
			\end{aligned}
			\right.
		\end{equation}
	\end{itemize}
\end{definition}

\begin{remark}
A rigorous analytical theory for existence, uniqueness, and regularity of the proposed nonlocal phase-field fracture model is beyond the scope of this paper.
Nevertheless, the numerical simulations in Section~\ref{sec:numerical_eg} consistently yield stable, reproducible, and physically meaningful responses, providing empirical evidence of the computational robustness and practical solvability of the formulation in the tested regimes.
These observations should be interpreted as computational support only and do not constitute a mathematical proof of well-posedness.
A complete theoretical analysis will be addressed in future work.
\end{remark}

\begin{theorem}\label{thm:boundness_psi}
If $\psi\in H^1\big(0,T;C(\overline{\Omega})\big)$ is a solution to the rate-dependent $(\bar{\eta}>0)$ nonlocal phase-field equation with the explicit irreversibility constraint
\begin{equation}
	\left\{
	\begin{aligned}
		&\bar{\eta}\dot{\psi}=
		\mathcal{Y}(\psi,D) -(\psi + \delta^{2}\mathcal{L}_\delta\psi),\quad \text{in }\Omega\times[0,T], \\
		&\dot{\psi}\ge0,\quad \text{in }\Omega\times[0,T],
	\end{aligned}
	\right.
\end{equation}
with initial condition $0\leq \psi(\mathbf{x},0)\leq 1$ for all $\mathbf{x}\in\Omega$, then $\psi(\mathbf{x},t)\in[0,1]$ for all $(\mathbf{x},t)\in\Omega\times[0,T]$.
\end{theorem}
\begin{proof}
The lower bound follows directly from the irreversibility constraint.
Since $\dot\psi\ge 0$ in $\Omega\times[0,T]$ and $\psi(\mathbf{x},0)\ge 0$ in $\Omega$, we have $
\psi(\mathbf{x},t)\ge \psi(\mathbf{x},0)\ge 0,$ $\forall(\mathbf{x},t)\in\Omega\times[0,T]$. 

It remains to prove $\psi\le 1$. Fix any $t\in[0,T]$, and let
$\mathbf{x}_t\in\overline\Omega$ be a point such that $
\psi(\mathbf{x}_t,t)=\max_{\mathbf{x}\in\overline\Omega}\psi(\mathbf{x},t)$.
Assume by contradiction that $\psi(\mathbf{x}_t,t)>1$. From the definition of $\mathcal{L}_\delta\psi(\mathbf{x},t)$ and $\omega_\delta\ge 0$, we obtain at the maximizer $\mathbf{x}_t$ that
$\psi(\mathbf{x}_t,t)-\psi(\mathbf{y},t)\ge 0$ for all
$\mathbf{y}\in B_\delta(\mathbf{x}_t)$, hence $
\mathcal{L}_\delta\psi(\mathbf{x}_t,t)\ge 0$.
Therefore,
\begin{equation}
\psi(\mathbf{x}_t,t)+\delta^2\mathcal{L}_\delta\psi(\mathbf{x}_t,t)
\ge \psi(\mathbf{x}_t,t)>1.
\end{equation}

On the other hand, $D(\mathbf{x},t)\ge 0$ by definition of the associated driving state function, and
$\mathcal{Y}(\psi,D)=(1-\psi)D$. Since $\psi(\mathbf{x}_t,t)>1$, we have
\begin{equation}
\mathcal{Y}(\psi,D)(\mathbf{x}_t,t)
=(1-\psi(\mathbf{x}_t,t))D(\mathbf{x}_t,t)\le 0.
\end{equation}
Substituting into
\begin{equation}
\bar{\eta}\dot\psi
=\mathcal{Y}(\psi,D)-\big(\psi+\delta^2\mathcal{L}_\delta\psi\big),
\end{equation}
yields
\begin{equation}
\bar{\eta}\dot\psi(\mathbf{x}_t,t)
\le 0-\big(\psi(\mathbf{x}_t,t)+\delta^2\mathcal{L}_\delta\psi(\mathbf{x}_t,t)\big)<0,
\end{equation}
which contradicts the irreversibility condition $\dot\psi\ge 0$.
Hence $\max_{\mathbf{x}\in\overline\Omega}\psi(\mathbf{x},t)\le 1$ for all $t\in[0,T]$, i.e.,
$\psi(\mathbf{x},t)\le 1$ in $\Omega\times[0,T]$.

Combining the two bounds gives $
0\le \psi(\mathbf{x},t)\le 1$, $ \forall(\mathbf{x},t)\in\Omega\times[0,T]$.

\end{proof}

\begin{remark}
In Definition~\ref{def:coupled_system}, the phase-field variable is assumed to satisfy $\psi\in H^1\big(0,T;\mathcal{M}_\omega(\Omega)\big)$, whereas Theorem~\ref{thm:boundness_psi} is proved under the stronger assumption $\psi\in H^1\big(0,T;C(\overline{\Omega})\big)$. This additional regularity is introduced to justify a pointwise maximum argument. A systematic characterization of the relation between $\mathcal{M}_\omega(\Omega)$ and $C(\overline{\Omega})$ (for example, embedding or compactness properties under assumptions on $\omega_\delta$ and $\Omega$) is beyond the scope of the present work and will be investigated in future work.
\end{remark}

\begin{theorem}\label{thm:Hamiltonian_balance}
	Assume that $(\mathbf{u},\psi)$ is a solution of the
	rate-dependent $(\bar{\eta}>0)$ coupled nonlocal phase-field fracture system in
	Definition~\ref{def:coupled_system}. Assume a possibly nonhomogeneous but time-independent prescribed displacement
	on $\Gamma_{\delta,u}$, and assume that the body force $\mathbf b$ and the
	prescribed nonlocal traction $\mathbf t$ are also time-independent. Define the Hamiltonian functional
	\begin{equation}
		\mathcal{H}(t)
		=
		\int_{\Omega\setminus\Gamma_\delta}
		\frac{\rho}{2}|\dot{\mathbf u}|^2\,\mathrm dv
		+
		\int_{\Omega}
		\mathcal{W}_{\delta,de}(\mathbf u,\psi)\,\mathrm dv
		+
		G_c\mathcal{A}_\delta(\psi,\mathcal L_\delta\psi)-\ell_u(\mathbf u),
	\end{equation}
	where $ \ell_u(\mathbf u)=	(\mathbf b,\mathbf u)_{L^2(\Omega\setminus\Gamma_\delta)}
	+
	(\mathbf t,\mathbf u)_{\Gamma_{\delta,t}}$ denotes the external loading
	potential. Then $\mathcal H(t)$ satisfies the Hamiltonian balance law
\begin{equation}
	\frac{\mathrm d}{\mathrm dt}\mathcal H(t)
	=-\eta\int_\Omega |\dot\psi|^2\,\mathrm dv
	\le 0,
\end{equation}
where $\eta=\frac{\bar{\eta} G_c}{\delta}$. 
\end{theorem}

\begin{proof}
	The proof is based on testing the coupled system by the corresponding rate
	variables. Multiplying the nonlocal momentum balance by $\dot{\mathbf u}$ and
	using the nonlocal integration-by-parts identity together with the nonlocal
	traction condition gives
	\begin{equation}\label{eq:kinetic_balance_for_energy_thm}
		\begin{aligned}
			\frac{\mathrm d}{\mathrm dt}
			\int_{\Omega\setminus\Gamma_\delta}
			\frac{\rho}{2}|\dot{\mathbf u}|^2\,\mathrm d v
			&=
			-\int_\Omega
			\boldsymbol\sigma_\delta^{\mathrm d}:
			\dot{\boldsymbol\varepsilon}_\delta
			\,\mathrm d v
			+
			(\mathbf b,\dot{\mathbf u})_{L^2(\Omega\setminus\Gamma_\delta)}
			+
			(\mathbf t,\dot{\mathbf u})_{\Gamma_{\delta,t}} .
		\end{aligned}
	\end{equation}
	Since the external loading potential is defined by
	\begin{equation}
		\ell_u(\mathbf u)
		=
		(\mathbf b,\mathbf u)_{L^2(\Omega\setminus\Gamma_\delta)}
		+
		(\mathbf t,\mathbf u)_{\Gamma_{\delta,t}},
	\end{equation}
	and the external data are assumed to be time-independent, we have
	\begin{equation}\label{eq:external_potential_derivative}
		\frac{\mathrm d}{\mathrm dt}\ell_u(\mathbf u)
		=
		(\mathbf b,\dot{\mathbf u})_{L^2(\Omega\setminus\Gamma_\delta)}
		+
		(\mathbf t,\dot{\mathbf u})_{\Gamma_{\delta,t}} .
	\end{equation}
	
	On the other hand, the chain rule for the degraded nonlocal elastic energy,
	combined with the phase-field evolution equation, yields
	\begin{equation}\label{eq:strain_fracture_balance_for_energy_thm}
		\begin{aligned}
			&
			\frac{\mathrm d}{\mathrm dt}
			\int_\Omega
			\mathcal W_{\delta,de}(\mathbf u,\psi)\,\mathrm d v
			+
			G_c
			\frac{\mathrm d}{\mathrm dt}
			\mathcal A_\delta(\psi,\mathcal L_\delta\psi)
			\\
			&\qquad
			=
			\int_\Omega
			\boldsymbol\sigma_\delta^{\mathrm d}:
			\dot{\boldsymbol\varepsilon}_\delta
			\,\mathrm d v
			-
			\frac{\bar{\eta}G_c}{\delta}
			\int_\Omega |\dot\psi|^2\,\mathrm d v .
		\end{aligned}
	\end{equation}
	
	Adding \eqref{eq:kinetic_balance_for_energy_thm} and
	\eqref{eq:strain_fracture_balance_for_energy_thm}, and then subtracting
	\eqref{eq:external_potential_derivative}, the stress-power terms and the
	external-power terms cancel. Therefore,
	\begin{equation}
		\frac{\mathrm d}{\mathrm dt}\mathcal H(t)
		=
		-
		\frac{\bar{\eta}G_c}{\delta}
		\int_\Omega |\dot\psi|^2\,\mathrm d v
		=
		-\eta
		\int_\Omega |\dot\psi|^2\,\mathrm d v
		\le 0,
	\end{equation}
	where $\eta:=\bar{\eta}G_c/\delta$. Hence, $\mathcal H(t)$ is non-increasing
	in time. The detailed algebraic derivation of
	\eqref{eq:kinetic_balance_for_energy_thm} and
	\eqref{eq:strain_fracture_balance_for_energy_thm} is provided in
	Appendix~\ref{app:energy-dissipation-proof}.
	
\end{proof}

\section{Numerical scheme for nonlocal phase-field fracture models}\label{sec:numerical_scheme}
\noindent This section is devoted to the numerical approximation of the coupled nonlocal system governing diffusive crack propagation introduced in Section~\ref{subsec:nonlocal_gover_eq}. Section~\ref{subsec:auxiliary_variable} introduces an auxiliary-variable reformulation for the nonlinear energy contribution.  Section~\ref{subsec:spa_temp} presents a staggered alternating time-discretization strategy, and Section~\ref{subsec:spa_dis} describes the finite element discretization in space.

\subsection{Structure-preserving time discretization}
\label{subsec:auxiliary_variable}

\noindent
The degraded nonlocal elastic energy introduces nonlinear coupling between the displacement field and the phase-field variable. To obtain a linearly implicit time-discrete formulation while retaining a discrete energy-dissipation
structure, we introduce a scalar auxiliary variable associated with the shifted nonlinear energy.

Let
\begin{equation}\label{eq:shifted_energy_A}
	A(\mathbf u,\psi)
	=
	\int_\Omega
	\mathcal W_{\delta,de}(\mathbf u,\psi)\,\mathrm d v
	-
	\ell_u(\mathbf u)
	+
	C_0,
\end{equation}
where $C_0>0$ is chosen sufficiently large such that
$A(\mathbf u,\psi)>0$.

The scalar auxiliary variable is defined by
\begin{equation}\label{eq:def_auxiliary_r}
	r(t) = \sqrt{A(\mathbf u(t),\psi(t))},
\end{equation}
and the corresponding modified total energy is then written as
\begin{equation}\label{eq:modified_energy_continuous}
	\mathcal H(t)
	=
	\frac{\rho}{2}
	\|\dot{\mathbf u}(t)\|_{L^2(\Omega\setminus\Gamma_\delta)}^2
	+
	G_c\mathcal A_\delta(\psi(t),\mathcal L_\delta\psi(t))
	+
	r(t)^2
	-
	C_0 .
\end{equation}

For later use, we denote the variational derivatives of $A$ by $	\mathbf F_u(\mathbf u,\psi)=
\frac{\delta A}{\delta\mathbf u}(\mathbf u,\psi)$ and  $F_\psi(\mathbf u,\psi)=
\frac{\delta A}{\delta\psi}(\mathbf u,\psi)$. Equivalently, for any admissible test functions
$\mathbf w\in\mathcal V_{\omega,\mathbf0}(\Omega)$ and
$\varphi\in\mathcal M_\omega(\Omega)$,
\begin{equation}\label{eq:def_Fu_Fpsi_weak}
	\begin{aligned}
		\left\langle\mathbf F_u(\mathbf u,\psi),\mathbf w\right\rangle
		&=
		\left(
		\boldsymbol\sigma_\delta^{\mathrm d}(\mathbf u,\psi),
		\boldsymbol\varepsilon_\delta(\mathbf w)
		\right)_{L^2(\Omega)}
		-
		\ell_u(\mathbf w),\\
		\left\langle F_\psi(\mathbf u,\psi),\varphi\right\rangle
		&=
		\int_\Omega
		\frac{\partial\mathcal W_{\delta,de}(\mathbf u,\psi)}
		{\partial\psi}
		\varphi\,\mathrm d v .
	\end{aligned}
\end{equation}
For the tensile--compressive split
$\mathcal W_{\delta,de}(\mathbf u,\psi)
=
g(\psi)\mathcal W_{\delta,e}^{+}(\mathbf u)
+
\mathcal W_{\delta,e}^{-}(\mathbf u)$, the second relation becomes
\begin{equation}
	\left\langle F_\psi(\mathbf u,\psi),\varphi\right\rangle
	=
	\int_\Omega
	g'(\psi)\mathcal W_{\delta,e}^{+}(\mathbf u)\,
	\varphi\,\mathrm d v .
\end{equation}
Using \eqref{eq:def_auxiliary_r}, the time derivative of the auxiliary variable
satisfies
\begin{equation}\label{eq:r_time_derivative}
	\dot r
	=
	\frac{1}{2\sqrt{A(\mathbf u,\psi)}}
	\left(
	\left\langle\mathbf F_u(\mathbf u,\psi),\dot{\mathbf u}\right\rangle
	+
	\left\langle F_\psi(\mathbf u,\psi),\dot\psi\right\rangle
	\right).
\end{equation}

Therefore, the weak form of the coupled system can be reformulated in terms of
$(\mathbf u,\psi,r)$ as
\begin{equation}\label{eq:continuous_auxiliary_system}
	\left\{
	\begin{aligned}
		&\rho(\ddot{\mathbf u},\mathbf w)_{L^2(\Omega\setminus\Gamma_\delta)}
		+
		\frac{r}{\sqrt{A(\mathbf u,\psi)}}
		\left\langle\mathbf F_u(\mathbf u,\psi),\mathbf w\right\rangle
		=0,\\
		&\eta(\dot\psi,\varphi)_{L^2(\Omega)}
		+
		G_c\,a_\delta(\psi,\varphi)
		+
		\frac{r}{\sqrt{A(\mathbf u,\psi)}}
		\left\langle F_\psi(\mathbf u,\psi),\varphi\right\rangle
		=0,\\
		&\dot r
		=
		\frac{1}{2\sqrt{A(\mathbf u,\psi)}}
		\left(
		\left\langle\mathbf F_u(\mathbf u,\psi),\dot{\mathbf u}\right\rangle
		+
		\left\langle F_\psi(\mathbf u,\psi),\dot\psi\right\rangle
		\right),
	\end{aligned}
	\right.
\end{equation}
where $	a_\delta(\psi,\varphi)=2\mathcal A_\delta(\psi,\mathcal L_\delta\psi)$. 

Let $\Delta t>0$ be the time step and $t^k=k\Delta t$ for $k=0,1,\cdots,\frac{T}{\Delta t}$. We denote by $\mathbf{u}^k$, $\psi^k$ and $r^k$ the numerical approximations of $\mathbf{u}(\cdot,t^k)$, $\psi(\cdot,t^k)$ and $r(t^k)$, respectively, and set $ \mathbf{b}^k= \mathbf{b}(\cdot,t^k)$ and  $ \mathbf{t}^k= \mathbf{t}(\cdot,t^k)$. For a sequence
$\{z^k\}_{k\ge0}$, we denote $D_t z^{k+1}=\frac{z^{k+1}-z^k}{\Delta t}$. 

Given $(\mathbf u^k,\mathbf v^k,\psi^k,r^k)$, where
$\mathbf v^k=\dot{\mathbf u}^k$, we set $A^k=A(\mathbf u^k,\psi^k)$, $\mathbf F_u^k=\mathbf F_u(\mathbf u^k,\psi^k)$ and $	F_\psi^k=F_\psi(\mathbf u^k,\psi^k)$. The auxiliary-variable ratio is defined by $	\xi^{k+1}=\frac{r^{k+1}}{\sqrt{A^k}}$. A first-order linearly implicit auxiliary-variable scheme is given as follows:
find $(\mathbf u^{k+1},\mathbf v^{k+1},\widetilde{\psi}^{k+1},r^{k+1})$ such that
\begin{equation}\label{eq:first_order_auxiliary_scheme}
	\left\{
	\begin{aligned}
		&\mathbf u^{k+1}
		=
		\mathbf u^k+\Delta t\,\mathbf v^{k+1},
		\\
		&\rho
		\big(D_t\mathbf v^{k+1},\mathbf w\big)_{L^2(\Omega\setminus\Gamma_\delta)}
		+
		\xi^{k+1}
		\left\langle
		\mathbf F_u^k,\mathbf w
		\right\rangle
		=0,
		\\
		&\eta
		\big(D_t\widetilde{\psi}^{k+1},\varphi\big)_{L^2(\Omega)}
		+
		G_c\,a_\delta(\widetilde{\psi}^{k+1},\varphi)
		+
		\xi^{k+1}
		\left\langle
		F_\psi^k,\varphi
		\right\rangle
		=0,
		\\
		&D_t r^{k+1}
		=
		\frac{1}{2\sqrt{A^k}}
		\left(
		\left\langle
		\mathbf F_u^k,D_t\mathbf u^{k+1}
		\right\rangle
		+
		\left\langle
		F_\psi^k,D_t\widetilde{\psi}^{k+1}
		\right\rangle
		\right),
	\end{aligned}
	\right.
\end{equation}
for all $\mathbf w\in\mathcal V_{\omega,\mathbf0}(\Omega)$ and $\varphi\in\mathcal{M}(\Omega)$. 
The irreversibility constraint is then enforced pointwise by the projection
\begin{equation}\label{eq:irreversibility_projection_auxiliary}
	\psi^{k+1}(\mathbf{x})
	=
	\max\left\{
	\widetilde{\psi}^{k+1}(\mathbf{x}),
	\psi^k(\mathbf{x})
	\right\},\quad
	\mathbf{x}\in\Omega .
\end{equation}

For a fixed value of $\xi^{k+1}$, the scheme
\eqref{eq:first_order_auxiliary_scheme} is linear. To express this more
compactly, we introduce the linear operators $\mathsf M_u$,
$\mathsf M_\psi$, and $\mathsf H_\psi$ by $\langle \mathsf M_u\mathbf v,\mathbf w\rangle=
\rho(\mathbf v,\mathbf w)_{L^2(\Omega\setminus\Gamma_\delta)}$, $\langle \mathsf M_\psi\chi,\varphi\rangle=
\eta(\chi,\varphi)_{L^2(\Omega)}$ and $\langle \mathsf H_\psi\chi,\varphi\rangle=
G_c\,a_\delta(\chi,\varphi)$ for all admissible test functions $\mathbf w$ and $\varphi$. Define $	\mathsf B_\psi=
\frac{1}{\Delta t}\mathsf M_\psi+\mathsf H_\psi $. The velocity and phase-field updates are decomposed as
\begin{equation}\label{eq:auxiliary_linear_decomposition}
	\mathbf v^{k+1}
	=
	\mathbf v_0^{k+1}
	+
	\xi^{k+1}\mathbf v_1^{k+1},
	\quad
	\widetilde{\psi}^{k+1}
	=
	\psi_0^{k+1}
	+
	\xi^{k+1}\psi_1^{k+1},
\end{equation}
where the two components are defined by the linear operator equations
\begin{equation}
	\begin{aligned}
		\mathsf M_u\mathbf v_0^{k+1}
		&=
		\mathsf M_u\mathbf v^k,
		&
		\mathsf M_u\mathbf v_1^{k+1}
		&=
		-\Delta t\,\mathbf F_u^k,
		\\
		\mathsf B_\psi\psi_0^{k+1}
		&=
		\frac{1}{\Delta t}\mathsf M_\psi\psi^k,
		&
		\mathsf B_\psi\psi_1^{k+1}
		&=
		-F_\psi^k.
	\end{aligned}
\end{equation}

The displacement is then recovered from the backward Euler kinematic relation
\begin{equation}
	\mathbf u^{k+1}
	=
	\mathbf u_0^{k+1}
	+
	\xi^{k+1}\mathbf u_1^{k+1},
\end{equation}
where $\mathbf u_0^{k+1}=
\mathbf u^k+\Delta t\,\mathbf v_0^{k+1}$ and $	\mathbf u_1^{k+1}=
\Delta t\,\mathbf v_1^{k+1}$.

Substituting the decompositions
\eqref{eq:auxiliary_linear_decomposition} into the auxiliary-variable equation
gives
\begin{equation}
	r^{k+1}
	=
	r^k
	+
	\frac{1}{2\sqrt{A^k}}
	\left[
	q_0+\xi^{k+1}q_1
	\right],
\end{equation}
where $q_0=
\Delta t\,
\left\langle
\mathbf F_u^k,\mathbf v_0^{k+1}
\right\rangle
+
\left\langle
F_\psi^k,
\psi_0^{k+1}-\psi^k
\right\rangle$ and $	q_1=
\Delta t\,
\left\langle
\mathbf F_u^k,\mathbf v_1^{k+1}
\right\rangle
+
\left\langle
F_\psi^k,
\psi_1^{k+1}
\right\rangle$.

Using $r^{k+1}=\xi^{k+1}\sqrt{A^k}$, we obtain the scalar algebraic equation
\begin{equation}\label{eq:xi_auxiliary_explicit}
	\xi^{k+1}
	=
	\frac{
		\dfrac{r^k}{\sqrt{A^k}}
		+
		\dfrac{q_0}{2A^k}
	}{
		1-\dfrac{q_1}{2A^k}
	}.
\end{equation}

Once $\xi^{k+1}$ is determined, the variables
$\mathbf v^{k+1}$, $\mathbf u^{k+1}$, $\widetilde{\psi}^{k+1}$, and
$r^{k+1}$ are recovered from the above decompositions. The scalar auxiliary-variable scheme preserves a discrete counterpart of the
Hamiltonian dissipation structure.
\begin{theorem}\label{thm:discrete_modified_hamiltonian_dissipation}
Assume that $A^k>0$ and that the external data entering the loading potential in
$A(\mathbf u,\psi)$ are time-independent; in particular, the prescribed
displacement on $\Gamma_{\delta,u}$ may be nonhomogeneous but is independent of
time, and the body force $\mathbf b$ and prescribed nonlocal traction
$\mathbf t$ are also time-independent.Then the auxiliary-variable update
	\eqref{eq:first_order_auxiliary_scheme}, before the irreversibility projection,
	satisfies
	\begin{equation}\label{eq:discrete_modified_hamiltonian_identity}
		\begin{aligned}
			&
		  \mathcal H^{k+1}
			-
			\mathcal H^k
			+
			\frac{\rho}{2}
			\|\mathbf v^{k+1}-\mathbf v^k\|_{L^2(\Omega\setminus\Gamma_\delta)}^2
			+
			\frac{G_c}{2}
			a_\delta(
			\widetilde{\psi}^{k+1}-\psi^k,
			\widetilde{\psi}^{k+1}-\psi^k
			)
			\\
			&\qquad
			+
			(r^{k+1}-r^k)^2
			+
			\Delta t\,\eta
			\|D_t\widetilde{\psi}^{k+1}\|_{L^2(\Omega)}^2
			=0,
		\end{aligned}
	\end{equation}
where the pre-projection modified Hamiltonian at the new time level is defined by
	\begin{equation}\label{eq:pre_projection_modified_hamiltonian}
		\mathcal H^{k+1}
		=
		\frac{\rho}{2}
		\|\mathbf v^{k+1}\|_{L^2(\Omega\setminus\Gamma_\delta)}^2
		+
		G_c\mathcal A_\delta(
		\widetilde{\psi}^{k+1},
		\mathcal L_\delta\widetilde{\psi}^{k+1}
		)
		+
		(r^{k+1})^2
		-
		C_0.
	\end{equation}
    Consequently,
	\begin{equation}
		\mathcal H^{k+1}
		\le
		\mathcal H^k .
	\end{equation}
\end{theorem}

\begin{proof}
	We use the elementary identity
	\begin{equation}\label{eq:identity_discrete_energy_auxiliary}
		(a-b,a)
		=
		\frac12
		\left(
		\|a\|^2-\|b\|^2+\|a-b\|^2
		\right),
	\end{equation}
	which holds in any inner-product space. Taking $\mathbf w=\mathbf v^{k+1}$ in the discrete momentum equation gives
	\begin{equation}\label{eq:proof_momentum_auxiliary}
		\begin{aligned}
			&
			\frac{\rho}{2\Delta t}
			\left(
			\|\mathbf v^{k+1}\|_{L^2(\Omega\setminus\Gamma_\delta)}^2
			-
			\|\mathbf v^k\|_{L^2(\Omega\setminus\Gamma_\delta)}^2
			+
			\|\mathbf v^{k+1}-\mathbf v^k\|_{L^2(\Omega\setminus\Gamma_\delta)}^2
			\right)
			\\
			&\qquad
			+
			\xi^{k+1}
			\left\langle
			\mathbf F_u^k,\mathbf v^{k+1}
			\right\rangle
			=0 .
		\end{aligned}
	\end{equation}
	
	Taking $\varphi=D_t\widetilde{\psi}^{k+1}$ in the discrete phase-field equation
	gives
	\begin{equation}\label{eq:proof_phase_auxiliary}
		\eta
		\|D_t\widetilde{\psi}^{k+1}\|_{L^2(\Omega)}^2
		+
		G_c\,a_\delta(
		\widetilde{\psi}^{k+1},
		D_t\widetilde{\psi}^{k+1}
		)
		+
		\xi^{k+1}
		\left\langle
		F_\psi^k,
		D_t\widetilde{\psi}^{k+1}
		\right\rangle
		=0 .
	\end{equation}
	Since $\mathcal A_\delta(\psi,\mathcal L_\delta\psi)
	=
	\frac12 a_\delta(\psi,\psi)$, we have
	\begin{equation}\label{eq:proof_crack_surface_auxiliary}
		\begin{aligned}
			G_c\,a_\delta(
			\widetilde{\psi}^{k+1},
			D_t\widetilde{\psi}^{k+1}
			)
			&=
			\frac{G_c}{\Delta t}
			\left[
			\mathcal A_\delta(
			\widetilde{\psi}^{k+1},
			\mathcal L_\delta\widetilde{\psi}^{k+1}
			)
			-
			\mathcal A_\delta(
			\psi^k,
			\mathcal L_\delta\psi^k
			)
			\right]
			\\
			&
			+
			\frac{G_c}{2\Delta t}
			a_\delta(
			\widetilde{\psi}^{k+1}-\psi^k,
			\widetilde{\psi}^{k+1}-\psi^k
			).
		\end{aligned}
	\end{equation}
Next, multiplying the auxiliary-variable equation by $2r^{k+1}$ gives
	\begin{equation}\label{eq:proof_r_auxiliary}
		\begin{aligned}
			&
			\frac{1}{\Delta t}
			\left[
			(r^{k+1})^2-(r^k)^2+(r^{k+1}-r^k)^2
			\right]
			\\
			&\qquad
			=
			\xi^{k+1}
			\left(
			\left\langle
			\mathbf F_u^k,D_t\mathbf u^{k+1}
			\right\rangle
			+
			\left\langle
			F_\psi^k,D_t\widetilde{\psi}^{k+1}
			\right\rangle
			\right).
		\end{aligned}
	\end{equation}
	Using $D_t\mathbf u^{k+1}=\mathbf v^{k+1}$, the coupling terms involving
	$\mathbf F_u^k$ and $F_\psi^k$ cancel exactly when
	\eqref{eq:proof_momentum_auxiliary}, \eqref{eq:proof_phase_auxiliary}, and
	\eqref{eq:proof_r_auxiliary} are combined. Substituting
	\eqref{eq:proof_crack_surface_auxiliary} and multiplying the resulting identity
	by $\Delta t$ gives \eqref{eq:discrete_modified_hamiltonian_identity}.
	Since all additional terms in \eqref{eq:discrete_modified_hamiltonian_identity}
	are nonnegative, it follows that
	$\mathcal H^{k+1}\le\mathcal H^k$.
	
\end{proof}

\begin{remark}
	The identity in Theorem~\ref{thm:discrete_modified_hamiltonian_dissipation}
	is established for the pre-projection variable
	$\widetilde{\psi}^{k+1}$. Therefore, the equality
	\eqref{eq:discrete_modified_hamiltonian_identity} should be understood as the
	structure-preserving property of the scalar auxiliary-variable update before the
	irreversibility projection.
\end{remark}

\subsection{Staggered alternating time discretization}\label{subsec:spa_temp}
\noindent We use a staggered alternating iteration at each time step $t_{k+1}$. Instead of a monolithic update, the displacement and phase-field subproblems are solved sequentially.

For the momentum equation, we adopt the Newmark-$\beta$ time discretization~\cite{belytschko2014nonlinear}. 
Given $(\mathbf u^{k},\mathbf v^{k},\mathbf a^{k})$ and the current phase-field iterate
$\psi^{k+1,m}$, the displacement subproblem is formulated as follows

Find $\mathbf u^{k+1,m+1}\in\mathcal V_{\omega,\mathbf u_g}(\Omega)$ such that, for all $\mathbf w\in\mathcal V_{\omega,\mathbf0}(\Omega)$,
\begin{equation}\label{eq:disp_var_problem}
	\rho\big(\mathbf{a}^{k+1,m+1},\mathbf{w}\big)_{L^2(\Omega\setminus\Gamma_\delta)}
	+
	\Big(
	\boldsymbol{\sigma}_{\delta}^{\mathrm d}
	(\mathbf{u}^{k+1,m+1},\psi^{k+1,m}),
	\boldsymbol{\varepsilon}_{\delta}(\mathbf{w})
	\Big)_{L^2(\Omega)}
	=
	\ell_u^{k+1}(\mathbf{w}),
\end{equation}
where $\ell_u^{k+1}(\mathbf{w})=(\mathbf{b}^{k+1},\mathbf{w})_{L^2(\Omega\setminus\Gamma_\delta)}
+(\mathbf{t}^{k+1},\mathbf{w})_{\Gamma_{\delta,t}}$ and the function space $\mathcal{V}_{\omega,\mathbf0}(\Omega)$ is defined by
\begin{equation}
	\mathcal{V}_{\omega,\mathbf0}(\Omega)=\{\mathbf{u}(\cdot,t)\in\mathcal{V}_{\omega}(\Omega):\mathbf{u}(\cdot,t) = \mathbf{0}\, \text{on}\,  \Gamma_{\delta,u}\}.
\end{equation}
The acceleration and velocity are updated by
\begin{equation}\label{eq:modelII_stagger_kinematics}
	\left\{
	\begin{aligned}
		\mathbf{u}^{k+1,m+1}
		&=
		\mathbf{u}^{k}
		+
		\Delta t\,\mathbf{v}^{k}
		+
		\Delta t^2
		\left[
		\left(\frac12-\beta\right)\mathbf{a}^{k}
		+
		\beta\mathbf{a}^{k+1,m+1}
		\right],\\
		\mathbf{v}^{k+1,m+1}
		&=
		\mathbf{v}^{k}
		+
		\Delta t
		\left[
		(1-\gamma)\mathbf{a}^{k}
		+
		\gamma\mathbf{a}^{k+1,m+1}
		\right], 
	\end{aligned}
	\right.
\end{equation}
and the Newmark parameters $(\gamma,\beta)$ are chosen such that
$\gamma\ge1/2$ and $\beta\ge\frac{1}{4}(\gamma+1/2)^2$.

After updating the displacement field, the driving state function is evaluated
pointwise as
\begin{equation}\label{eq:modelII_stagger_history}
	D^{k+1,m+1}(\mathbf{x})
	=
	\begin{cases}
		\displaystyle
		\frac{2\mathcal{W}_{\delta,e}^{+}
			(\mathbf{u}^{k+1,m+1}(\mathbf{x}))}{G_c/\delta},
		& \bar{\eta}>0,\\[1.0em]
		\displaystyle
		\max\!\left\{
		D^{k}(\mathbf{x}),
		\frac{2\mathcal{W}_{\delta,e}^{+}
			(\mathbf{u}^{k+1,m+1}(\mathbf{x}))}{G_c/\delta}
		\right\},
		& \bar{\eta}=0.
	\end{cases}
\end{equation}

The phase-field subproblem is then discretized by a backward Euler scheme:
find $\psi^{k+1,m+1}\in\mathcal{M}_\omega(\Omega)$ such that
\begin{equation}\label{eq:psi_var_problem}
	\begin{aligned}
		&
		\bar{\eta}
		\big(D_t\psi^{k+1,m+1},\varphi\big)_{L^2(\Omega)}
		+
		\big(\psi^{k+1,m+1},\varphi\big)_{L^2(\Omega)}
		+
		\frac{\delta^2}{2}
		\big[\psi^{k+1,m+1},\varphi\big]_{\omega_\delta(\Omega)}
		\\
		&
		+
		\big(D^{k+1,m+1}\psi^{k+1,m+1},\varphi\big)_{L^2(\Omega)}
		=
		\big(D^{k+1,m+1},\varphi\big)_{L^2(\Omega)},
		\quad
		\forall\varphi\in\mathcal{M}_\omega(\Omega),
	\end{aligned}
\end{equation}
where $	D_t\psi^{k+1,m+1}
=
\frac{\psi^{k+1,m+1}-\psi^k}{\Delta t}$.

We monitor the alternating iteration by
\begin{equation}\label{eq:modelII_stagger_stop}
R^{m+1}:=
\frac{\|\mathbf{u}^{k+1,m+1}-\mathbf{u}^{k+1,m}\|_{L^2(\Omega)}}{\|\mathbf{u}^{k+1,m+1}\|_{L^2(\Omega)}}
+\frac{\|\psi^{k+1,m+1}-\psi^{k+1,m}\|_{L^2(\Omega)}}{\|\psi^{k+1,m+1}\|_{L^2(\Omega)}},
\end{equation}
and stop when $R^{m+1}\le\varepsilon_{\mathrm{tol}}$. The full alternating procedure is summarized in Algorithm~\ref{alg:modelII_staggered}.

\begin{algorithm}[htbp]
\caption{Staggered alternating solve at time step $t_{k+1}$}
\label{alg:modelII_staggered}
\KwIn{$(\mathbf{u}^{k},\mathbf{v}^{k},\mathbf{a}^{k},\psi^{k},D^{k})$, $\Delta t$, Newmark parameters $(\beta,\gamma)$, tolerance $\varepsilon_{\mathrm{tol}}$}
\KwOut{$(\mathbf{u}^{k+1},\mathbf{v}^{k+1},\mathbf{a}^{k+1},\psi^{k+1},D^{k+1})$}
Initialize $\mathbf{u}^{k+1,0}=\mathbf{u}^{k}$, $\mathbf{v}^{k+1,0}=\mathbf{v}^{k}$, $\psi^{k+1,0}=\psi^{k}$, $D^{k+1,0}=D^{k}$, $m=0$\;
Set $R^{m+1}=+\infty$\;
\While{$R^{m+1}>\varepsilon_{\mathrm{tol}}$}{
Solve the displacement variational subproblem \eqref{eq:disp_var_problem} with fixed $\psi^{k+1,m}$  by Newton-Raphson method\;
Update the driving state function by \eqref{eq:modelII_stagger_history}\;
Solve the phase-field variational subproblem \eqref{eq:psi_var_problem}\;
Evaluate $R^{m+1}$ using \eqref{eq:modelII_stagger_stop}\;
$m\leftarrow m+1$\;
}
Set converged iterate as $(\mathbf{u}^{k+1},\mathbf{v}^{k+1},\mathbf{a}^{k+1},\psi^{k+1},D^{k+1})$\;
\end{algorithm}

\subsection{Finite element method for spatial discretization}
\label{subsec:spa_dis}
\noindent We partition $\Omega$ into a shape-regular, quasi-uniform quadrilateral mesh 
$\mathcal{T}_h$. The mesh is assumed to resolve the nonlocal boundary layer 
$\Gamma_\delta$. We denote by $h$ the maximum diameter of all elements 
$T\in\mathcal{T}_h$. The element subsets associated with the interior domain and 
the nonlocal boundary layers are defined as
\begin{equation}
	\begin{aligned}
		\mathcal{T}_{h,int}
		&=\{T\in\mathcal{T}_h:\, T\subset \Omega\setminus\Gamma_\delta\},\\
		\mathcal{T}_{h,t}
		&=\{T\in\mathcal{T}_h:\, T\subset  \Gamma_{\delta,t}\},\\
		\mathcal{T}_{h,u}
		&=\{T\in\mathcal{T}_h:\, T\subset \Gamma_{\delta,u}\}.
	\end{aligned}
\end{equation}
Further, we introduce 
\begin{equation}
	\begin{aligned}
		S_h&=\left\{\mathbf{u}_h\in C^0(\overline{\Omega};\mathbb{R}^d)\, \big|\,
		\mathbf{u}_h|_T\in[Q_1(T)]^d,\ \forall T\in \mathcal{T}_h\ \text{and}\
		\mathbf{u}_h|_{T'}=\mathbf{u}_{g,h},\ \forall T'\in\mathcal{T}_{h,u}\right\},\\
		S_h^0&=\left\{\mathbf{u}_h\in C^0(\overline{\Omega};\mathbb{R}^d)\, \big|\,
		\mathbf{u}_h|_T\in[Q_1(T)]^d,\ \forall T\in \mathcal{T}_h\ \text{and}\
		\mathbf{u}_h|_{T'}=\mathbf{0},\ \forall T'\in\mathcal{T}_{h,u}\right\},\\
		\widetilde{S}_h&=\left\{w_h\in C^0(\overline{\Omega};\mathbb{R})\, \big|\,
		w_h|_T\in Q_1(T),\ \forall T\in \mathcal{T}_h\right\},
	\end{aligned}
\end{equation}
where $\mathbf{u}_{g,h}$ is the finite element interpolation of the prescribed boundary displacement $\mathbf{u}_{g}$ on $\Gamma_{\delta,u}^h$.

The fully discrete Galerkin scheme is obtained from \eqref{eq:disp_var_problem} and \eqref{eq:psi_var_problem} by replacing $\mathcal V_{\omega,\mathbf u_g}(\Omega), \mathcal V_{\omega,\mathbf 0}(\Omega), \mathcal M_\omega(\Omega) $ with $ S_h, S_h^0, \widetilde S_h$, respectively. The discrete driving state $D_h^{k+1,m+1}$ is evaluated from \eqref{eq:modelII_stagger_history} using the finite element displacement $\mathbf u_h^{k+1,m+1}$. The resulting nonlinear displacement system is solved by Newton iterations, while the phase-field subproblem leads to a linear system at each staggered iteration. The finite element implementation of the nonlocal operators $\mathcal L_\delta$, $\mathcal G_\delta$, and $\mathcal D_\delta$ follows the quadrature-based treatment of nonlocal interactions described in \cite{d2021cookbook}.

A key additional ingredient in spatial discretization of nonlocal model is the mesh-dependent approximation of the interaction domain $B_\delta(\mathbf{x})$. At the continuous level, the nonlocal operators involve integrals over the interaction ball
$B_\delta(\mathbf{x})$. On a mesh $\mathcal{T}_h$, this ball is not represented exactly and must be approximated by a
discrete interaction neighborhood assembled from mesh entities.
For a given quadrature point $\mathbf{x}_{T,q}\in T\in \mathcal{T}_h$,
we define its discrete nonlocal interaction domain set by
\begin{equation}
	\label{eq:discrete_ball}
	B_{\delta,h}(\mathbf{x}_{T,q})=\big\{\, T'\in\mathcal{T}_h \; \big| \;
	|\mathbf{x}_{T',c}-\mathbf{x}_{T,c}|<\delta \,\big\},
\end{equation}
where $\mathbf{x}_{T,c}$ is the center of element $T\in\mathcal{T}_h$. We approximate the continuous interaction integral by a quadrature-based sum over
$B_{\delta,h}(\mathbf{x}_{T,q})$. A schematic illustration of the approximation of $B_\delta(\mathbf{x})$ on a finite element mesh is
shown in Figure~\ref{fig:ball_approx}.

\begin{figure}[htbp]
	\centering
	\begin{tikzpicture}
\node[anchor=south west, inner sep=0] (img) 
    at (0,0) {\includegraphics[width=0.6\textwidth]{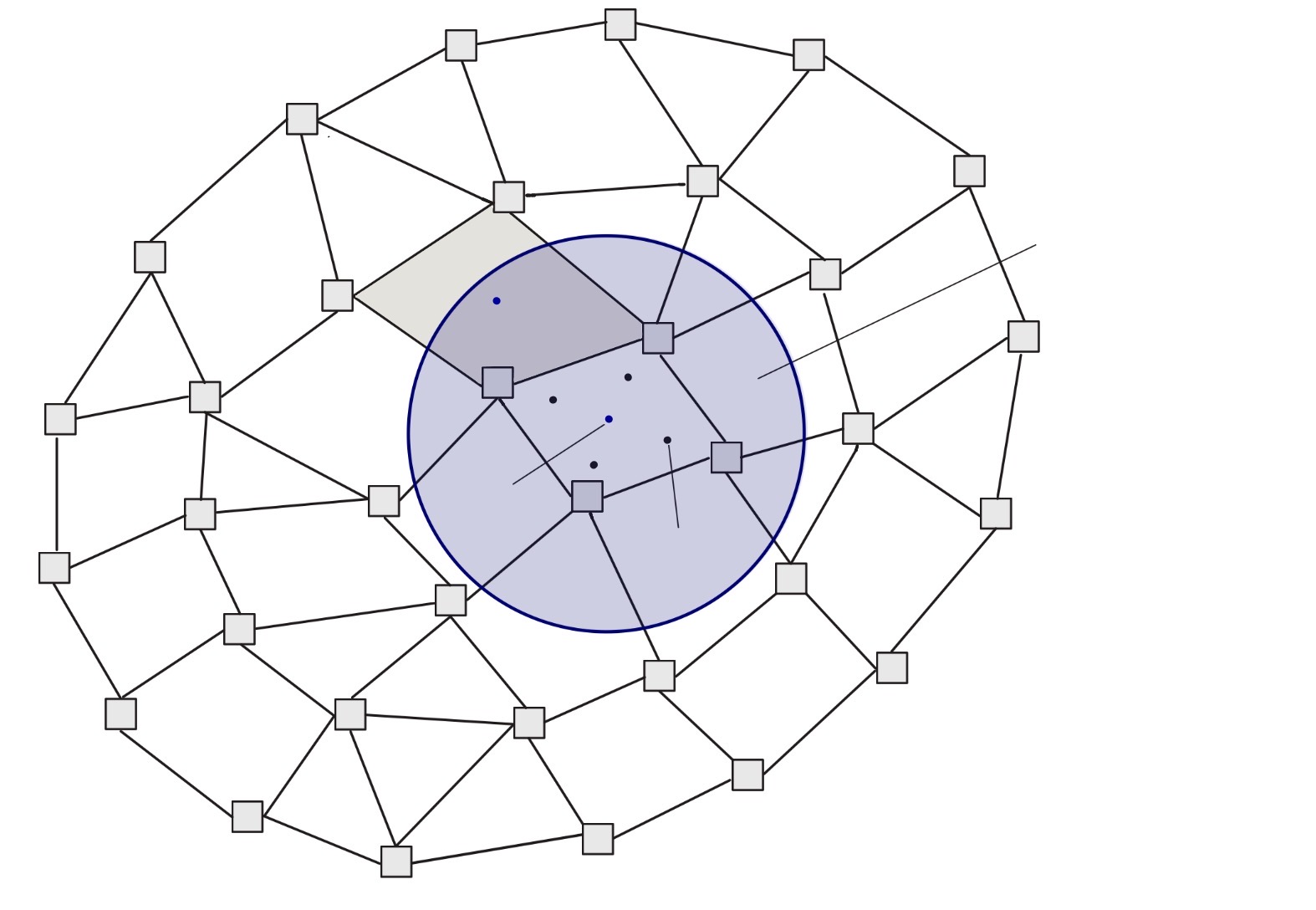}};
\begin{scope}[x={(img.south east)}, y={(img.north west)}]
\node at (0.42,0.63) {$\mathbf{x}_{T',c}$};
\node at (0.89,0.74) {$B_{\delta,h}(\mathbf{x}_{T,q})$};
\node at (0.52,0.38) {$\mathbf{x}_{T,q}$};
\node at (0.37,0.45) {$\mathbf{x}_{T,c}$};
\end{scope}
\end{tikzpicture}
\caption{ Schematic illustration of the mesh-dependent approximation of the interaction domain. For a quadrature point $\mathbf{x}_{T,q}$ in element $T$, the continuous ball $B_\delta(\mathbf{x}_{T,q})$ is approximated by the discrete neighborhood $B_{\delta,h}(\mathbf{x}_{T,q})$, which consists of elements $T'\in\mathcal{T}_h$ selected according to the distance between element centers. The corresponding nonlocal integral is then evaluated by applying quadrature over these selected interaction elements. }
	\label{fig:ball_approx}
\end{figure}
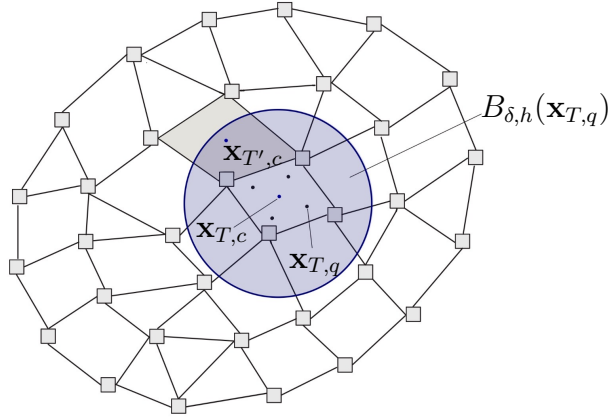

\section{Numerical examples}\label{sec:numerical_eg}
\noindent This section presents four numerical examples to assess the proposed nonlocal phase-field framework for dynamic fracture. The first example considers Mode-I fracture in a single-edge notched plate. The second example examines dynamic crack branching in a pre-notched plate, a widely used benchmark for assessing crack-path instability and branching behavior. The third example studies crack evolution under dynamic shear loading. The fourth example investigates fragmentation under impulsive internal pressure. All four examples are considered in two-dimensional settings. Example 1 is used to examine the structure-preserving SAV scheme and Examples 2, 3, and 4 are solved using the staggered alternating scheme. In all simulations, the displacement subproblem employs the Gaussian-type kernel defined in Eq.~\eqref{eq:kernel} with $\alpha=(d+1)^2$, while the phase-field subproblem uses the kernel $2\omega_\delta$. For the staggered alternating simulations, the displacement subproblem is integrated in time by the Newmark--$\beta$ method with $\gamma=0.5$ and $\beta=0.25$, corresponding to the classical average-acceleration scheme. The tolerance $\varepsilon_{\mathrm{tol}}$ for the staggered alternating iteration is reported in the parameter table of each relevant example and is chosen to balance accuracy and computational efficiency.

\subsection{Mode-I fracture of a single-edge notched plate}\label{eg: mode-I}
\noindent Consider a pre-notched rectangular plate in $\mathbb{R}^{2}$ subjected to vertical displacement applied to the top and bottom nonlocal boundary layers, as illustrated in Figure~\ref{fig:mode_I_fracture}. The plate occupies the domain $\Omega=[0,0.04]\text{m}\times[0,0.04]\text{m}$. An initial notch is introduced from the left edge along the midline of the plate. The plate is initially at rest, and the prescribed displacement first increases linearly and then remains constant according to
\begin{equation}
\begin{aligned}
u_{g,2}(x_1,x_2,t)&=
\begin{cases}
\bar{u}(t), & x_1\in[0,0.04],\ x_2\in[0.04-\delta,0.04],\\[0.4em]
-\bar{u}(t), & x_1\in[0,0.04],\ x_2\in[0,\delta],
\end{cases}\\
\bar{u}(t)&=
\begin{cases}
\dfrac{\Delta u}{t_0}\,t, & 0\le t\le t_0,\\[0.4em]
\Delta u, & t>t_0,
\end{cases}
\label{eq:mode_I_prescribed_displacement}
\end{aligned}
\end{equation}
where $t_0 = 20\,\mu\mathrm{s}$ and $\Delta u = 0.0015\,\mathrm{mm}$. The physical and numerical parameters used in the simulations are given in Table~\ref{tab:material_parameters_1}, and the plane-stress condition is assumed in this numerical example.

\begin{figure}[htbp]
	\centering
	\includegraphics[width=0.75\textwidth]{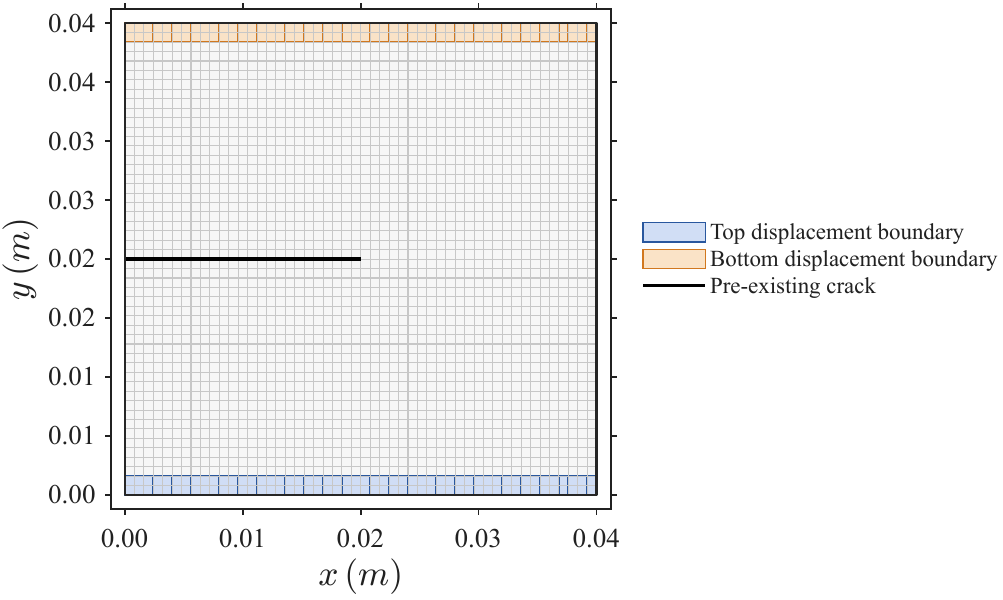}
	\caption{Schematic of the Mode-I fracture test. The mesh shown is for illustration only.}
	\label{fig:mode_I_fracture}
\end{figure}

Figure~\ref{fig:mode_I_coarse_mesh} shows the distributions of the phase-field variable $\psi(\mathbf{x},t)$ at two representative time steps. The results indicate that a single crack initiates from the notch tip and subsequently propagates approximately straight toward the opposite boundary. The corresponding distributions of the in-plane nonlocal hydrostatic stress,
$\sigma_{n,\mathrm{hyd}}=\frac{1}{2}\mathrm{tr}\bm{\sigma}_{n}$,
are presented in Figure~\ref{fig:mode_I_stress}. The temporal evolutions of the kinetic energy, nonlocal degraded elastic energy, nonlocal dissipated fracture energy, and modified Hamiltonian are shown in Figure~\ref{fig:energy_evolution_mode_I}. Before crack initiation, which occurs at approximately $t=20\,\mu\mathrm{s}$, the energy curves obtained on different meshes remain in close agreement, indicating that the initial elastic response is consistently captured. After crack initiation, visible differences appear among the three meshes. This behavior is not unexpected, since the post-initiation stage involves strong localization, rapid stress release, wave propagation, and crack-path evolution, all of which are sensitive to the spatial resolution of the diffusive crack zone and the nonlocal interaction neighborhood. Nevertheless, the qualitative energy-transfer mechanism remains consistent across the meshes: the degraded elastic energy decreases as the dissipated fracture energy increases, while the kinetic energy exhibits oscillations associated with dynamic wave effects. The modified Hamiltonian is plotted as a diagnostic quantity for the SAV update; in the presence of prescribed displacement loading and the subsequent irreversibility projection, it should not be interpreted as a strictly monotone quantity over the entire loading process.

\begin{table}[htbp]
	\centering
	\begin{threeparttable}
		\caption{Physical and numerical parameters used in the simulations of Mode-I fracture of a single-edge notched plate.}
		\label{tab:material_parameters_1}
		\begin{tabular}{l c c c}
			\toprule
			Physical parameters & Symbol & Unit & Value \\
			\midrule
			Young's modulus              & $E$      & GPa        & 32 \\   
			Poisson's ratio              & $\nu$    & --         & 0.2 \\ 
			Density                      & $\rho$   & kg/m$^{3}$ & $2.45\times10^{3}$ \\ 
			Critical energy release rate & $G_c$    & J/m$^2$    & 2 \\ 
			Viscous resistance of crack  & $\bar{\eta}$ & s      & $1.00\times10^{-6}$ \\
			\midrule
			Numerical parameters &  &  &  \\
			\midrule
		    Mesh resolution              & $n_x=n_y$ & -- & $100$, $200$, $400$ \\
            Time step                    & $\Delta t$ & $\mu$s & $0.15$, $0.075$, $0.0375$ \\
			Nonlocal interaction length scale & $\delta$ & mm & $2.37h$ \\
			SAV energy shift             & $C_0$      & J      & $8.00\times10^{-2}$ \\
			Alternating iteration tolerance & $\varepsilon_{\mathrm{tol}}$ & -- & $1.00\times10^{-3}$ \\
			\bottomrule
		\end{tabular}
		\begin{tablenotes}[flushleft]
			\footnotesize
			\item Note:
			$C_0=\max\{E_{\mathrm{el}}^{\ast},E_{\mathrm{frac}}^{\ast}\}$, where
			$E_{\mathrm{el}}^{\ast}=0.5\,E(\Delta u/L_y)^2|\Omega|$ and
			$E_{\mathrm{frac}}^{\ast}=G_cL_x$ under the unit-thickness assumption.
			Here $\Delta u=1.5\times10^{-6}\,\mathrm{m}$.
		\end{tablenotes}
	\end{threeparttable}
\end{table}

\begin{figure}[htbp]
	\makebox[\textwidth][c]{  
\begin{minipage}{1.2\textwidth}
\centering
	\subfigure[]{
		\includegraphics[width=0.475\textwidth]{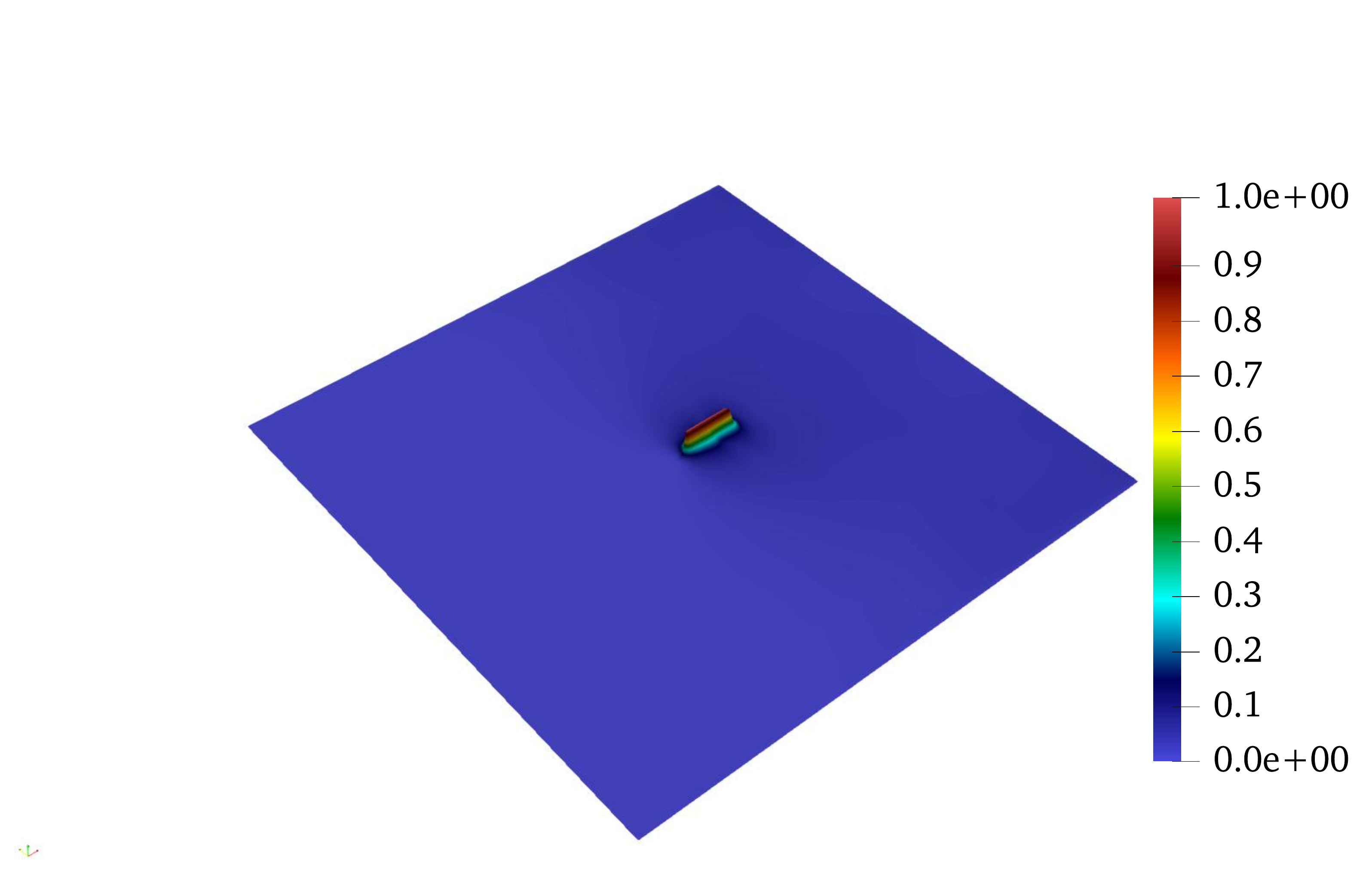}
		\label{fig:mode_I_coarse_250}
    }
		\subfigure[]{
		\includegraphics[width=0.475\textwidth]{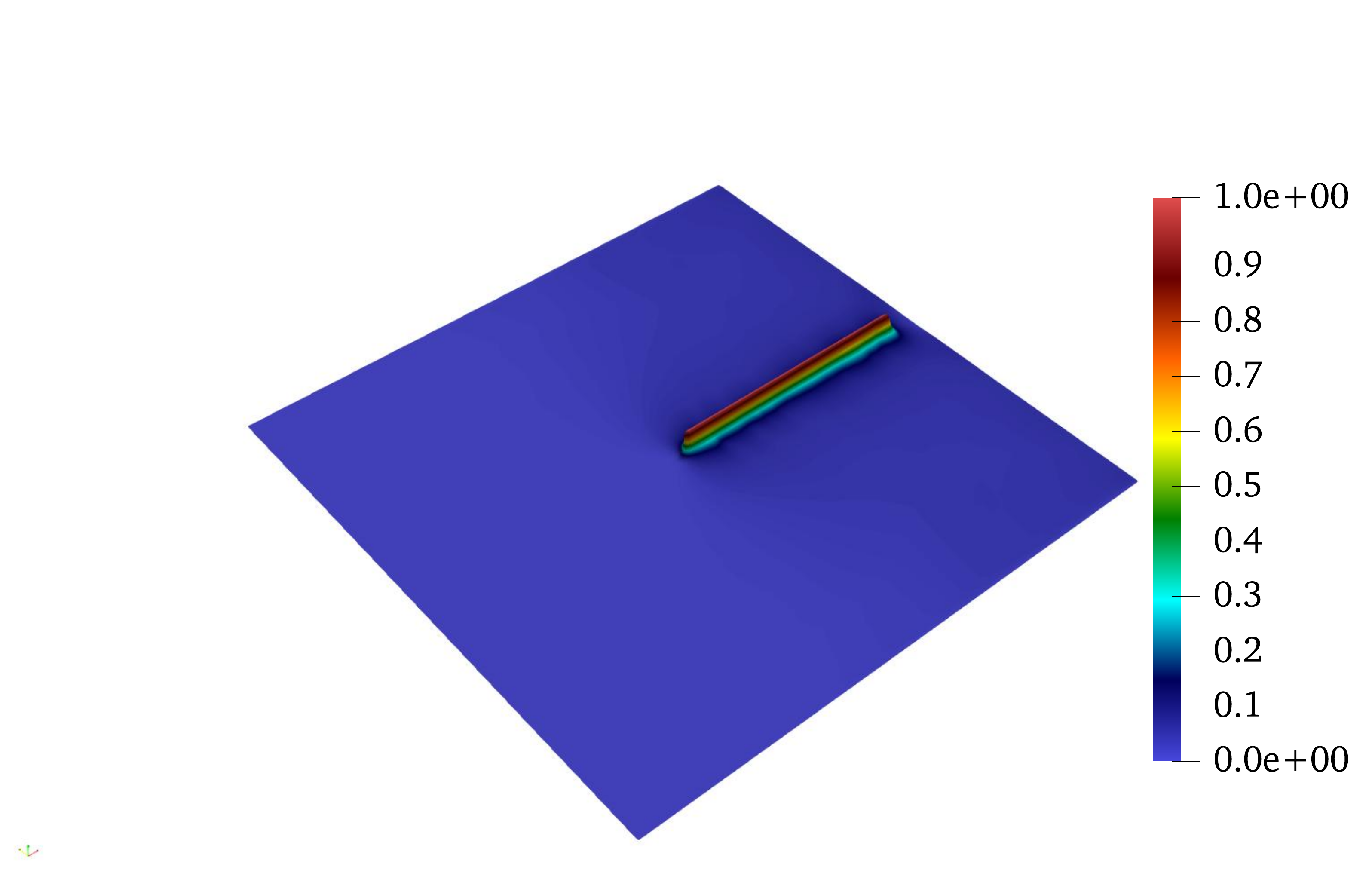}
		\label{fig:mode_I_coarse_1000}
	}
	\end{minipage}
}
\caption{Snapshots of the phase-field $\psi(\mathbf{x},t)$ distribution obtained using the SAV scheme at two representative times: (a) $t=37.5\,\mu\mathrm{s}$ and (b) $t=87.2\,\mu\mathrm{s}$.}
	\label{fig:mode_I_coarse_mesh}
\end{figure}

\begin{figure}[htbp]
		\makebox[\textwidth][c]{  
		\begin{minipage}{1.2\textwidth}
\centering
	\subfigure[]{
		\includegraphics[width=0.475\textwidth]{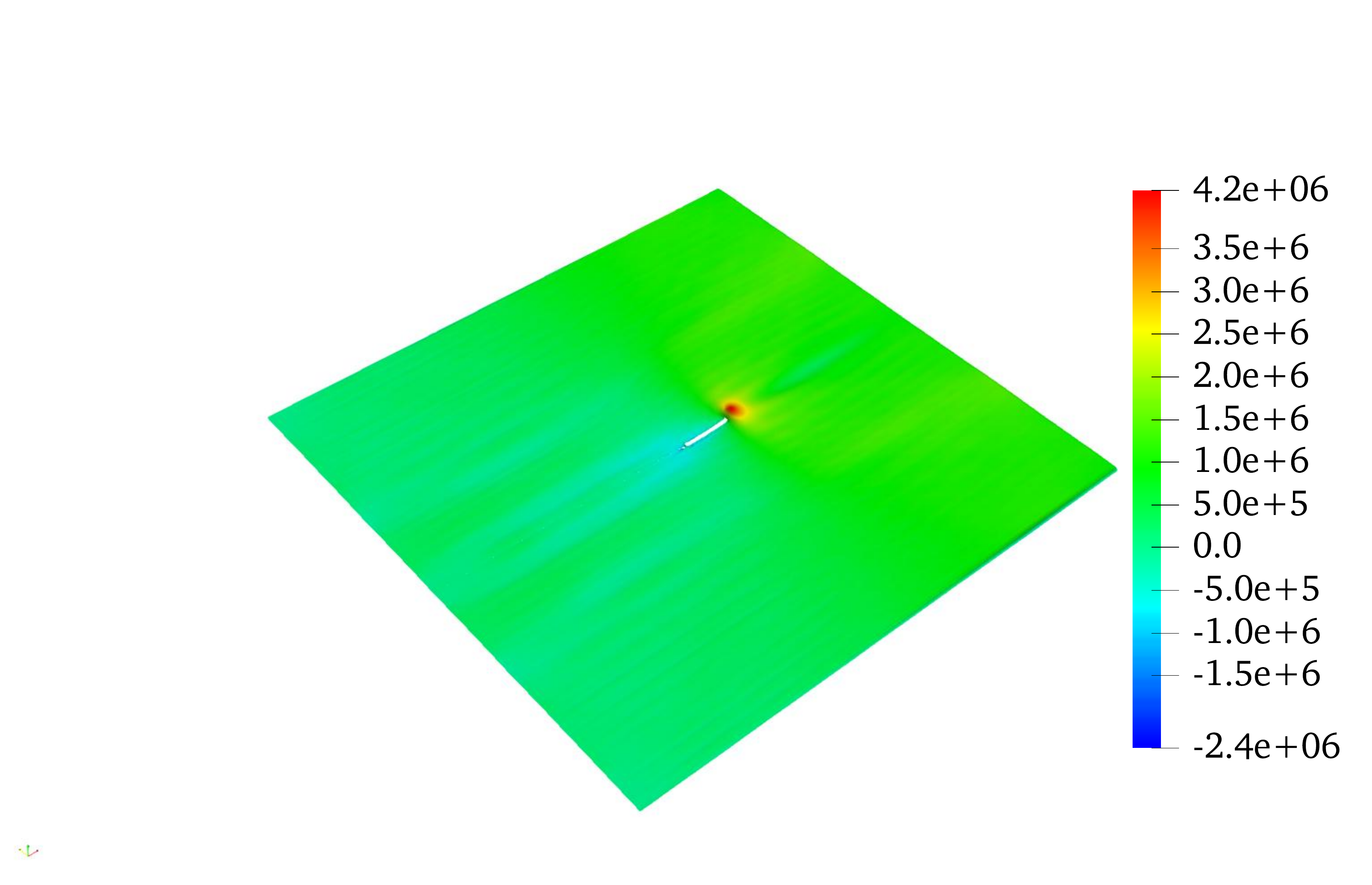}
		\label{fig:mode_I_coarse_stress_250}
    }
		\subfigure[]{
		\includegraphics[width=0.475\textwidth]{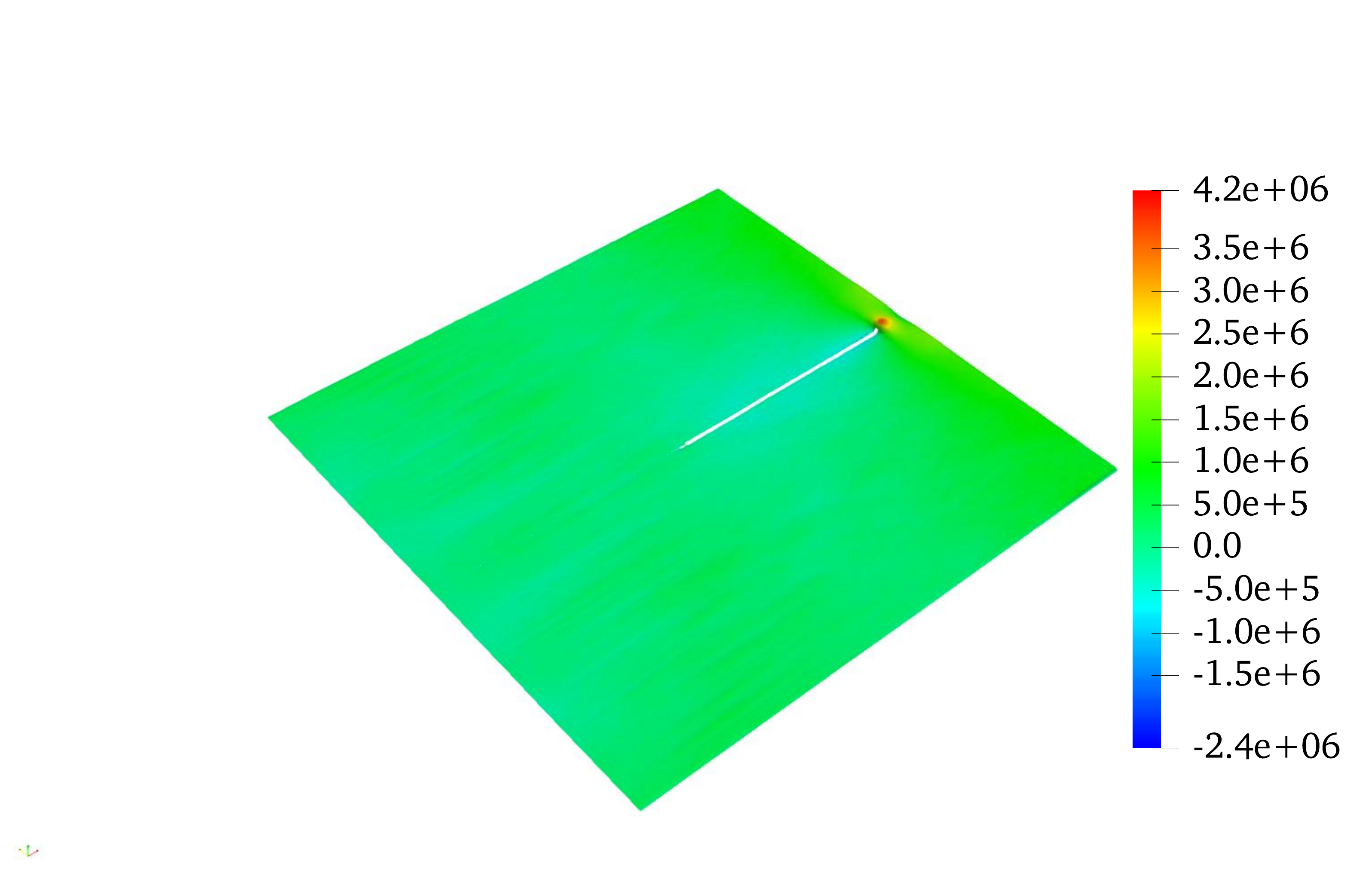}
		\label{fig:mode_I_coarse_stress_1000}
	}
		\end{minipage}
}
\caption{Snapshots of the nonlocal hydrostatic stress $\sigma_{n,\mathrm{hyd}}(\mathbf{x},t)$ distribution obtained using the SAV scheme at two representative times: (a) $t=37.5\,\mu\mathrm{s}$ and (b) $t=87.2\,\mu\mathrm{s}$. The displacement field is magnified by a factor of $500$, and regions with $\psi>0.95$ are omitted for clarity.}
	\label{fig:mode_I_stress}
\end{figure}

\begin{figure}[htbp]
	\centering
		\subfigure[]{
		\includegraphics[width=0.475\textwidth]{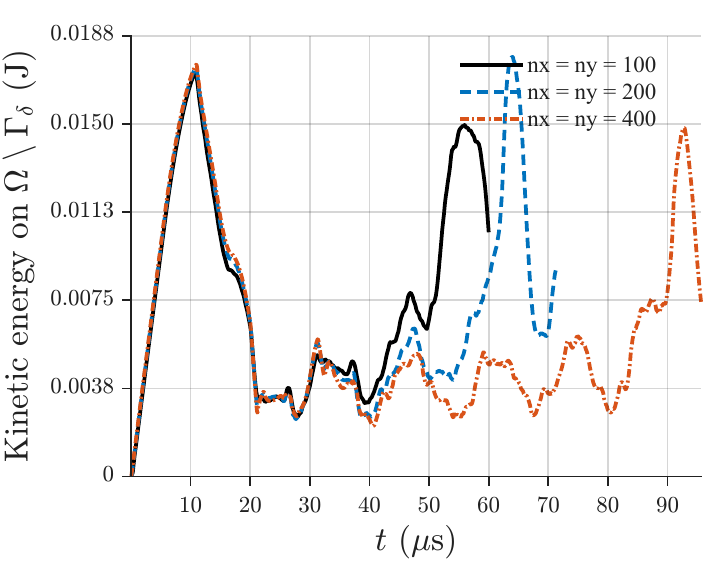}
		\label{fig:kinetic_energy_mode_I}
	}
	\subfigure[]{
		\includegraphics[width=0.475\textwidth]{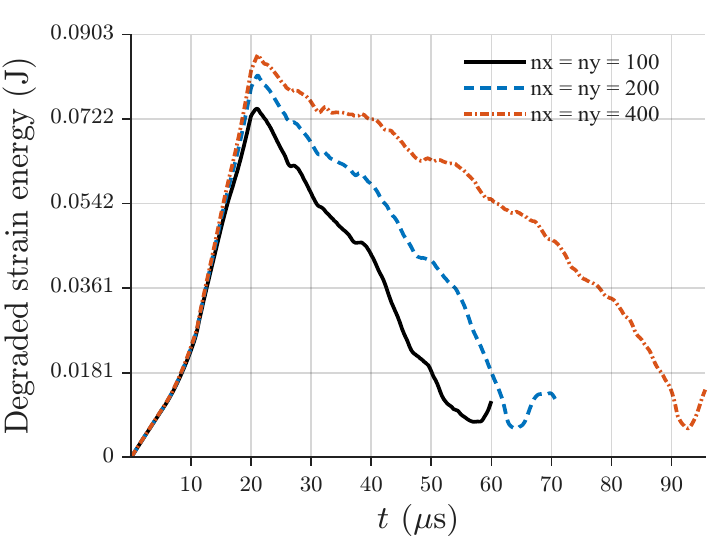}
		\label{fig:energy_elastic_mode_I}
	}
	\subfigure[]{
		\includegraphics[width=0.475\textwidth]{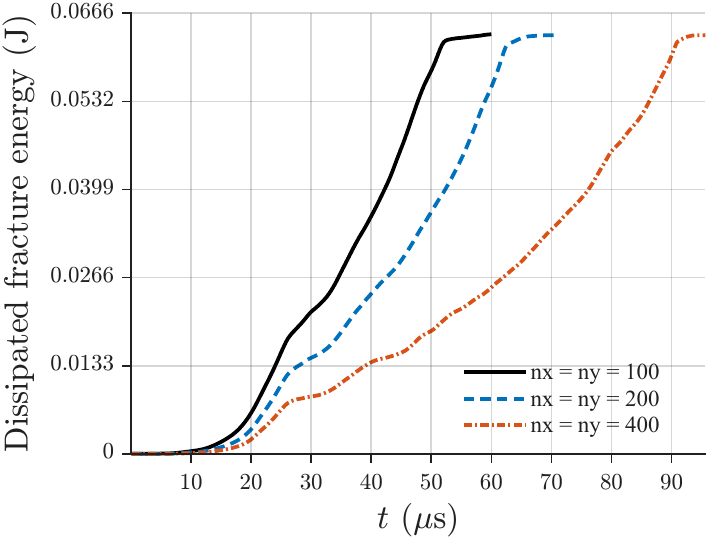}
		\label{fig:energy_fracture_mode_I}
	}
	\subfigure[]{
	\includegraphics[width=0.475\textwidth]{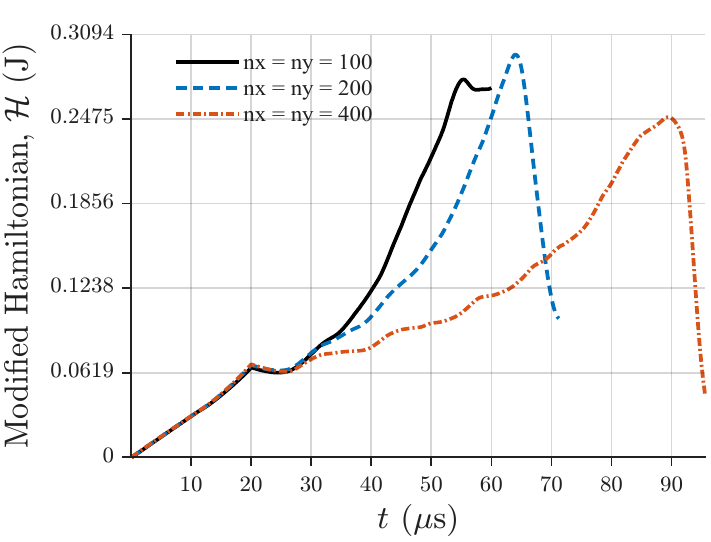}
	\label{fig:total_energy_mode_I}
}
\caption{Time evolution of the energy components in the SAV simulation:
	(a) kinetic energy $\frac{\rho}{2}\|\dot{\mathbf u}\|_{L^2(\Omega}^2$,
	(b) nonlocal degraded elastic energy $\int_{\Omega}\mathcal W_{\delta,de}(\mathbf u,\psi)\,\mathrm d v$,
	(c) nonlocal dissipated fracture energy $G_c\mathcal A_\delta(\psi,\mathcal L_\delta\psi)$, and
	(d) modified Hamiltonian
	$\mathcal H=\frac{\rho}{2}\|\dot{\mathbf u}\|_{L^2(\Omega\setminus\Gamma_\delta)}^2
	+G_c\mathcal A_\delta(\psi,\mathcal L_\delta\psi)+r^2-C_0$.}
	\label{fig:energy_evolution_mode_I}
\end{figure}

\subsection{Dynamic crack branching}
\noindent Consider a pre-notched rectangular plate in $\mathbb{R}^{2}$ subjected to dynamic tensile loading, as illustrated in Figure~\ref{fig:dynamic_branching}. The geometry and boundary conditions follow the standard dynamic crack branching benchmark configuration\cite{xu1994numerical,belytschko2003dynamic,song2008comparative}. The plate occupies the domain $\Omega=[0,0.1]\text{m}\times[0,0.04]\text{m}$ and a traction $\bm{\sigma}=1\text{MPa}$ is applied to the top and bottom boundaries at the initial time and then maintained constant, while all remaining boundaries are traction-free. In the numerical simulation, the traction $\bm{\sigma}$ on the local top and bottom boundaries is transformed into the nonlocal traction using the following local-nonlocal traction equivalence condition given by 
\begin{equation}\label{eq:condition}
	\int_{\Gamma_t}\mathbf{t}_{local}\ \mathrm{d}a=\int_{\Gamma_t}\bm{\sigma}\hat{\mathbf{n}}\ \mathrm{d}a= \int_{\Gamma_{\delta,t}}\mathbf{t}\ \mathrm{d}v,
\end{equation}
where $\Gamma_t\subset\partial\Omega$ is the local traction boundary, $\mathbf{t}_{local}$ is the local traction and $\bm{\sigma}$ denotes the Cauchy stress.

\begin{figure}[htbp]
	\centering
	\includegraphics[width=0.8\textwidth]{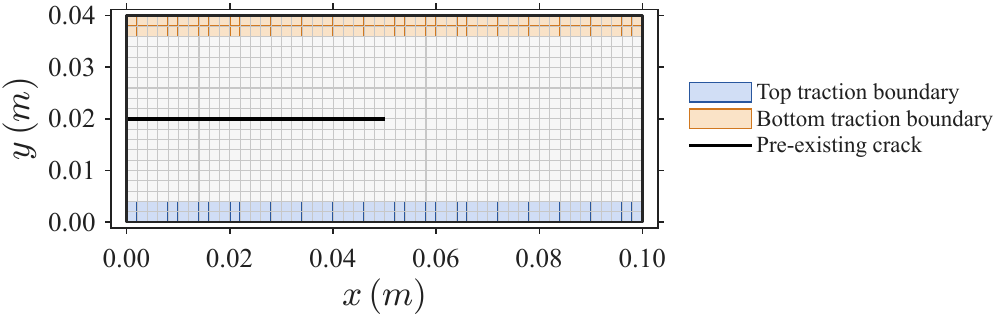}
	\caption{Schematic of the dynamic branching test. The mesh shown is for illustration only.}
	\label{fig:dynamic_branching}
\end{figure}

The applied loading is sufficiently large to trigger dynamic crack branching during the evolution process. The initial crack is explicitly prescribed as a geometrical discontinuity in the domain, representing a pre-existing sharp crack from which subsequent crack propagation and branching evolve. The material parameters used in this example are summarized in Table~\ref{tab:material_parameters_2} and the plane-stress condition are assumed in this numerical example.

\begin{table}[htbp]
	\centering
	\begin{threeparttable}
		\caption{Physical and numerical parameters used in the simulations of dynamic crack branching.}
		\label{tab:material_parameters_2}
		\begin{tabular}{l c c c}
			\toprule
			Physical parameters & Symbol & Unit & Value \\
			\midrule
			Young's modulus              & $E$      & GPa        & 32 \\   
			Poisson's ratio              & $\nu$    & --         & 0.2 \\ 
			Density                      & $\rho$  & kg/m$^{3}$ & $2.45\times10^{3}$ \\ 
			Critical energy release rate & $G_c$   & J/m$^2$    & 3 \\ 
			Viscous resistance of crack  & $\bar{\eta}$ & s & $0.0$ \\
			Rayleigh wave speed          & $c_R$   & m/s        & $2.13\times10^{3}$ \\ 
			\midrule
			Numerical parameters &  &  &  \\
			\midrule  
			Mesh size & $h$ & mm & $0.167$, $0.1$ \\
			Time step & $\Delta t$ & s & $0.75\,h\times10^{-3}/c_d$ \\  
			Nonlocal interaction length scale & $\delta$ & mm & $2h$ \\  
			Alternating iteration tolerance & $\varepsilon_{\mathrm{tol}}$ & -- & $1.00\times10^{-6}$ \\
			\bottomrule
		\end{tabular}
	\end{threeparttable}
\end{table}

Figure~\ref{fig:dy_energy} shows the temporal evolutions of kinetic energy, nonlocal degraded strain energy, nonlocal dissipated fracture energy, and total mechanical-fracture energy. In the early loading stage, the degraded strain energy gradually accumulates, while the kinetic energy remains relatively small. Once crack propagation starts, the stored elastic energy is released and converted into kinetic energy and dissipated fracture energy. As a result, the kinetic energy increases rapidly at later times, and the dissipated fracture energy grows monotonically, reflecting the irreversible nature of crack evolution.

The coarse-mesh ($h=0.167\,\mathrm{mm}$) and fine-mesh ($h=0.1\,\mathrm{mm}$) results exhibit consistent overall trends. Minor deviations appear mainly in the degraded strain energy and dissipated fracture energy at the later stage of crack propagation. Nevertheless, the total mechanical-fracture energy evolution is well reproduced by both meshes, indicating that the proposed formulation captures the main energy-conversion mechanism in a mesh-consistent manner.

Figure~\ref{fig:dy_crack_velocity} shows the corresponding crack-tip velocity histories. 
The crack-tip position is sampled every $0.25~\mu\mathrm{s}$, and the velocity is computed using a local three-point linear fitting procedure~\cite{wu2020phase}. After crack initiation, the velocity exhibits pronounced fluctuations, mainly due to the discrete mesh-based crack-tip tracking. Nevertheless, the predicted crack-tip velocities remain below $0.6c_R$ in all cases, where $c_R$ denotes the Rayleigh wave speed. This result is consistent with commonly reported observations in brittle dynamic fracture, where crack propagation speeds are typically bounded well below the Rayleigh wave speed.

\begin{figure}[htbp]
	\centering
	\subfigure[]{
		\includegraphics[width=0.475\textwidth]{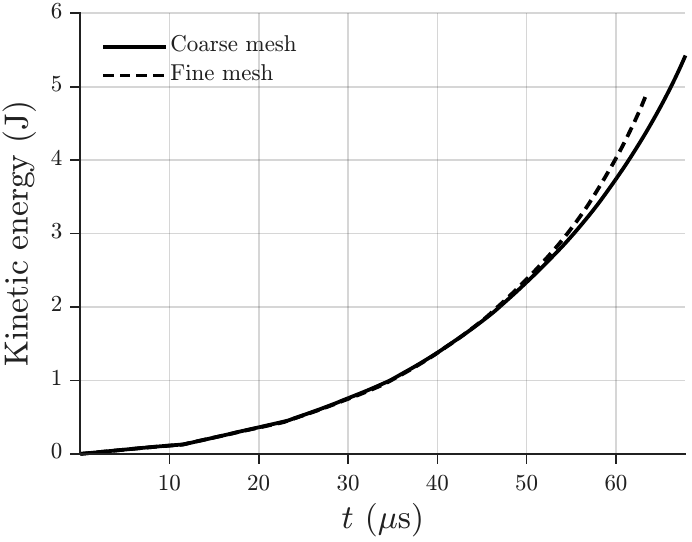}
		\label{fig:dy_kinetic_energy}
	}
	\subfigure[]{
		\includegraphics[width=0.475\textwidth]{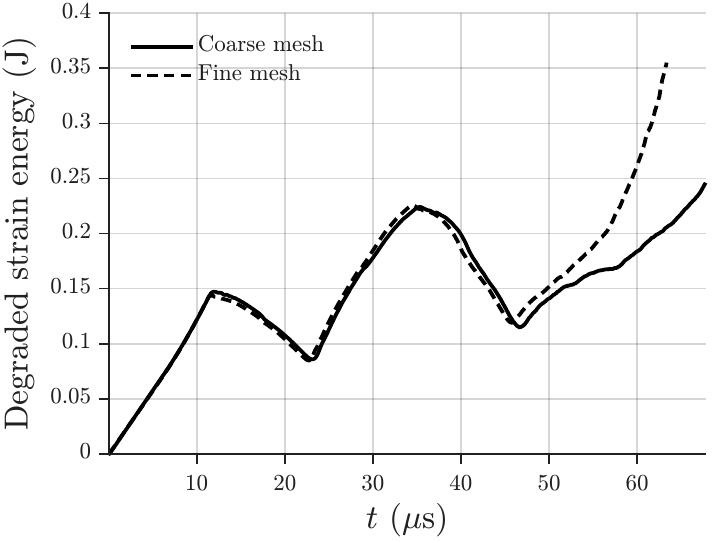}
		\label{fig:dy_degraded_strain_energy}
	}
	\subfigure[]{
		\includegraphics[width=0.475\textwidth]{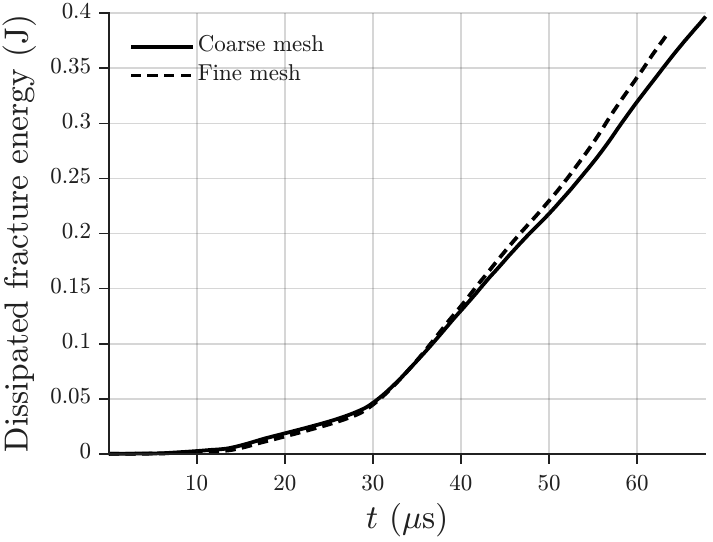}
		\label{fig:dy_fracture_energy}
	}
	\subfigure[]{
	\includegraphics[width=0.475\textwidth]{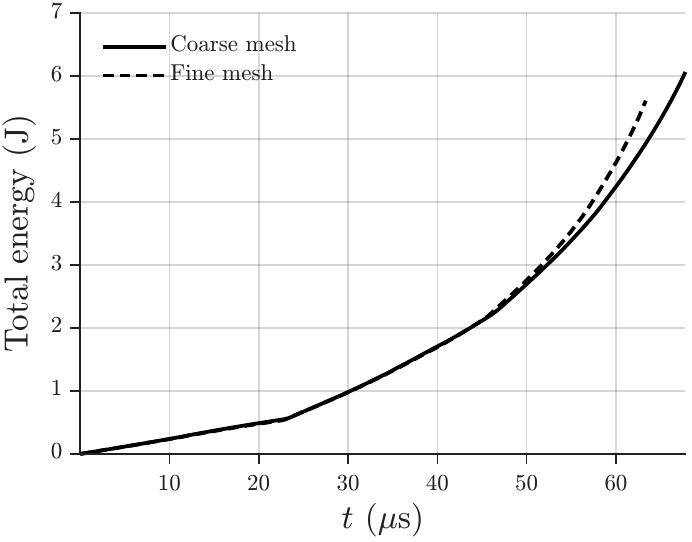}
	\label{fig:total_energy}		
	}
	\caption{Time evolution of the energy components during dynamic crack branching: (a) kinetic energy $\int_{\Omega}\frac{\rho}{2}|\dot{\mathbf u}|^2\,\mathrm dv$, (b) nonlocal degraded strain energy $\mathcal{W}_{\delta,de}(\mathbf{u},\psi)$, (c) nonlocal dissipated fracture energy $G_c\mathcal{A}_\delta(\psi,\mathcal{L}_\delta\psi)$, and (d) total mechanical-fracture energy $\int_{\Omega}\frac{\rho}{2}|\dot{\mathbf u}|^2\,\mathrm dv+\mathcal{W}_{\delta,de}(\mathbf{u},\psi)+G_c\mathcal{A}_\delta(\psi,\mathcal{L}_\delta\psi)$.}
	\label{fig:dy_energy}
\end{figure}

\begin{figure}[htbp]
	\centering
		\includegraphics[width=0.6\textwidth]{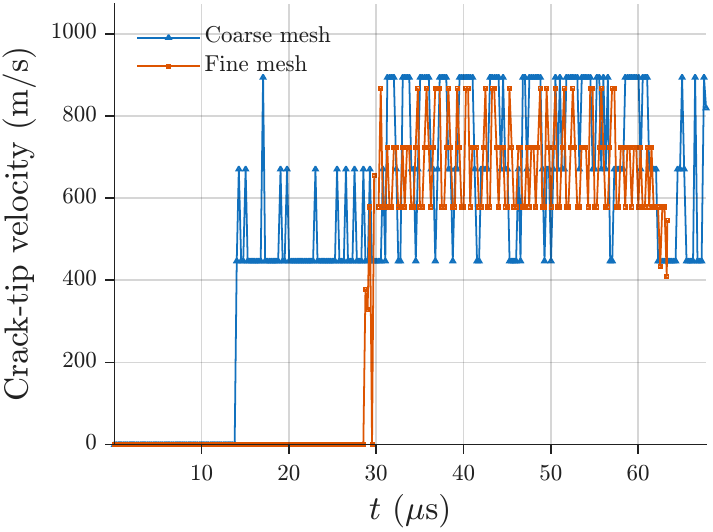}
	\caption{Time evolution of the crack-tip velocity for the coarse and fine meshes in dynamic crack branching. After crack branching, the reported crack-tip velocity is evaluated along the upper branch.}
	\label{fig:dy_crack_velocity}
\end{figure}

The distributions of the phase-field variable $\psi(\mathbf{x},t)$ and the nonlocal hydrostatic stress $\sigma_{n,\mathrm{hyd}}(\mathbf{x},t)$ at selected time steps are shown in 
Figures~\ref{fig:dy_phase_field_coarse_fine_mesh} and \ref{fig:dy_plate_stress_coarse_fine_mesh}, respectively. 
For both mesh resolutions, the crack initiates from the pre-existing notch tip and subsequently develops into a branched crack pattern. The coarse mesh ($h=0.167\,\mathrm{mm}$) and the fine mesh ($h=0.1\,\mathrm{mm}$) produce broadly similar crack trajectories, although minor differences can be observed in the detailed crack width and local branch morphology. 

The corresponding nonlocal hydrostatic stress fields show pronounced stress concentrations near the active crack tips, while the stress is significantly released along the fully developed crack surfaces. These results indicate that the proposed formulation captures the main features of dynamic crack branching in a qualitatively consistent manner under the mesh refinement considered here.

\begin{figure}[htbp]
				\makebox[\textwidth][c]{  
		\begin{minipage}{1.2\textwidth}
	\centering
	\subfigure[Coarse mesh, $t\approx 45\,\mu\mathrm{s}$]{
		\includegraphics[width=0.475\textwidth]{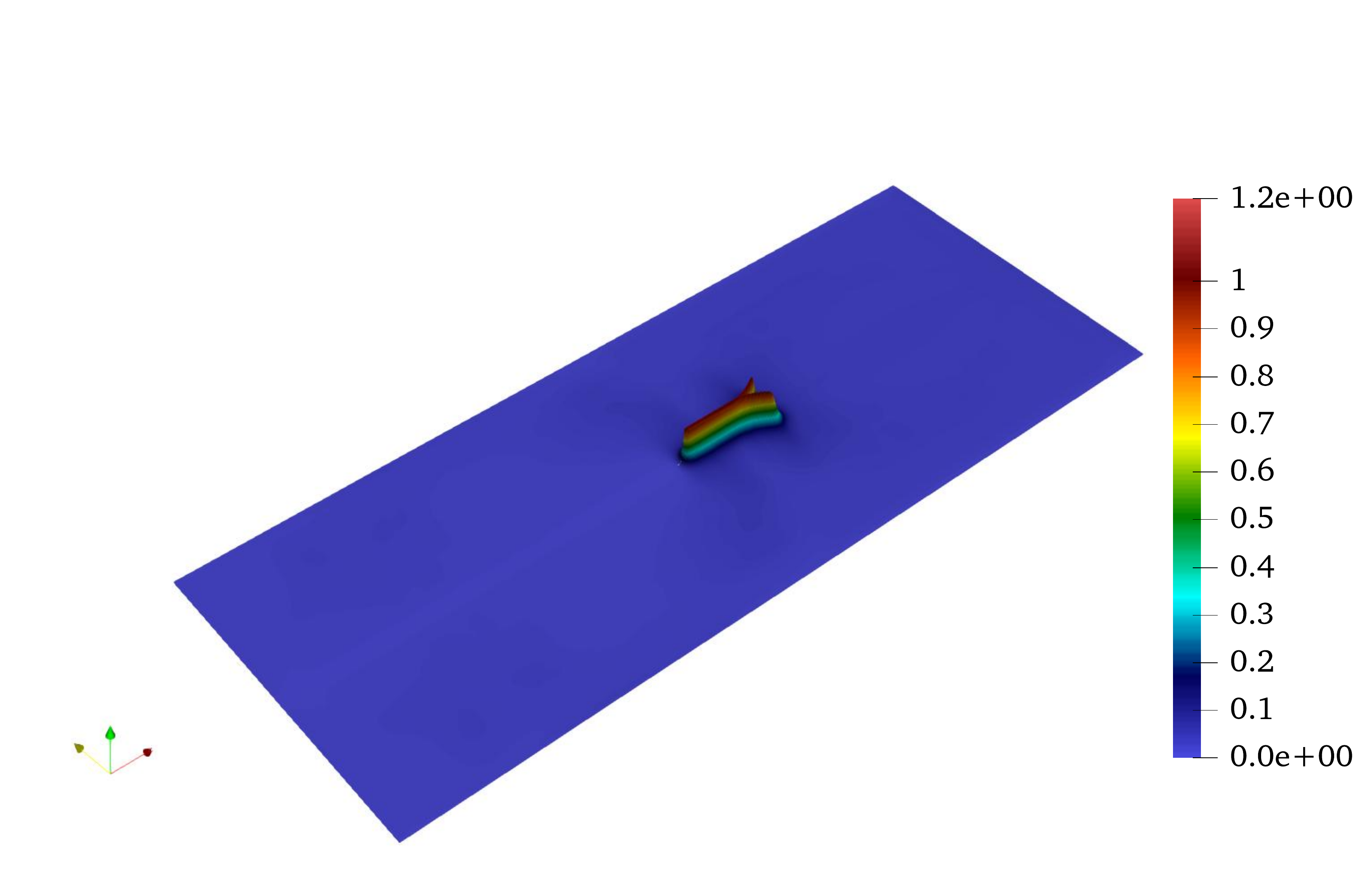}
		\label{fig:dy_plate_coarse_1000}
	}
	\subfigure[Coarse mesh, $t\approx 68\,\mu\mathrm{s}$]{
		\includegraphics[width=0.475\textwidth]{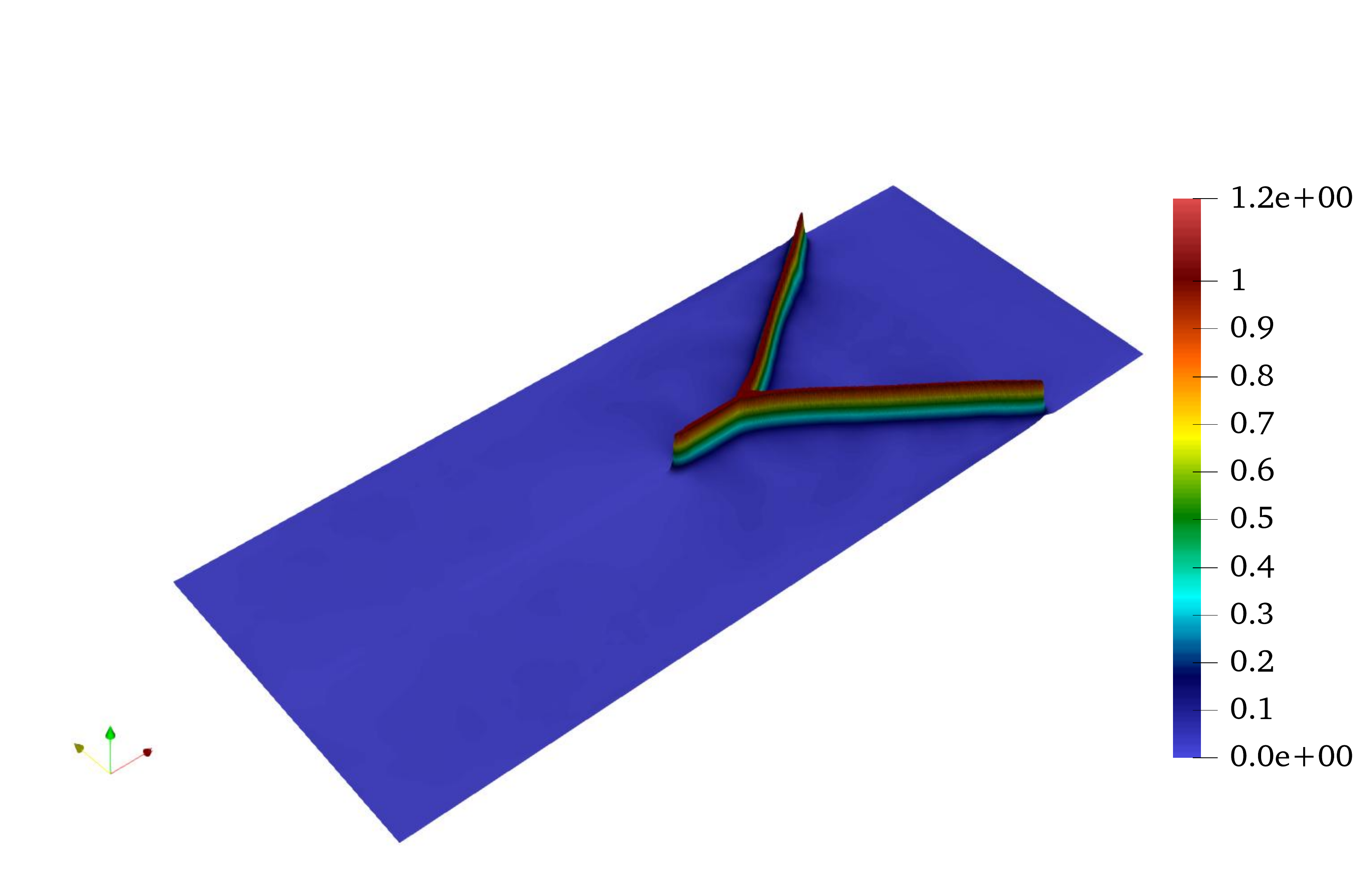}
		\label{fig:dy_plate_coarse_2000}
	}
	
	\subfigure[Fine mesh, $t\approx 45\,\mu\mathrm{s}$]{
		\includegraphics[width=0.475\textwidth]{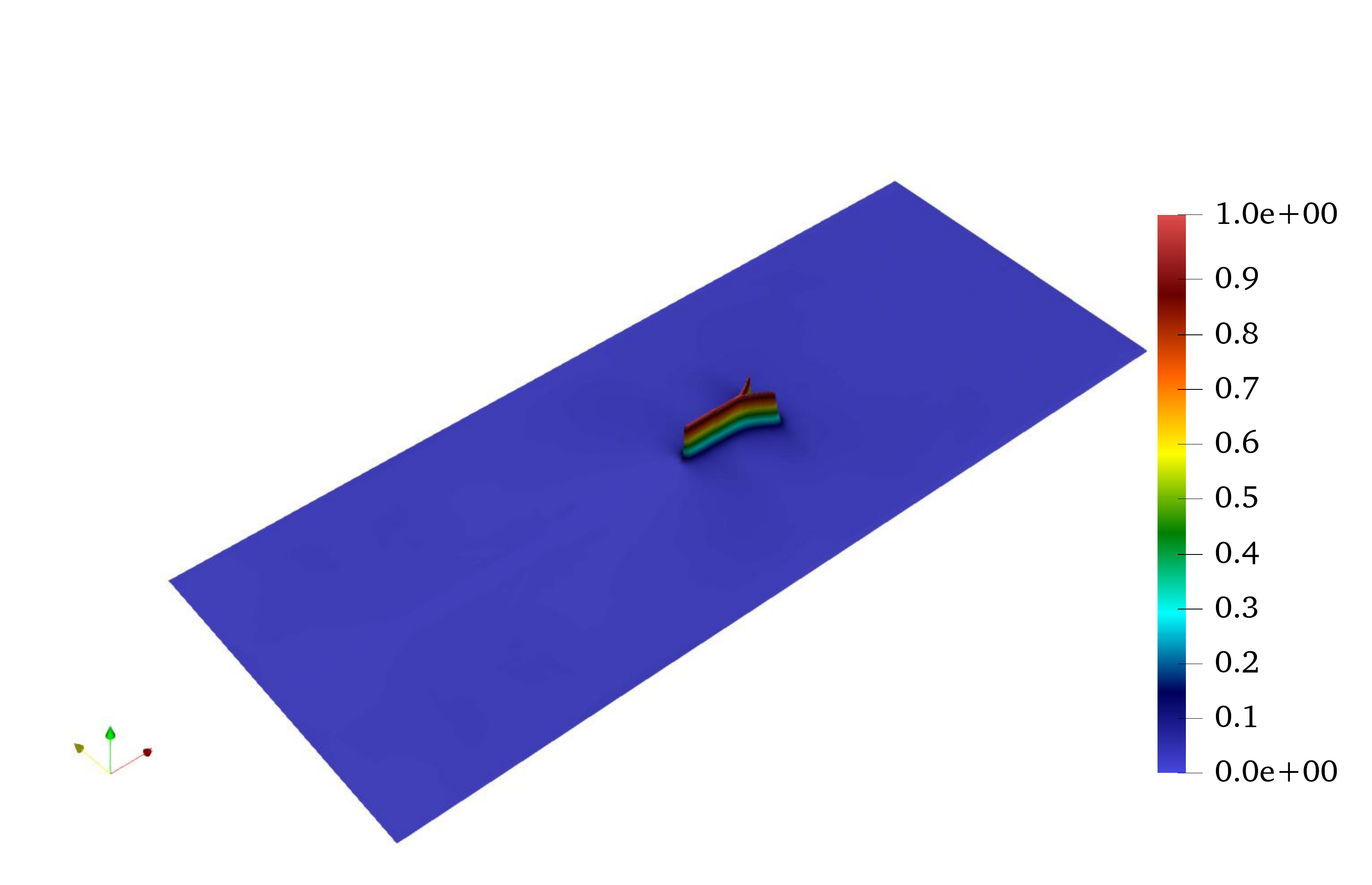}
		\label{fig:dy_plate_fine_1666}
	}
	\subfigure[Fine mesh, $t\approx 68\,\mu\mathrm{s}$]{
		\includegraphics[width=0.475\textwidth]{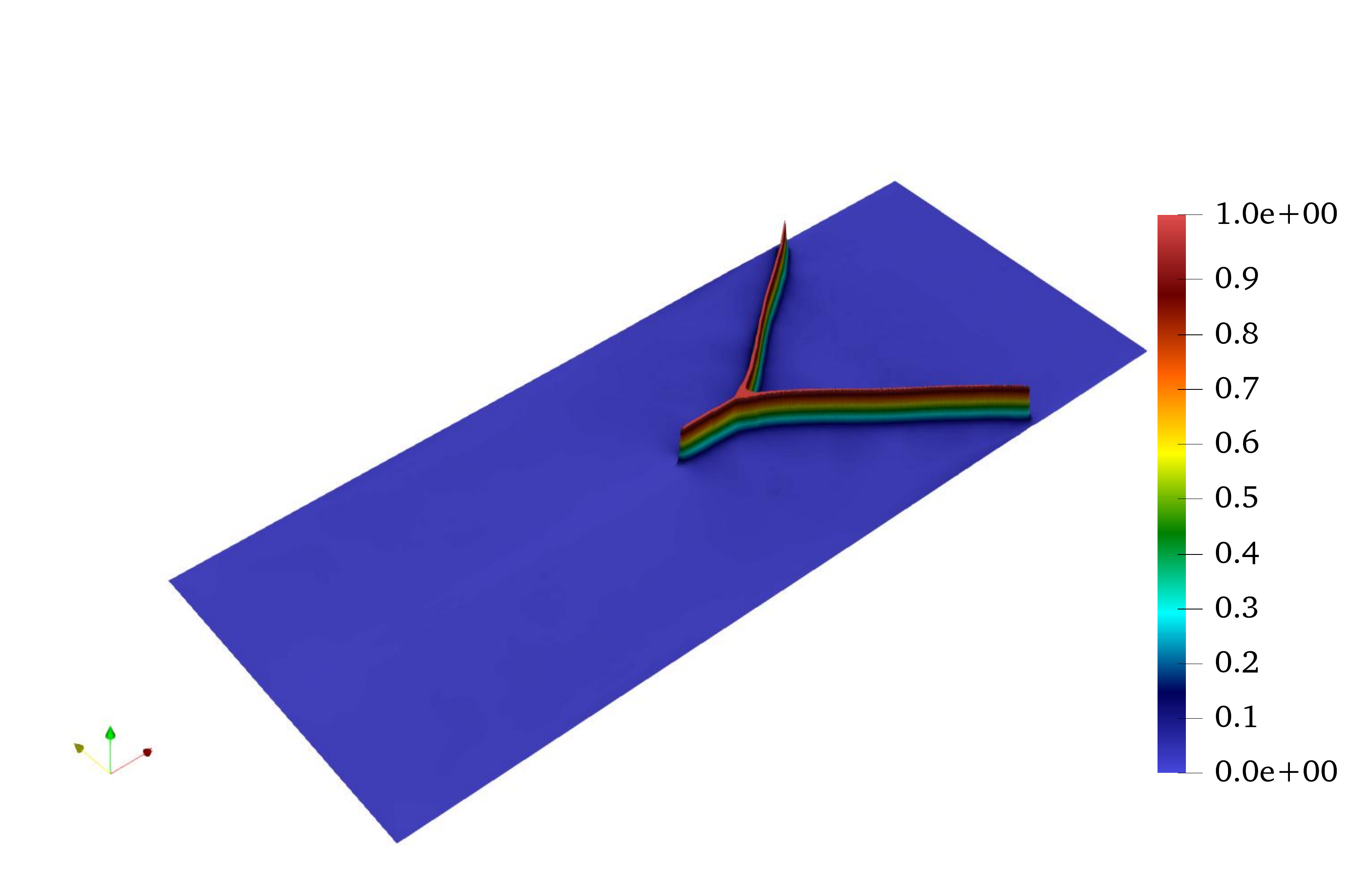}
		\label{fig:dy_plate_fine_3115}
	}
				\end{minipage}
}
\caption{Snapshots of the phase-field $\psi(\mathbf{x},t)$ distribution for the dynamic branching test on coarse and fine meshes.}
\label{fig:dy_phase_field_coarse_fine_mesh}
\end{figure}

\begin{figure}[htbp]
					\makebox[\textwidth][c]{  
		\begin{minipage}{1.2\textwidth}
	\centering
	\subfigure[Coarse mesh, $t\approx45\,\mu\mathrm{s}$]{
		\includegraphics[width=0.475\textwidth]{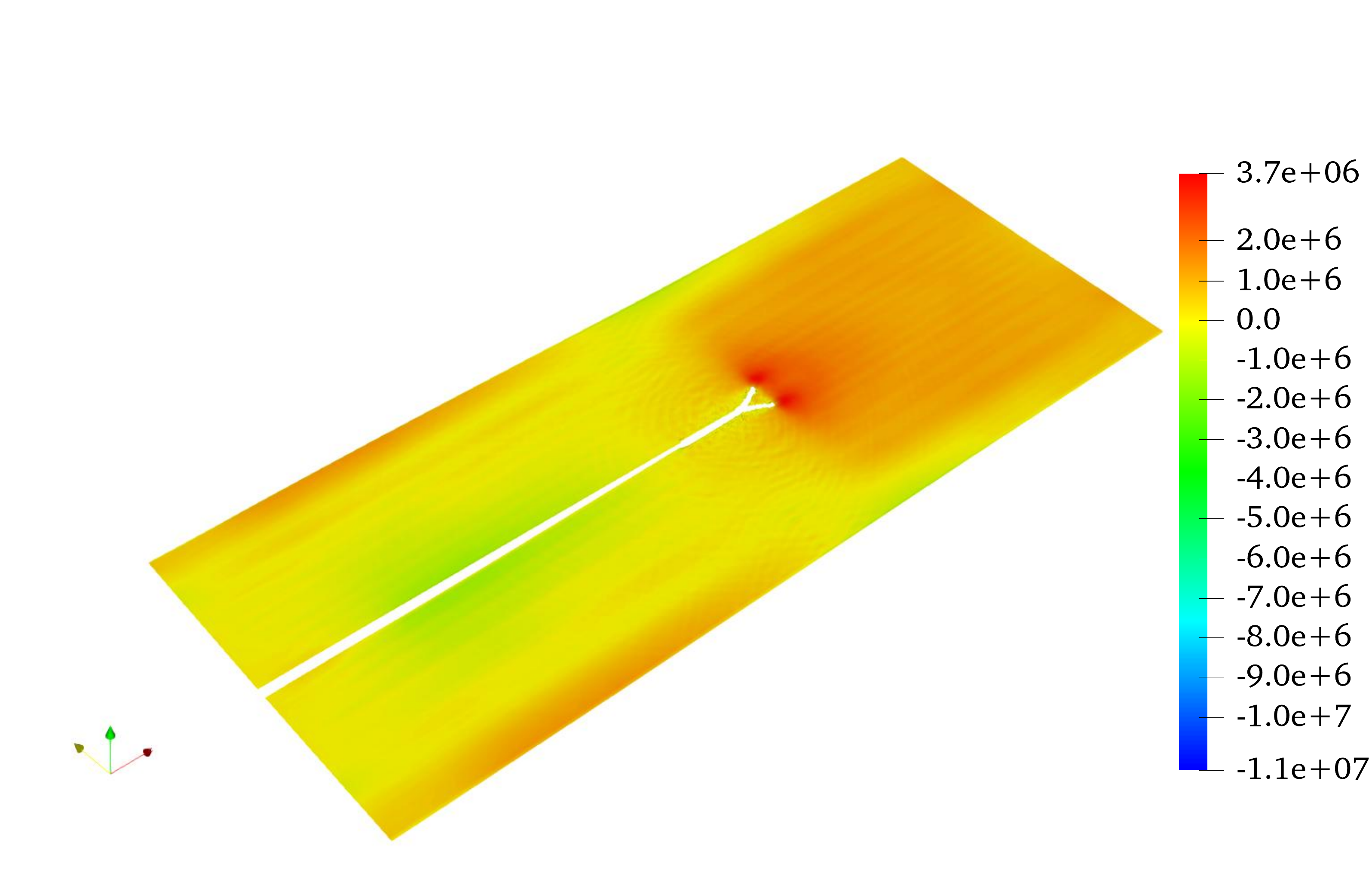}
		\label{fig:dy_plate_stress_coarse_1000}
	}
	\subfigure[Coarse mesh, $t\approx68\,\mu\mathrm{s}$]{
		\includegraphics[width=0.475\textwidth]{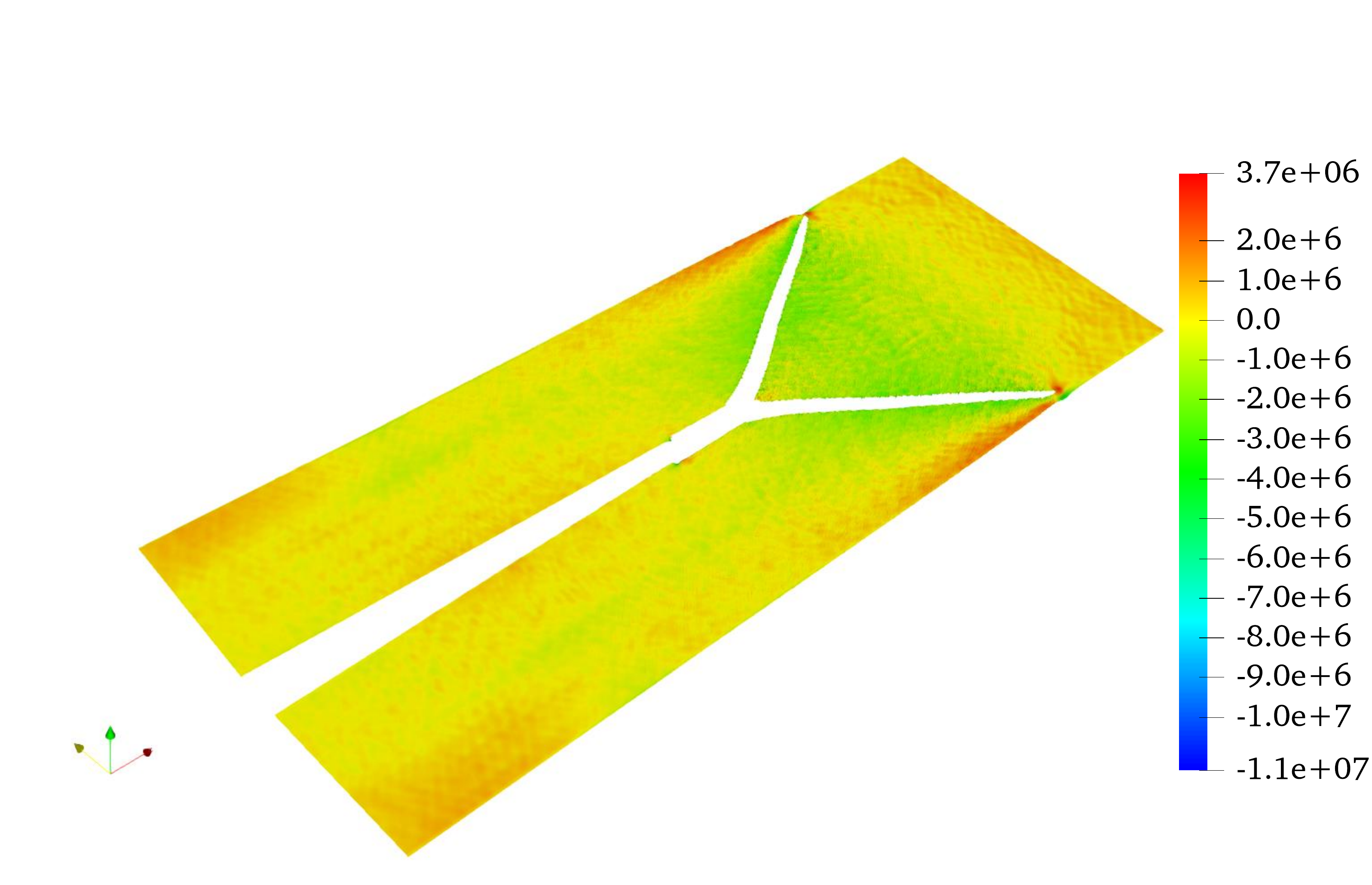}
		\label{fig:dy_plate_stress_coarse_2000}
	}
	
	\subfigure[Fine mesh, $t\approx45\,\mu\mathrm{s}$]{
		\includegraphics[width=0.475\textwidth]{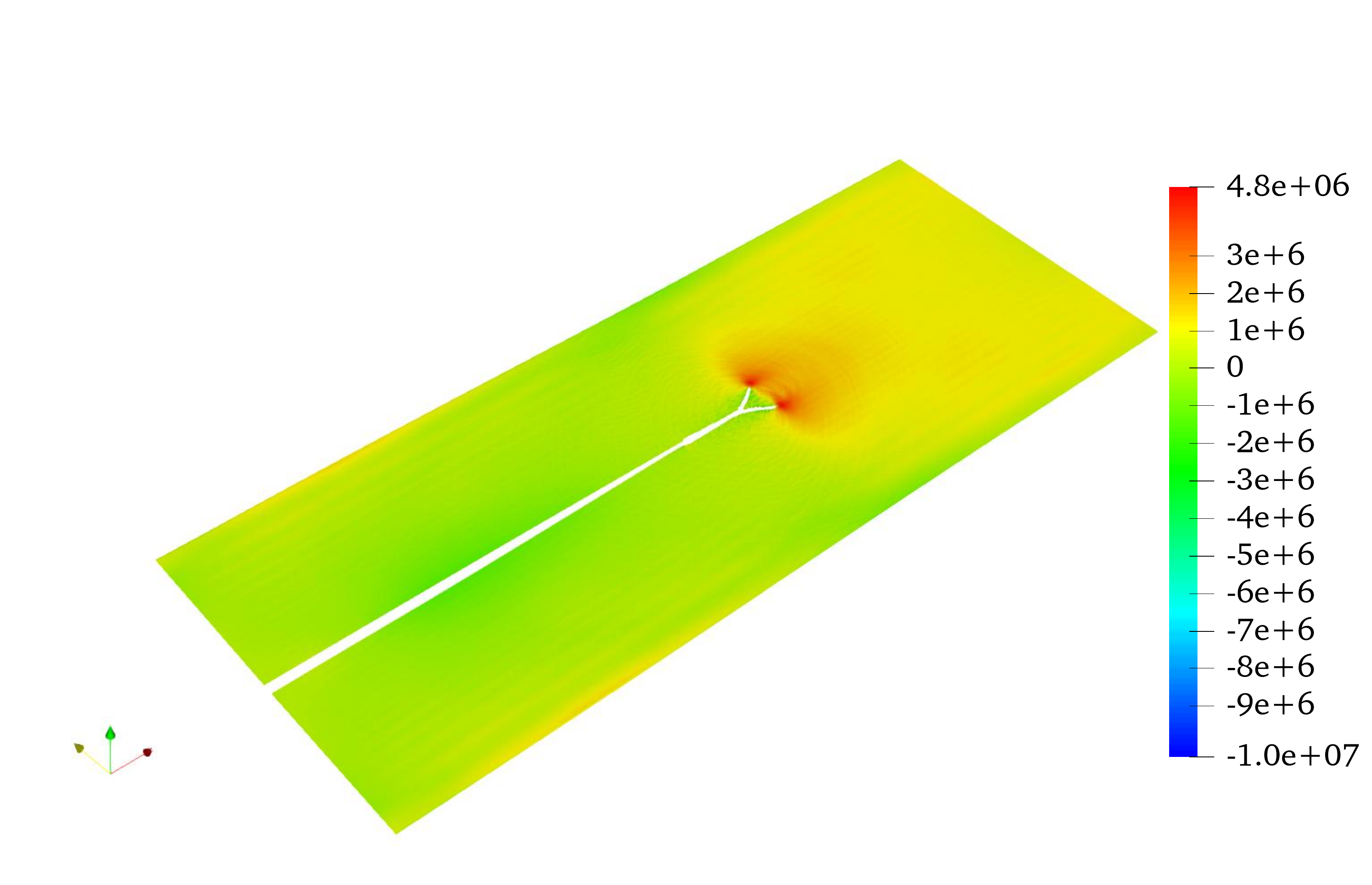}
		\label{fig:dy_plate_stress_fine_1666}
	}
	\subfigure[Fine mesh, $t\approx68\,\mu\mathrm{s}$]{
		\includegraphics[width=0.475\textwidth]{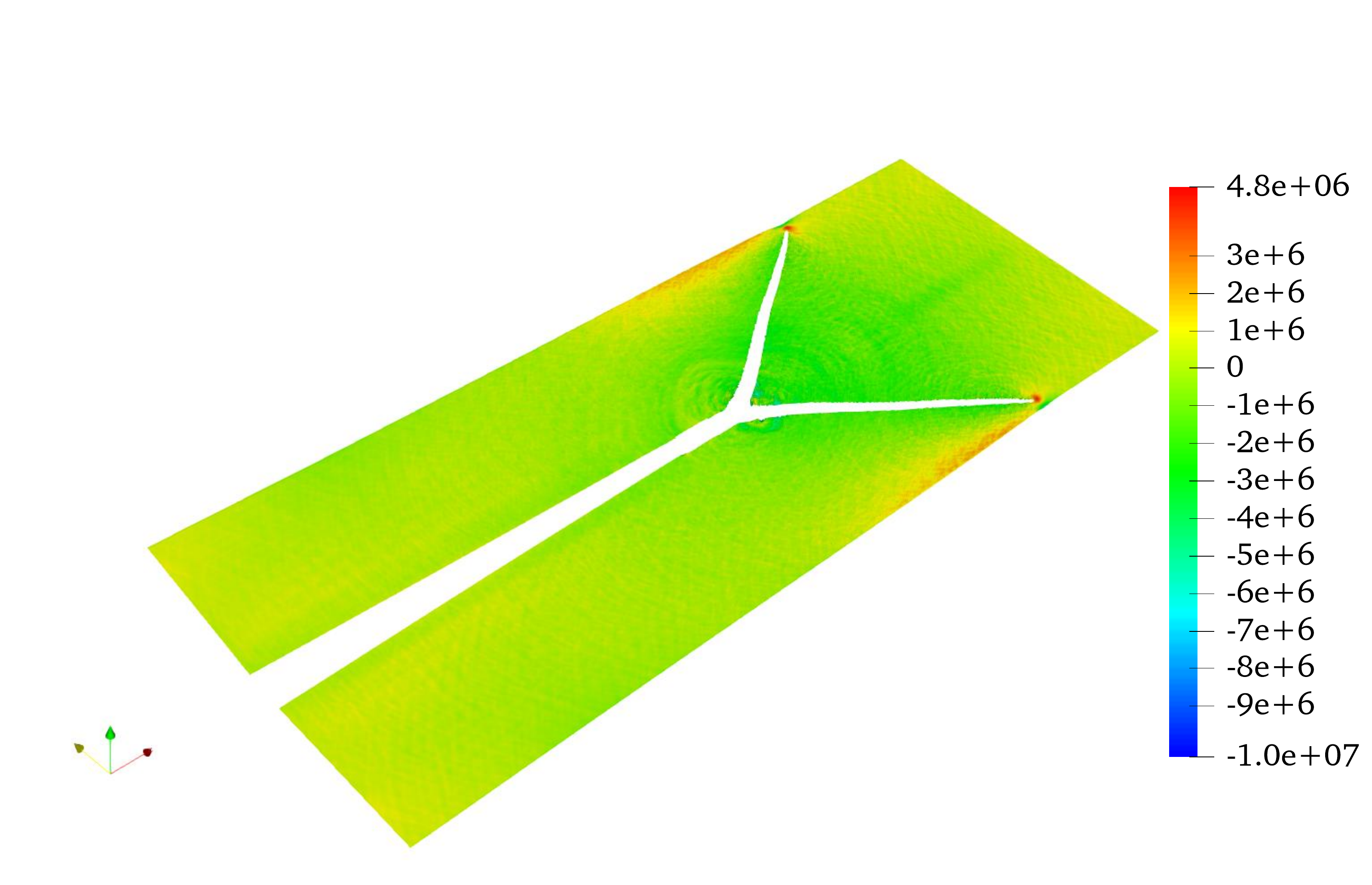}
		\label{fig:dy_plate_stress_fine_3115}
	}
					\end{minipage}
}
	\caption{Snapshots of the nonlocal hydrostatic stress $\sigma_{n,\mathrm{hyd}}(\mathbf{x},t)$ distribution for the dynamic branching test on coarse and fine meshes. The displacement field is magnified by a factor of $50$, and regions with $\psi>0.95$ are omitted for clarity.}
	\label{fig:dy_plate_stress_coarse_fine_mesh}
\end{figure}

\subsection{Dynamic shear loading}
\noindent We next consider the classical Kalthoff--Winkler experiment for dynamic shear failure, in which a projectile impacts a plate with two edge notches. Following the standard low-impact configuration, symmetry is assumed and only the upper half of the specimen is explicitly modeled, as illustrated in Figure~\ref{fig:dy_shear_loading}. The computational domain is taken as $\Omega=[0,0.1]\text{m}\times[0,0.1]\text{m}$, and the plate is initially at rest. To approximate the impact loading, a Dirichlet boundary condition is prescribed for the displacement field on the notched segment of the left boundary,
\begin{equation}\label{eq:dynamic_shear_displacement}
	u_{g,1}(x_1\in[0,\delta],x_2\in[0,0.025],t)=
	\begin{cases}
		\dfrac{v_0}{2t_0}t^2, & t<t_0,\\
		v_0t-\dfrac{1}{2}v_0t_0, & t\ge t_0.
	\end{cases}
\end{equation}
where $v_0=16.5~\mathrm{m/s}$ and $t_0=1~\mu\mathrm{s}$. The lower boundary is treated as a symmetry boundary, whereas the remaining boundaries are traction-free. The physical and numerical parameters used in this example are summarized in Table~\ref{tab:material_parameters_shear}, and plane-strain conditions are assumed throughout the simulation.

\begin{figure}[htbp]
	\centering
	\includegraphics[width=0.75\textwidth]{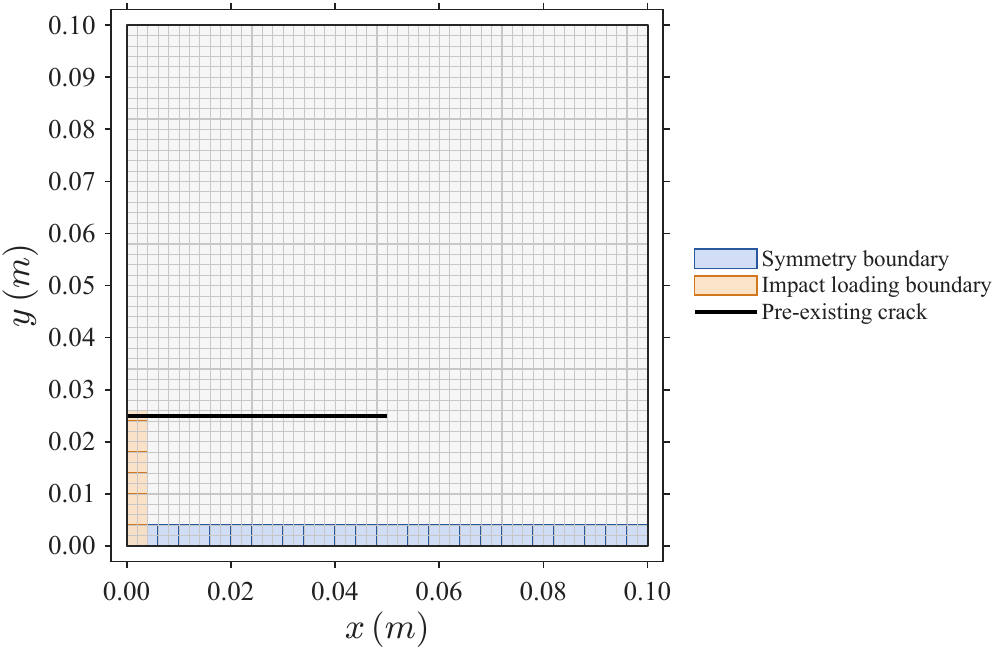}
	\caption{Schematic of the dynamic shear loading test. The mesh shown is for illustration only.}
	\label{fig:dy_shear_loading}
\end{figure}

\begin{table}[htbp]
	\centering
	\caption{Physical and numerical parameters used in the simulations of dynamic shear loading.}
	\label{tab:material_parameters_shear}
	\begin{tabular}{l c c c}
		\toprule
		Physical parameters & Symbol & Unit & Value \\
		\midrule
		Young's modulus              & $E$      & GPa        & 190 \\
		Poisson's ratio              & $\nu$     & --         & 0.3 \\
		Density                      & $\rho$    & kg/m$^{3}$ & $8.00\times10^{3}$ \\
		Critical energy release rate & $G_c$     & J/m$^2$    & $2.21\times10^{4}$ \\
	    Viscous resistance of crack  & $\bar{\eta}$     & s    & $0.0$   \\
		 Rayleigh wave speed     & $c_R$     & m/s        & $2.80\times10^{3}$ \\ 
		\midrule
		Numerical parameters &  &  &  \\
		\midrule
		Time step                             & $\Delta t$       &  $\mu\mathrm{s}$       & $0.25$\\
		Mesh size                    & $h$        & mm         & $0.15$, $0.10$ \\
		Nonlocal interaction length scale  & $\delta$   & mm         & $2.37h$ \\
		Alternating iteration tolerance&$\varepsilon_{\mathrm{tol}}$&--& $1.00\times10^{-4}$\\
		\bottomrule
	\end{tabular}
\end{table}

Figure~\ref{fig:kw_crack_velocity} reports the crack-tip velocity histories, computed from the sampled crack-tip positions using a local three-point linear approximation. After crack initiation, both meshes predict a rapid acceleration followed by a fluctuating propagation stage. The velocities remain well below $0.6c_R$, indicating physically admissible dynamic crack growth. Overall, the coarse and fine meshes produce comparable velocity histories with similar oscillatory characteristics.

\begin{figure}[htbp]
	\centering
	\includegraphics[width=0.6\textwidth]{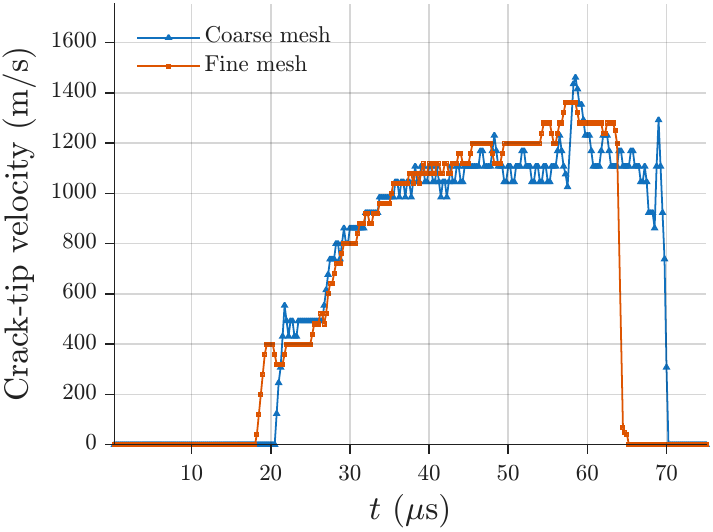}
	\caption{Time evolution of the crack-tip velocity for the coarse and fine meshes in dynamic shear loading.}
	\label{fig:kw_crack_velocity}
\end{figure}

Snapshots of the phase-field $\psi(\mathbf{x},t)$ distribution for the coarse mesh ($h=0.15\mathrm{mm}$) and fine mesh ($h=0.10\mathrm{mm}$) at different times are presented in Figure~\ref{fig:kw_phase_field_coarse_fine_mesh}. As can be seen, in both cases the crack initiates from the notch tip and propagates along an inclined path with an angle of approximately $48^\circ$, which is characteristic of the shear-dominated fracture pattern observed in the Kalthoff--Winkler experiment~\cite{kalthoff2000modes}. Figure~\ref{fig:kw_plate_stress_coarse_fine_mesh} further presents the evolution of the nonlocal hydrostatic stress $\sigma_{n,\mathrm{hyd}}(\mathbf{x},t)$ on the coarse and fine meshes, respectively. It can be observed that the crack-tip stress distribution is consistent with the evolving crack path for both mesh resolutions. These results show that the proposed formulation reproduces the main inclined crack-growth pattern of the Kalthoff--Winkler benchmark. The coarse and fine meshes give qualitatively comparable crack trajectories, while local differences remain visible in the crack-tip morphology and stress-wave patterns.

\begin{figure}[htbp]
						\makebox[\textwidth][c]{  
		\begin{minipage}{1.2\textwidth}
	\centering
	\subfigure[Coarse mesh, $t=37.5\,\mu\mathrm{s}$]{
		\includegraphics[width=0.475\textwidth]{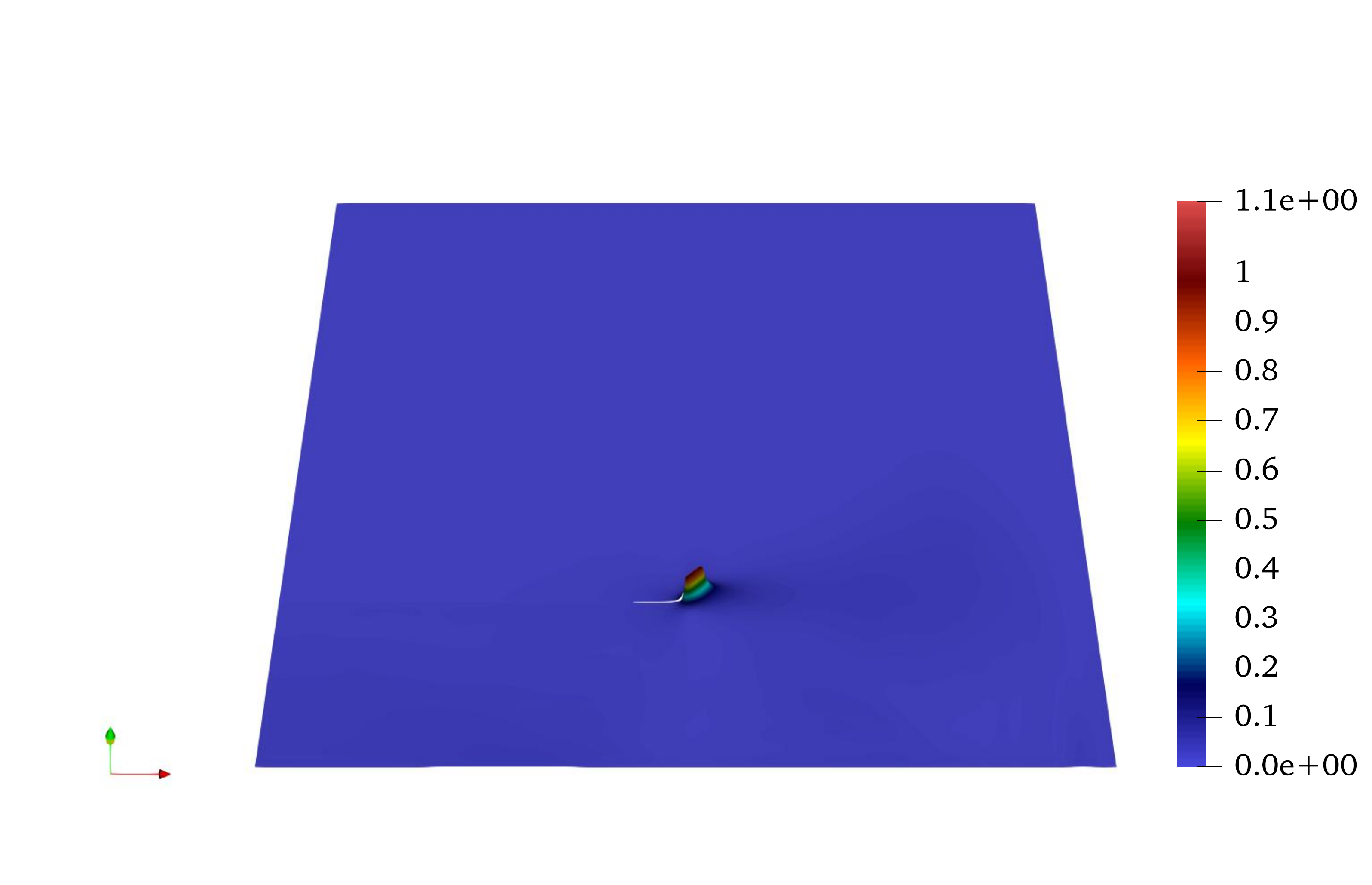}
		\label{fig:kw_plate_coarse_100}
	}
	\subfigure[Coarse mesh, $t=62.5\,\mu\mathrm{s}$]{
		\includegraphics[width=0.475\textwidth]{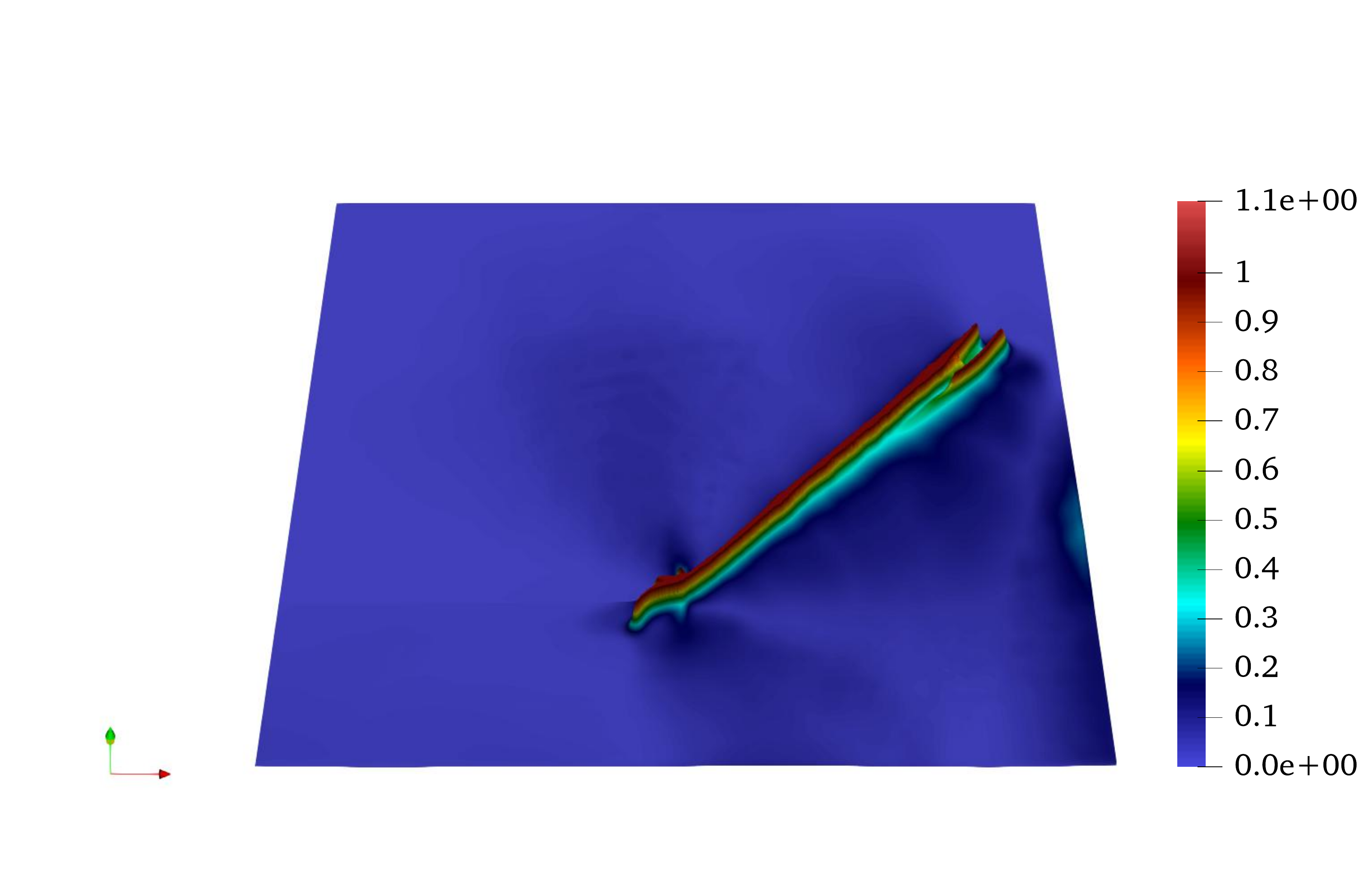}
		\label{fig:kw_plate_coarse_250}
	}
	
	\subfigure[Fine mesh, $t=37.5\,\mu\mathrm{s}$]{
		\includegraphics[width=0.475\textwidth]{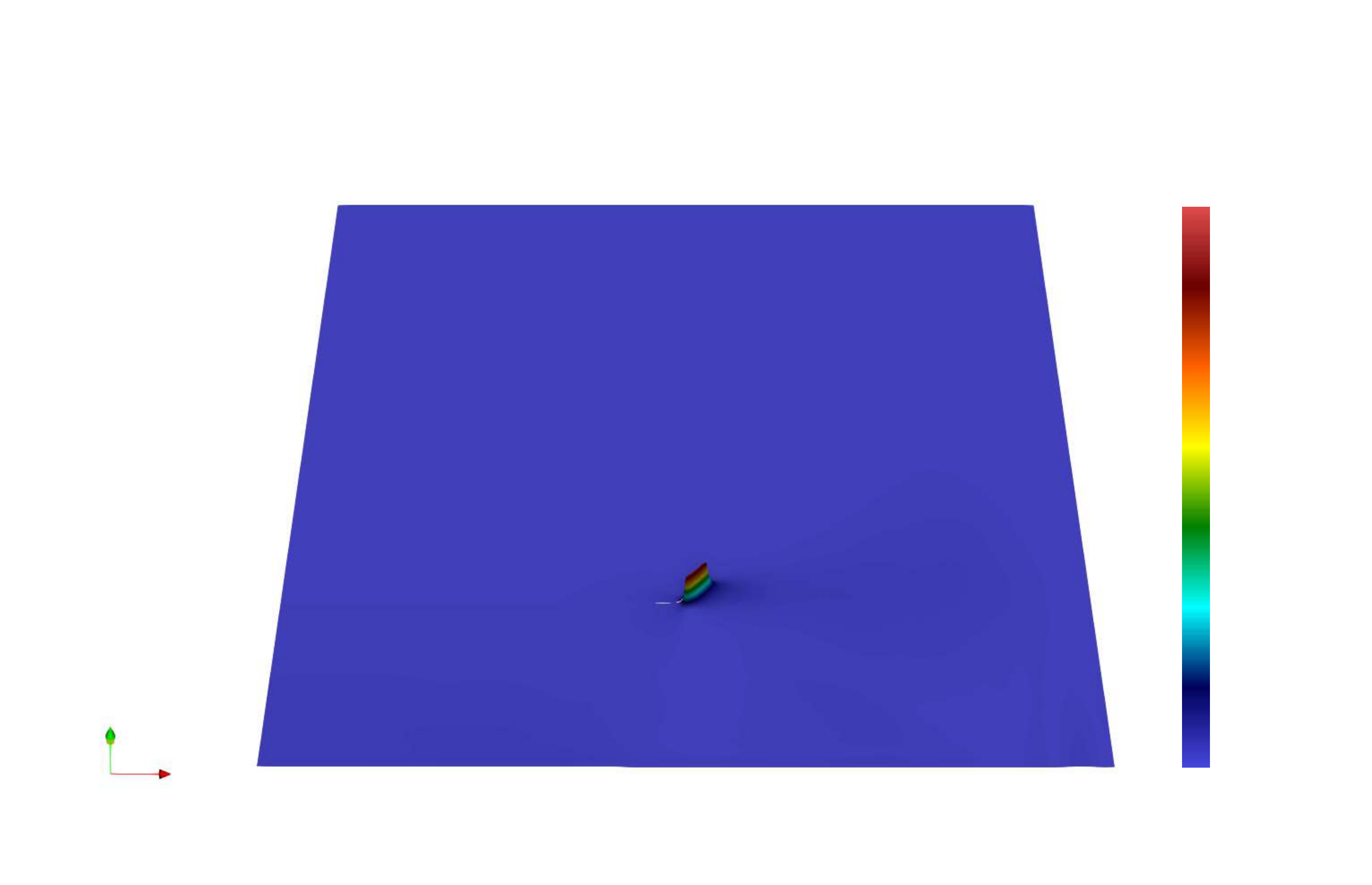}
		\label{fig:kw_plate_fine_100}
	}
	\subfigure[Fine mesh, $t=57.5\,\mu\mathrm{s}$]{
		\includegraphics[width=0.475\textwidth]{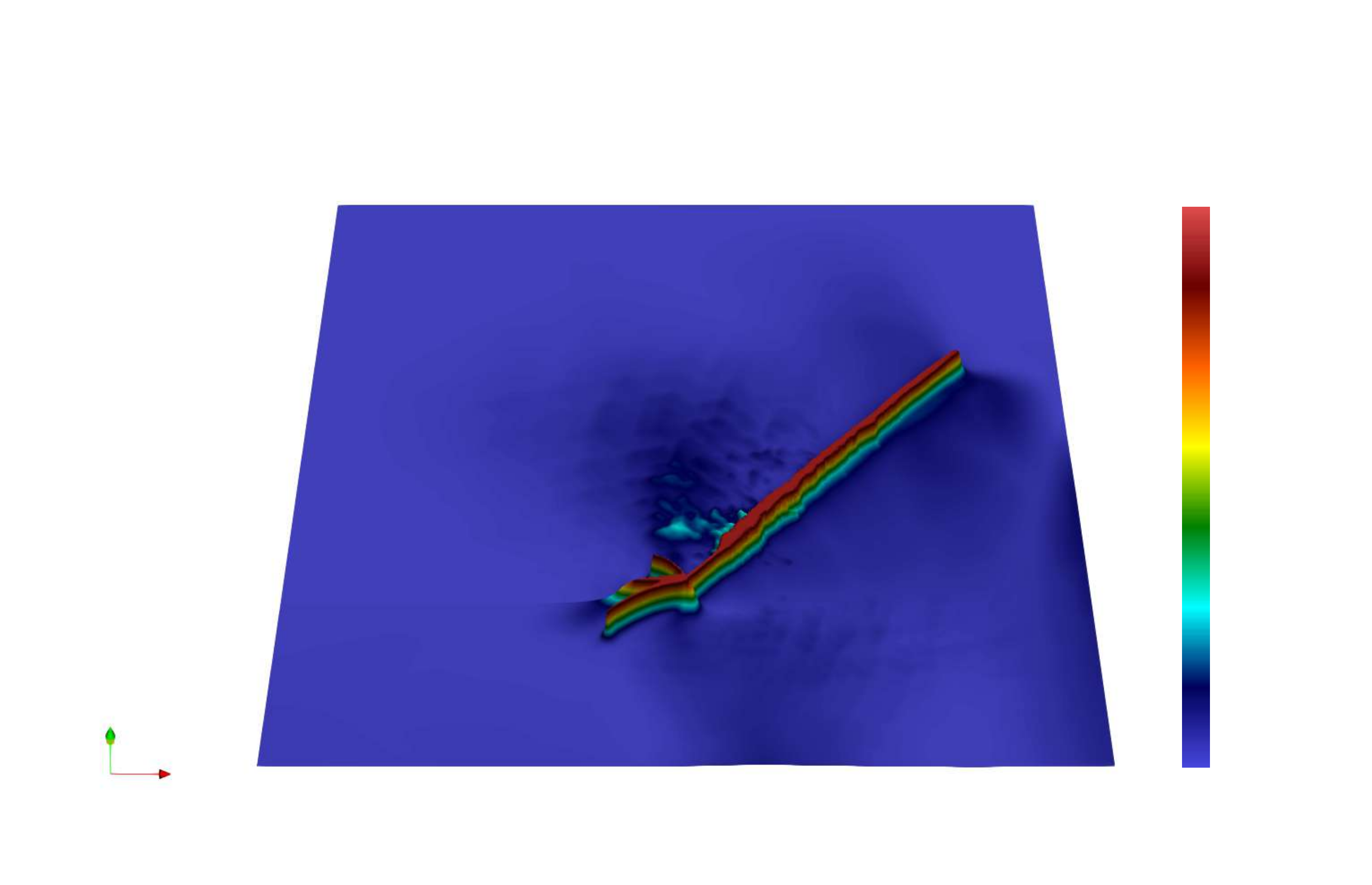}
		\label{fig:kw_plate_fine_230}
	}
\end{minipage}
}
	\caption{Snapshots of the phase-field $\psi(\mathbf{x},t)$ distribution for the Kalthoff--Winkler-type shear fracture test on coarse and fine meshes.}
	\label{fig:kw_phase_field_coarse_fine_mesh}
\end{figure}

\begin{figure}[htbp]
							\makebox[\textwidth][c]{  
		\begin{minipage}{1.2\textwidth}
	\centering
	\subfigure[Coarse mesh, $t=37.5\,\mu\mathrm{s}$]{
		\includegraphics[width=0.475\textwidth]{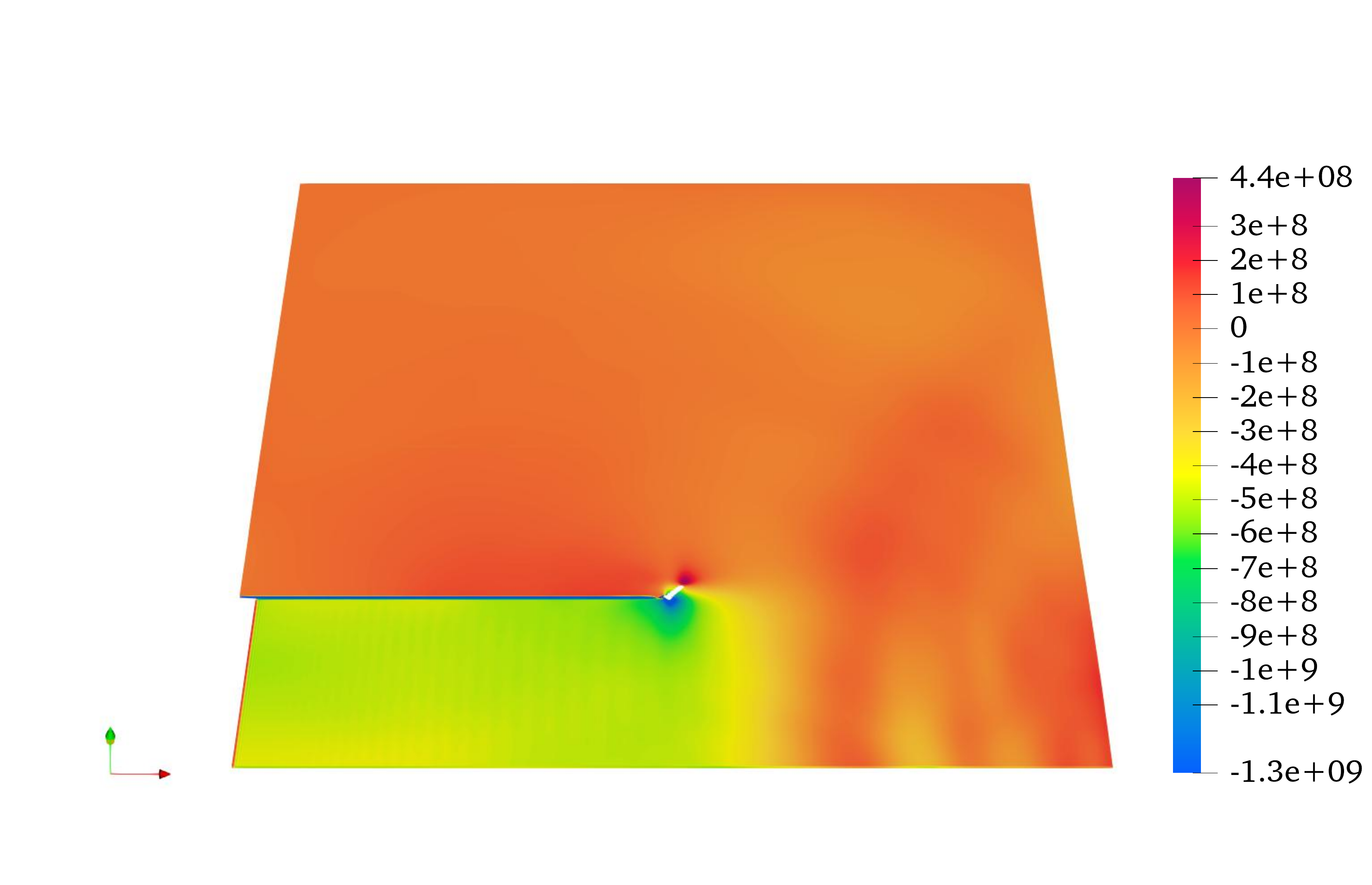}
		\label{fig:kw_plate_stress_coarse_100}
	}
	\subfigure[Coarse mesh, $t=62.5\,\mu\mathrm{s}$]{
		\includegraphics[width=0.475\textwidth]{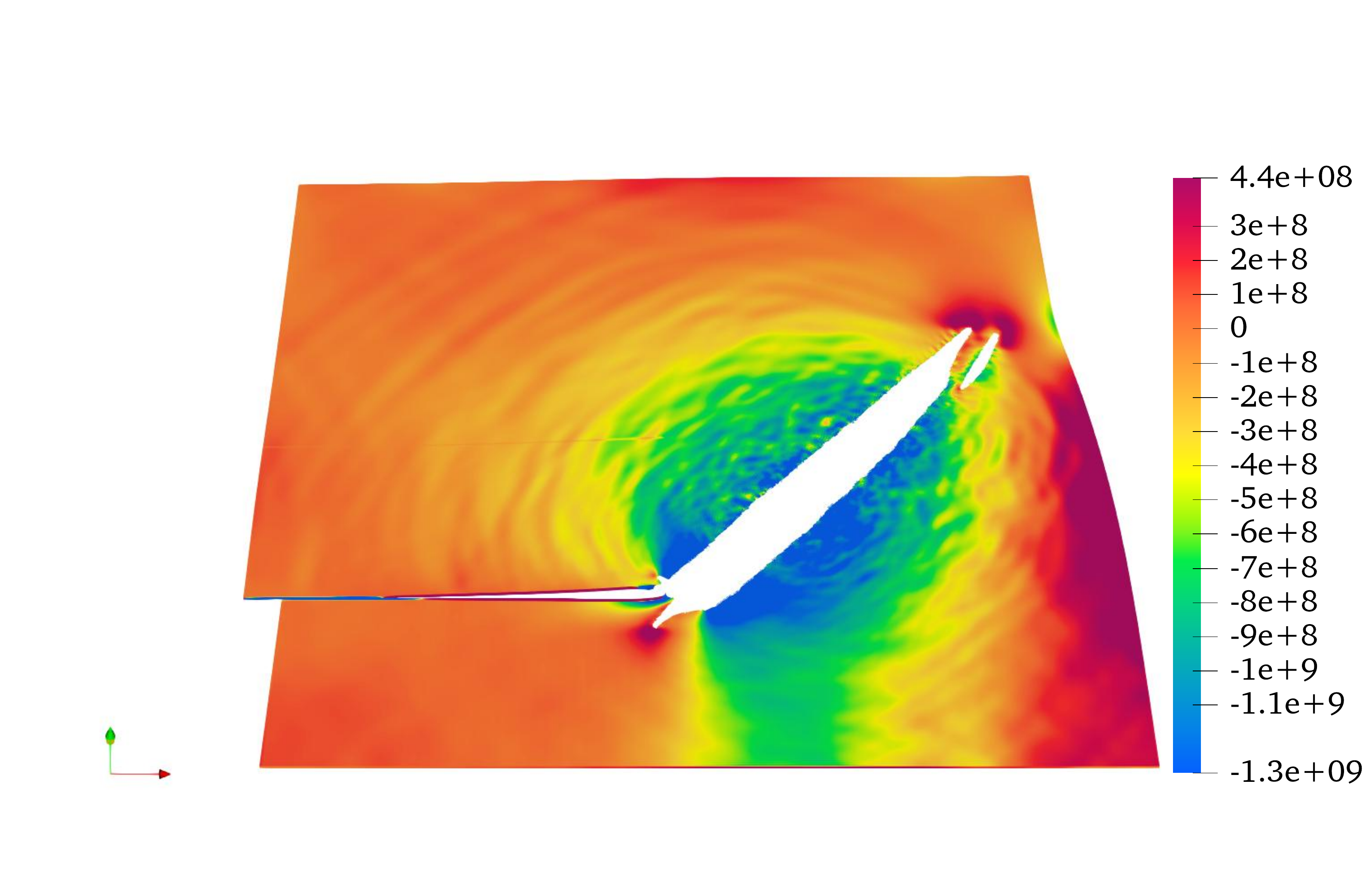}
		\label{fig:kw_plate_stress_coarse_250}
	}
	
	\subfigure[Fine mesh, $t=37.5\,\mu\mathrm{s}$]{
		\includegraphics[width=0.475\textwidth]{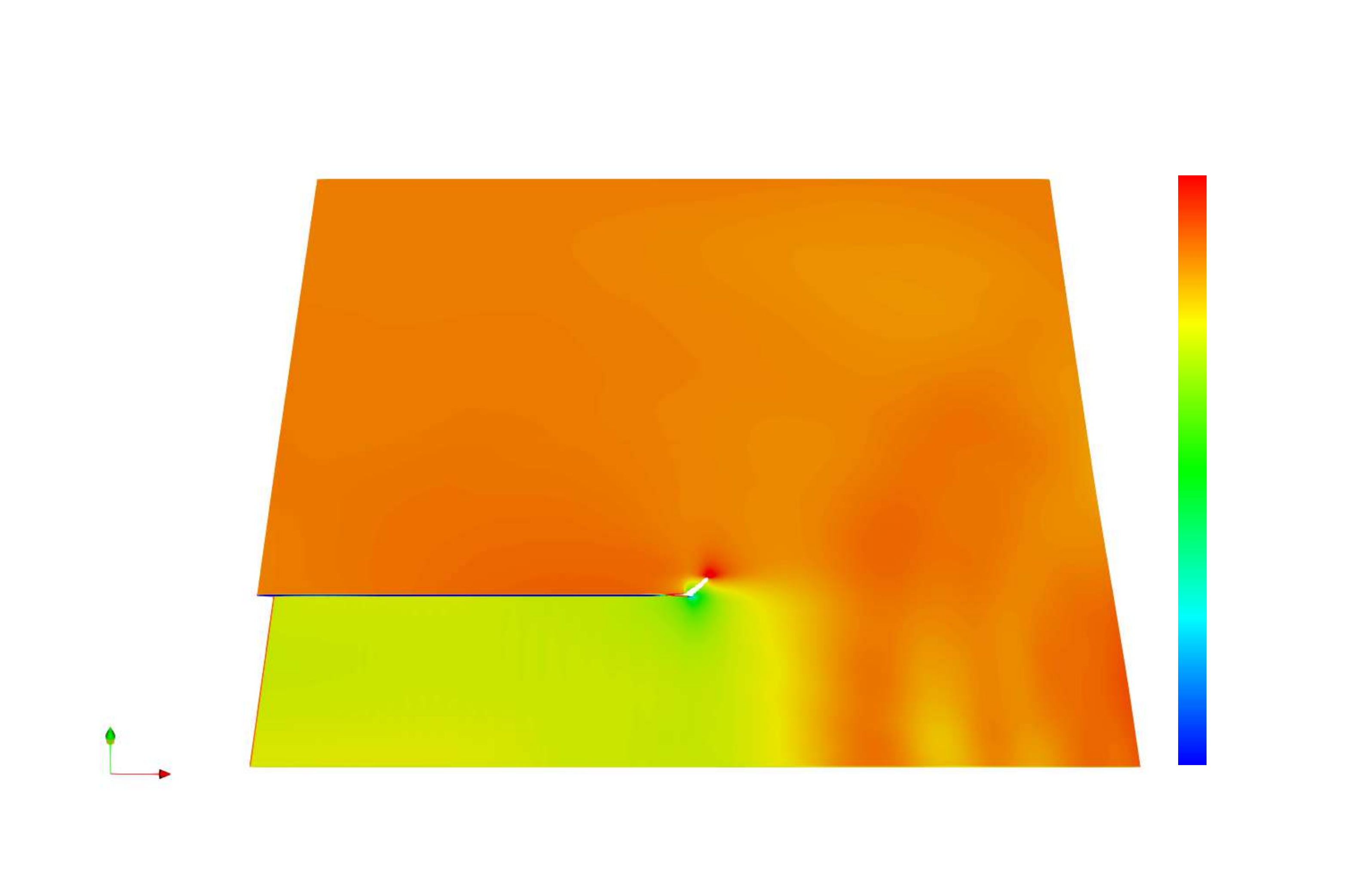}
		\label{fig:kw_plate_stress_fine_100}
	}
	\subfigure[Fine mesh, $t=57.5\,\mu\mathrm{s}$]{
		\includegraphics[width=0.475\textwidth]{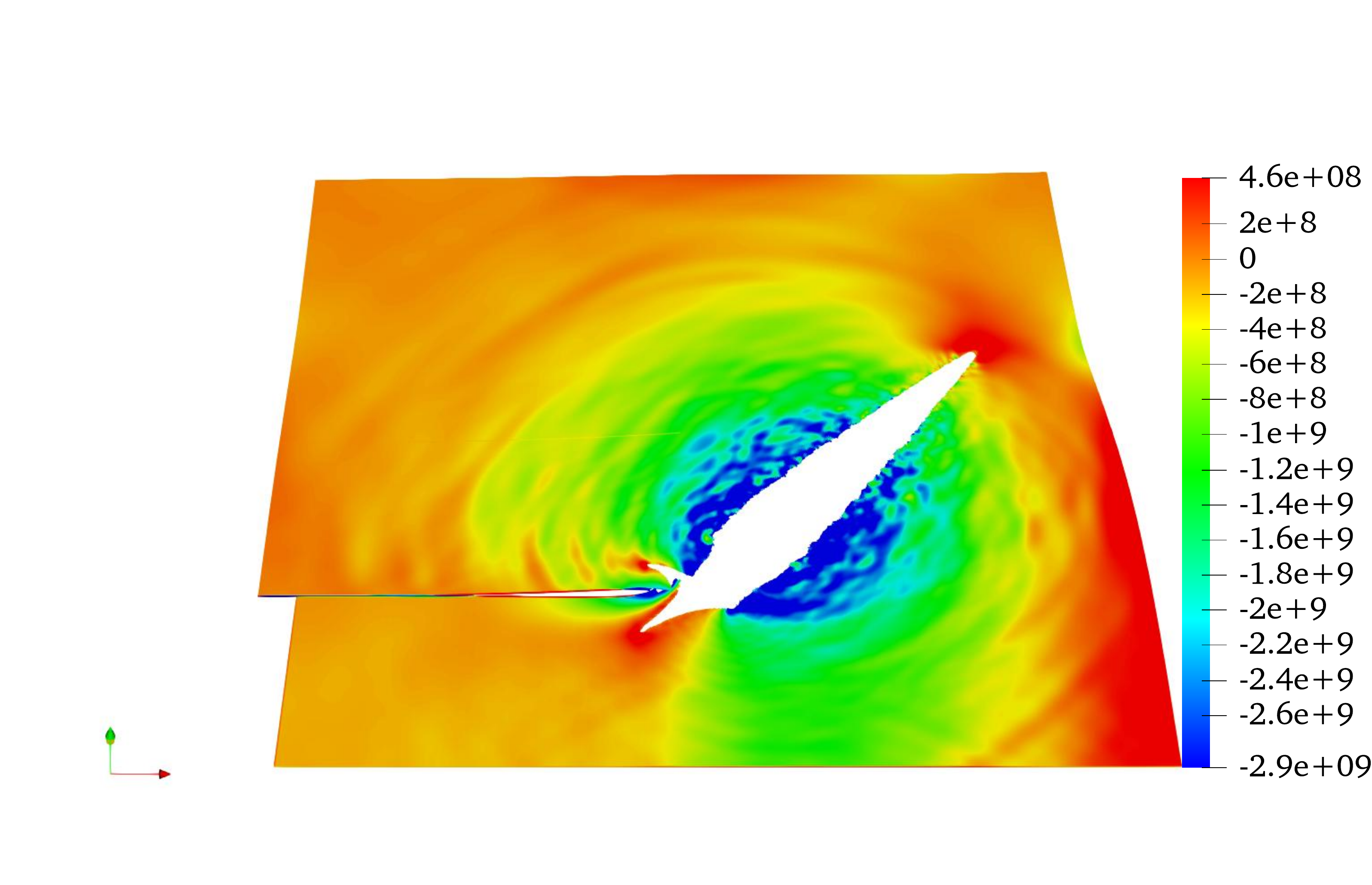}
		\label{fig:kw_plate_stress_fine_230}
	}
						\end{minipage}
}
	\caption{Snapshots of the nonlocal hydrostatic stress $\sigma_{n,\mathrm{hyd}}(\mathbf{x},t)$ distribution for the Kalthoff--Winkler-type shear fracture test on coarse and fine meshes. The displacement field is magnified by a factor of $5$, and regions with $\psi>0.95$ are omitted for clarity.}
	\label{fig:kw_plate_stress_coarse_fine_mesh}
\end{figure}

\subsection{Fragmentation of an annular disk under internal pressure}
\noindent We finally consider the fragmentationof an annular disk \cite{song2009cracking, ren2025dual}  in $\mathbb{R}^2$ subjected to an impulsive pressure applied on its inner boundary, as illustrated in Figure~\ref{fig:hollow_disk}. The inner and outer radii are $80\,\mathrm{mm}$ and $100\,\mathrm{mm}$, respectively. The pressure loading is prescribed as
\begin{equation}
	\mathbf{p}(\mathbf{x}, t) = -p_0 \exp(-t/t_0)\hat{\mathbf{n}}(\mathbf{x}),
\end{equation}
where $\hat{\mathbf{n}}(\mathbf{x})$ denotes the outward unit normal vector on the inner boundary of the computational domain. The negative sign indicates that the pressure acts outward on the material from the inner surface. The initial pressure is $p_0 = 800\,\mathrm{MPa}$, and the decay time is $t_0 = 100\,\mu\mathrm{s}$. The local pressure $\mathbf{p}(\mathbf{x},t)$ is converted into the corresponding nonlocal boundary pressure $\mathbf{p}_n(\mathbf{x},t)$ using the local--nonlocal traction equivalence condition in \eqref{eq:condition}.

\begin{figure}[htbp]
	\centering
	\includegraphics[width=0.5\textwidth]{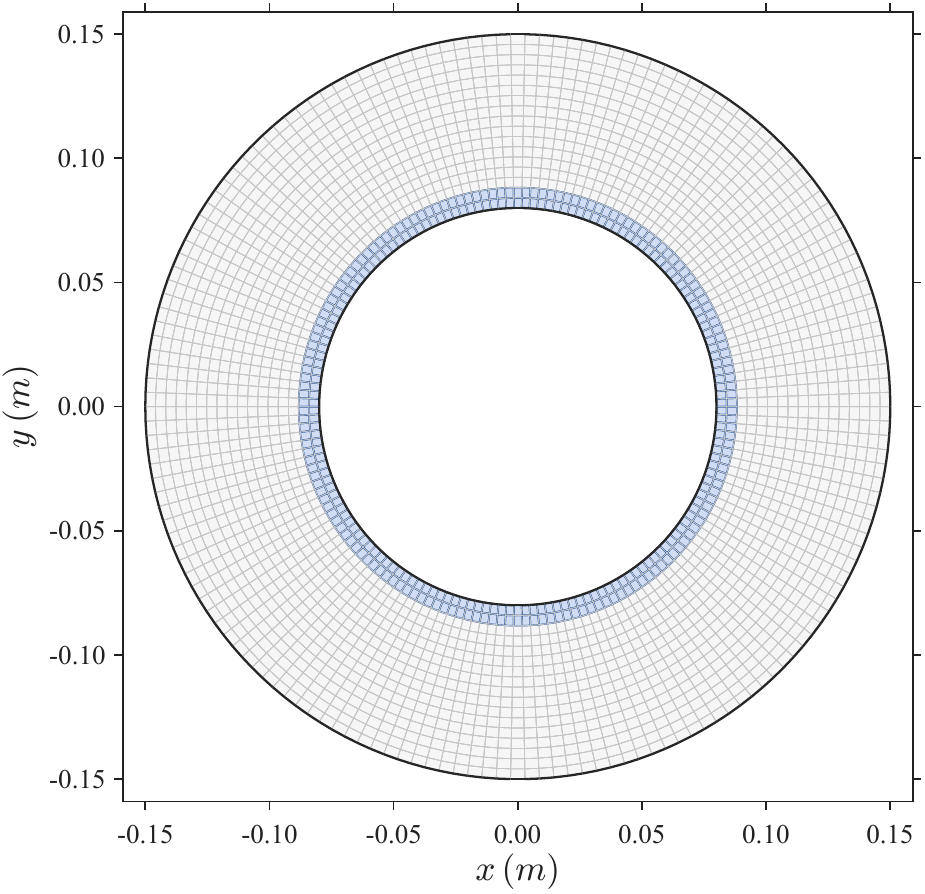}
	\caption{Schematic of the fragmentation of an annular disk, where the blue region denotes the nonlocal traction boundary.}
	\label{fig:hollow_disk}
\end{figure}

The physical and numerical parameters used in this example are summarized in Table~\ref{tab:material_parameters_frag}. To account for material heterogeneity, Young's modulus is perturbed spatially by $5\%$ around its nominal value $E_0$.This example illustrates the ability of the proposed formulation to capture complex crack patterns involving multiple crack nucleation, crack interaction, and fragmentation under impulsive internal pressure.

\begin{table}[htbp]
	\centering
	\caption{Physical and numerical parameters used in the fragmentation simulations.}
	\label{tab:material_parameters_frag}
	\begin{tabular}{l c c c}
		\toprule
		Physical parameters & Symbol & Unit & Value \\
		\midrule
	    Nominal Young's modulus              & $E_0$      & GPa        &  210 \\
		Poisson's ratio              & $\nu$     & --         & 0.3 \\
		Density                      & $\rho$    & kg/m$^{3}$ & $7.85\times10^3$ \\
		Critical energy release rate & $G_c$     & J/m$^2$     & $2.00\times10^4$   \\
		Viscous resistance of crack  & $\bar{\eta}$     & s & $0.0$   \\
		\midrule
		Numerical parameters &  &  &  \\
		\midrule
		Time step                             & $\Delta t$       &  $\mu\mathrm{s}$       & $0.2$\\
		Mesh size                              & $h$       & mm         & 1.0  \\
		Nonlocal interaction length scale  & $\delta$  & mm         &     $2.37h$   \\
		Alternating iteration tolerance & $\varepsilon_{\mathrm{tol}}$& --&$1.00\times10^{-3}$\\
		\bottomrule
	\end{tabular}
\end{table}

Figure~\ref{fig:sphere_phase_field_stress} shows the distributions of the phase-field variable 
$\psi(\mathbf{x},t)$ and the corresponding nonlocal hydrostatic stress 
$\sigma_{n,\mathrm{hyd}}(\mathbf{x},t)$ at two representative times. Multiple cracks nucleate from the inner boundary of the annular disk and subsequently propagate toward the outer boundary. As the pressure wave evolves, the crack pattern becomes increasingly complex, with several competing crack branches developing around the annulus, demonstrating the capability of the proposed formulation to capture multiple crack nucleation, crack interaction, and fragmentation under impulsive loading. The stress distributions exhibit pronounced stress concentrations near the active crack tips, while stress release occurs along the fully developed crack surfaces.

\begin{figure}[htbp]
\makebox[\textwidth][c]{  
		\begin{minipage}{1.2\textwidth}
	\centering
	\subfigure[Phase-field, $t=48\,\mu\mathrm{s}$]{
		\includegraphics[width=0.475\textwidth]{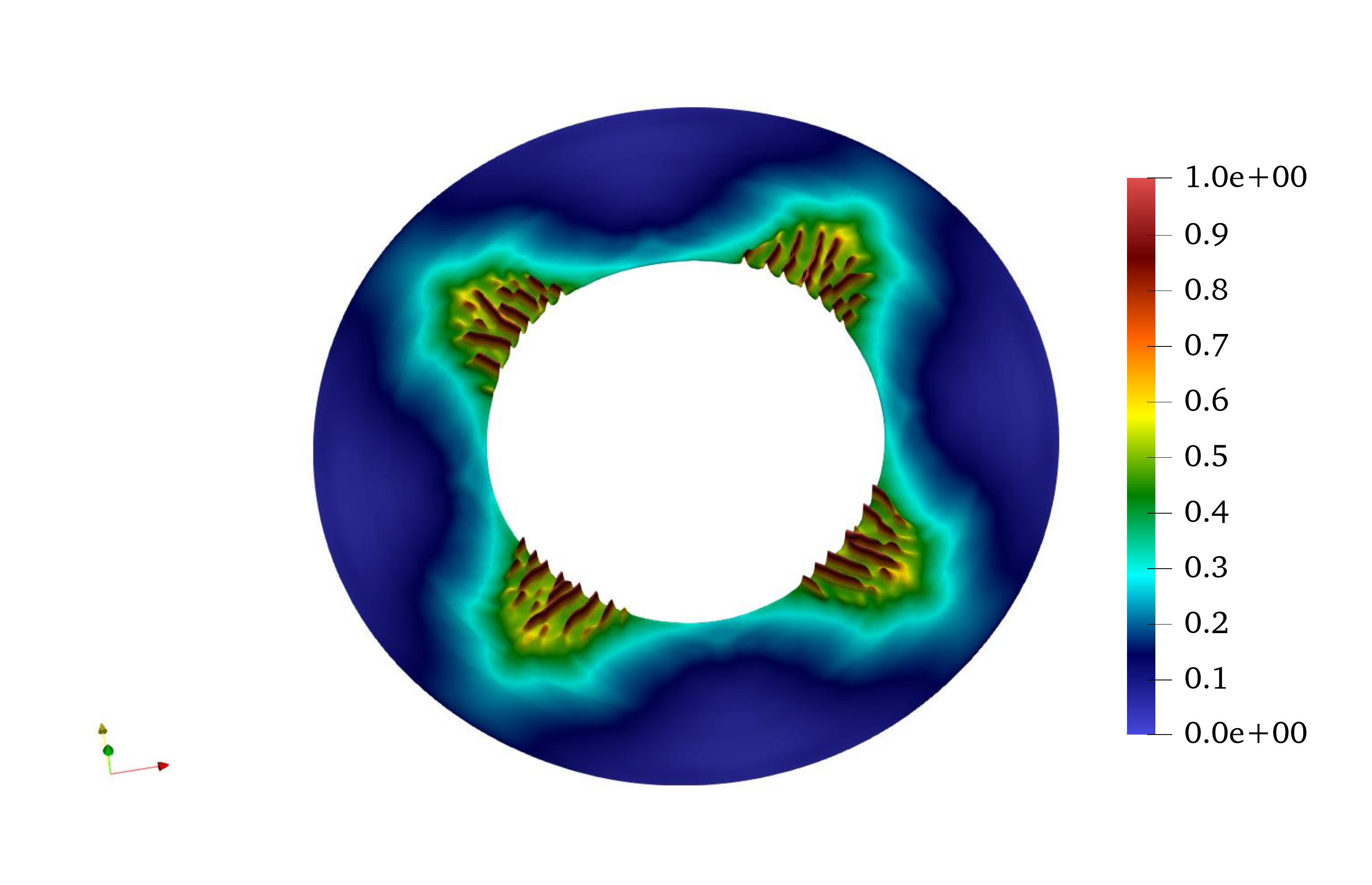}
		\label{fig:sphere_800MPa_210}
	}
	\subfigure[Phase-field, $t=60\,\mu\mathrm{s}$]{
		\includegraphics[width=0.475\textwidth]{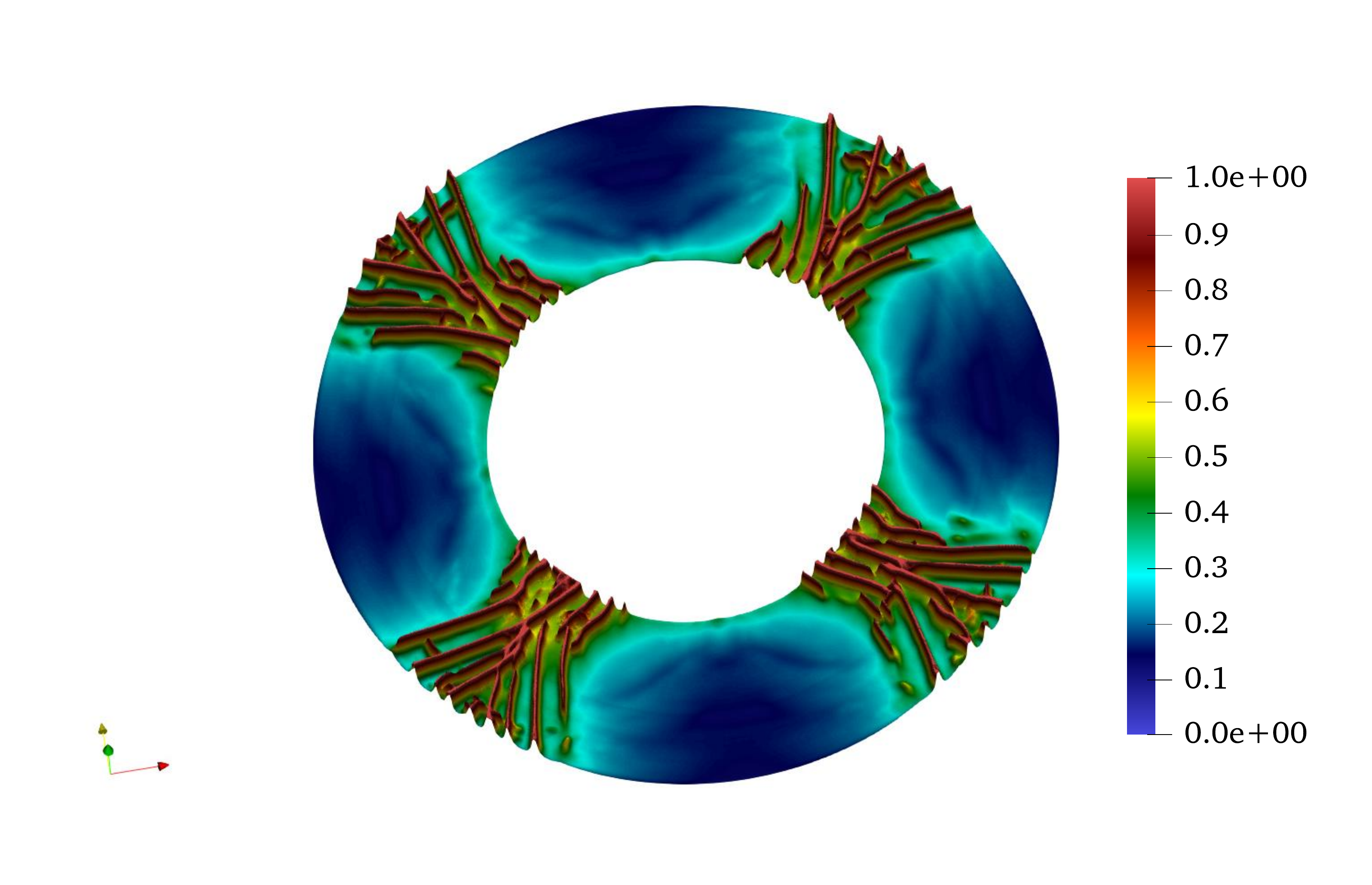}
		\label{fig:sphere_800MPa_300}
	}
	
	\subfigure[Hydrostatic stress, $t=48\,\mu\mathrm{s}$]{
		\includegraphics[width=0.475\textwidth]{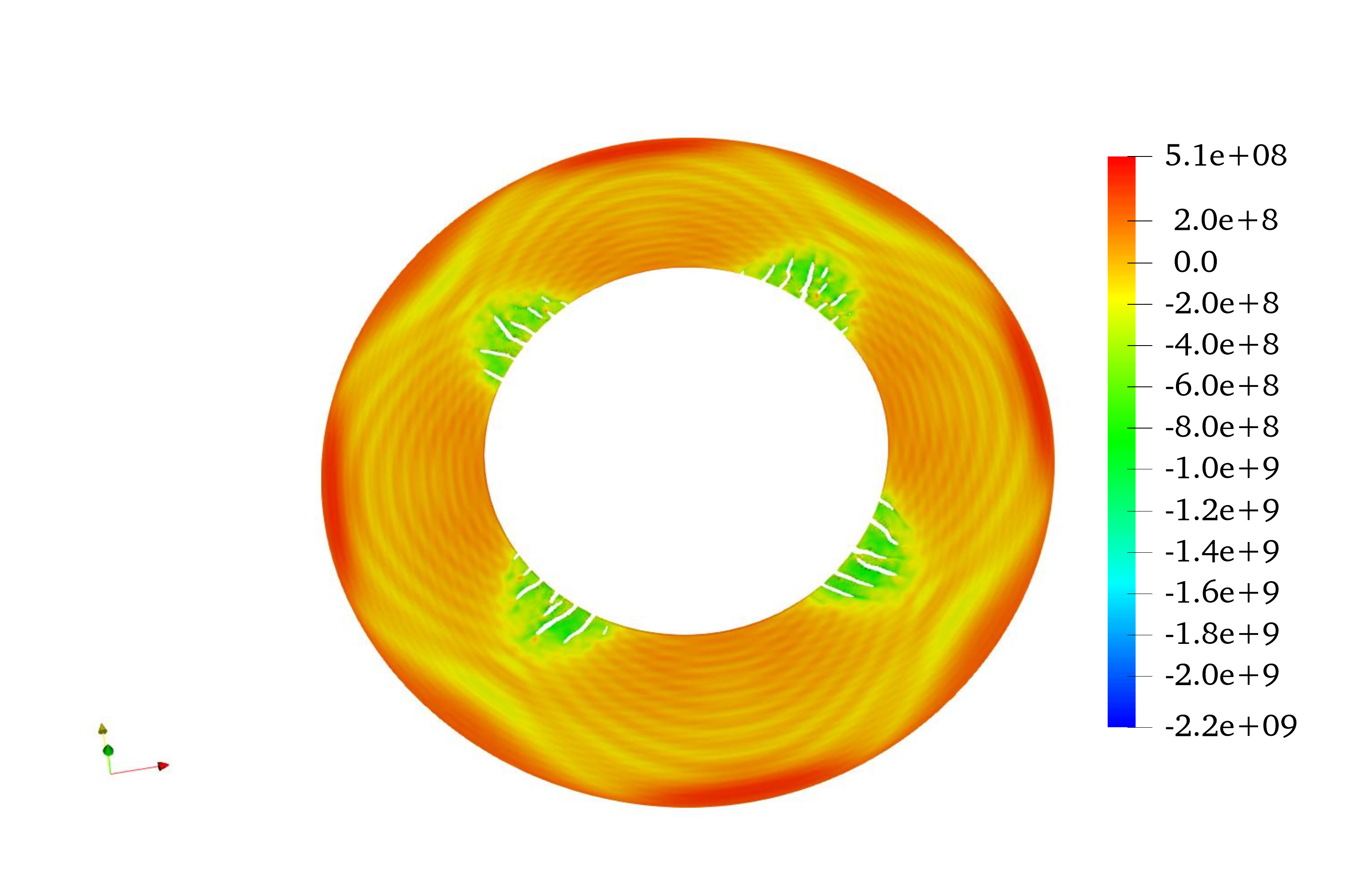}
		\label{fig:sphere_stress_fine_210}
	}
	\subfigure[Hydrostatic stress, $t=60\,\mu\mathrm{s}$]{
		\includegraphics[width=0.475\textwidth]{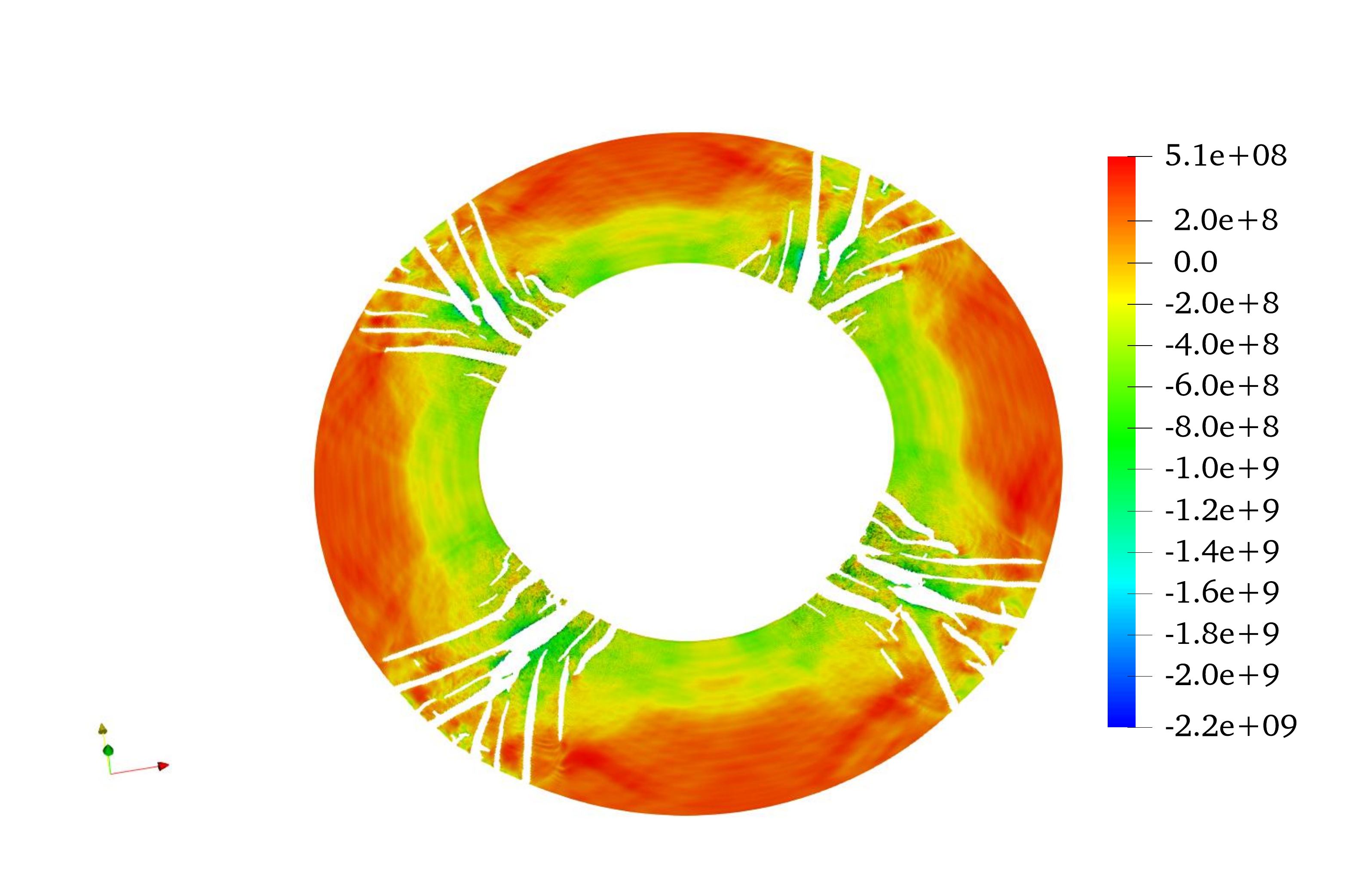}
		\label{fig:sphere_stress_fine_300}
	}
    \end{minipage}
}
	\caption{Snapshots of the phase-field $\psi(\mathbf{x},t)$ and the nonlocal hydrostatic stress $\sigma_{n,\mathrm{hyd}}(\mathbf{x},t)$ distributions at two representative times. In the stress plots, the displacement field is magnified by a factor of $5$, and regions with $\psi>0.95$ are omitted for clarity.}
	\label{fig:sphere_phase_field_stress}
\end{figure}

\section{Conclusions and future work}
\label{sec:conclusion}

\noindent This paper developed a variational nonlocal phase-field formulation for dynamic fracture in elastic solids. The model combines nonlocal kinematics, a nonlocal crack-surface functional, and an irreversible phase-field evolution law within a unified variational framework. The main conclusions are summarized as follows:

\begin{itemize}
	
		\item  A kernel-dependent nonlocal formulation was constructed for the displacement and phase-field variables. It allows weaker regularity requirements than classical local continuum formulations and recovers the local theory as the nonlocal interaction length scale vanishes.
	
		\item A nonlocal crack-surface functional was introduced as an integral counterpart of the Ambrosio--Tortorelli regularization. The diffusive-crack width is implicitly controlled by the nonlocal interaction domain, rather than by an independently prescribed length scale.
	
		\item A variationally consistent displacement--phase-field system was developed, together with a structure-preserving SAV scheme and a staggered alternating scheme. The numerical examples show that the proposed model captures Mode-I crack propagation, dynamic branching, shear-dominated fracture, and fragmentation-like crack patterns without explicit crack tracking.
	
\end{itemize}

	Several limitations remain. The present numerical examples are restricted to two-dimensional problems, and a rigorous well-posedness theory and fully discrete energy stability analysis for the staggered scheme are still open. In addition, the element-center approximation of the discrete interaction neighborhood may introduce mesh-dependent geometric errors. Future work will focus on the well-posedness and convergence analysis of the proposed model, fully discrete energy estimates, systematic studies with respect to $h$, $\delta$, and $\delta/h$, and three-dimensional extensions using efficient interaction-domain construction, such as octree-based or hierarchical neighbor-search strategies. Extensions to multiphysical fracture problems will also be investigated.

\section*{Acknowledgements}
\noindent

% \appendix

\appendix
\section{Proof of Lemma~\ref{lem:nonlocal-gradient-estimate}} \label{app:proof-nonlocal-gradient-estimate}
\noindent \begin{proof}
	For simplicity, the dependence on  $t$ is omitted. Then, for any $\mathbf{x}\in\Omega$,
	by the Cauchy--Schwarz inequality, we obtain
	\begin{equation}
		\begin{aligned}
			\bigl|\mathcal{G}_\delta(\mathbf{u})(\mathbf{x})\bigr|^2
			&= \biggl|
			\int_{\Omega}
			\big(\mathbf{u}(\mathbf{y})-\mathbf{u}(\mathbf{x})\big)
			\otimes \boldsymbol{\omega}_\delta(\mathbf{x},\mathbf{y})\,
			\mathrm{d}v_\mathbf{y}
			\biggr|^2 \\[2mm]
			&\le 
			\biggl(
			\int_{\Omega}
			\lvert\mathbf{u}(\mathbf{y})-\mathbf{u}(\mathbf{x})\rvert^2
			\,\omega_\delta(\lvert\mathbf{y}-\mathbf{x}\rvert)
			\,\mathrm{d}v_\mathbf{y}
			\biggr)
			\biggl(
			\int_{\Omega}
			\frac{\lvert\boldsymbol{\omega}_\delta(\mathbf{x},\mathbf{y})\rvert^2}
			{\omega_\delta(\lvert\mathbf{y}-\mathbf{x}\rvert)}
			\,\mathrm{d}v_\mathbf{y}
			\biggr).
		\end{aligned}
	\end{equation}
	Assumption~{[\textbf{A2}]} ensures that there exists a constant
	$C_\omega>0$ such that
	\begin{equation}
		\int_{\Omega}
		\frac{\lvert\boldsymbol{\omega}_\delta(\mathbf{x},\mathbf{y})\rvert^2}
		{\omega_\delta(\lvert\mathbf{y}-\mathbf{x}\rvert)}
		\,\mathrm{d}v_\mathbf{y}\leq\int_{B_{\delta}(\mathbf{0})}|\bm{\xi}|^2\omega_\delta(|\bm{\xi}|)
		\,\mathrm{d}v_{\bm{\xi}}
		=: C_\omega,
		\quad\forall\mathbf{x}\in\Omega.
	\end{equation}
	Hence, 
	\begin{equation}
		\bigl|\mathcal{G}_\delta(\mathbf{u})(\mathbf{x})\bigr|^2
		\le C_\omega
		\int_{\Omega}
		\lvert\mathbf{u}(\mathbf{y})-\mathbf{u}(\mathbf{x})\rvert^2
		\omega_\delta(\lvert\mathbf{y}-\mathbf{x}\rvert)
		\,\mathrm{d}v_\mathbf{y}.
	\end{equation}
	Integrating over $\mathbf{x}\in\Omega$ yields
	\begin{equation}
		\|\mathcal{G}_\delta(\mathbf{u})\|_{L^2(\Omega)}^2
		\le C_\omega 
		\int_{\Omega}\int_{\Omega}
		\lvert\mathbf{u}(\mathbf{y})-\mathbf{u}(\mathbf{x})\rvert^2
		\omega_\delta(\lvert\mathbf{y}-\mathbf{x}\rvert)\,
		\mathrm{d}v_\mathbf{y}\mathrm{d}v_\mathbf{x}
		= C_\omega\,[\mathbf{u}]_{\mathcal{V}_\omega(\Omega)}^2,
	\end{equation}
	which proves \eqref{eq:nonlocal-G-upper} and shows that 
	$\mathcal{G}_\delta(\mathbf{u})\in L^2(\Omega; \mathbb{R}^{d\times d})$ whenever 
	$[\mathbf{u}]_{\mathcal{V}_\omega(\Omega)}<\infty$.
	
\end{proof}

\section{Algebraic details for the proof of Theorem~\ref{thm:Hamiltonian_balance}}
\label{app:energy-dissipation-proof}

In this appendix, we provide the intermediate algebraic steps omitted in the proof of
Theorem~\ref{thm:Hamiltonian_balance}. Multiplying the nonlocal momentum balance by
$\dot{\mathbf u}$ and integrating over $\Omega\setminus\Gamma_\delta$ gives
\begin{equation}
	\begin{aligned}
		\int_{\Omega\setminus\Gamma_\delta}
		\rho\ddot{\mathbf u}\cdot\dot{\mathbf u}\,\mathrm dv
		&=
		\int_{\Omega\setminus\Gamma_\delta}
		\mathcal D_\delta(\boldsymbol\sigma_\delta^{\mathrm d})
		\cdot\dot{\mathbf u}\,\mathrm dv
		+
		\int_{\Omega\setminus\Gamma_\delta}
		\mathbf b\cdot\dot{\mathbf u}\,\mathrm dv .
	\end{aligned}
\end{equation}
Using the nonlocal integration-by-parts identity and the nonlocal traction
condition on $\Gamma_{\delta,t}$, we obtain
\begin{equation}
	\begin{aligned}
		\int_{\Omega\setminus\Gamma_\delta}
		\mathcal D_\delta(\boldsymbol\sigma_\delta^{\mathrm d})
		\cdot\dot{\mathbf u}\,\mathrm dv
		&=
		-\int_\Omega
		\boldsymbol\sigma_\delta^{\mathrm d}:
		\dot{\boldsymbol\varepsilon}_\delta
		\,\mathrm dv
		+
		\int_{\Gamma_{\delta,t}}
		\mathbf t\cdot\dot{\mathbf u}\,\mathrm dv .
	\end{aligned}
\end{equation}
Moreover,
\begin{equation}
	\int_{\Omega\setminus\Gamma_\delta}
	\rho\ddot{\mathbf u}\cdot\dot{\mathbf u}\,\mathrm dv
	=
	\frac{\mathrm d}{\mathrm dt}
	\int_{\Omega\setminus\Gamma_\delta}
	\frac{\rho}{2}|\dot{\mathbf u}|^2\,\mathrm dv .
\end{equation}
Therefore, the kinetic-energy balance becomes
\begin{equation}\label{eq:kinetic_balance_for_energy_thm_app}
	\begin{aligned}
		\frac{\mathrm d}{\mathrm dt}
		\int_{\Omega\setminus\Gamma_\delta}
		\frac{\rho}{2}|\dot{\mathbf u}|^2\,\mathrm dv
		&=
		-\int_\Omega
		\boldsymbol\sigma_\delta^{\mathrm d}:
		\dot{\boldsymbol\varepsilon}_\delta
		\,\mathrm dv
		+
		(\mathbf b,\dot{\mathbf u})_{L^2(\Omega\setminus\Gamma_\delta)}
		+
		(\mathbf t,\dot{\mathbf u})_{\Gamma_{\delta,t}} .
	\end{aligned}
\end{equation}
Since the external loading potential is defined by
\begin{equation}
	\ell_u(\mathbf u)
	=
	(\mathbf b,\mathbf u)_{L^2(\Omega\setminus\Gamma_\delta)}
	+
	(\mathbf t,\mathbf u)_{\Gamma_{\delta,t}},
\end{equation}
and the external data are assumed to be time-independent, we have
\begin{equation}\label{eq:external_loading_potential_derivative_app}
	\frac{\mathrm d}{\mathrm dt}\ell_u(\mathbf u)
	=
	(\mathbf b,\dot{\mathbf u})_{L^2(\Omega\setminus\Gamma_\delta)}
	+
	(\mathbf t,\dot{\mathbf u})_{\Gamma_{\delta,t}} .
\end{equation}
Next, using
\[
\mathcal W_{\delta,de}(\mathbf u,\psi)
=
g(\psi)\mathcal W_{\delta,e}^{+}(\mathbf u)
+
\mathcal W_{\delta,e}^{-}(\mathbf u),
\qquad
g(\psi)=(1-\psi)^2,
\]
we obtain by the chain rule
\begin{equation}
	\begin{aligned}
		\frac{\mathrm d}{\mathrm dt}
		\int_\Omega
		\mathcal W_{\delta,de}(\mathbf u,\psi)\,\mathrm dv
		&=
		\int_\Omega
		\boldsymbol\sigma_\delta^{\mathrm d}:
		\dot{\boldsymbol\varepsilon}_\delta
		\,\mathrm dv
		+
		\int_\Omega
		g'(\psi)\mathcal W_{\delta,e}^{+}(\mathbf u)\dot\psi
		\,\mathrm dv  \\
		&=
		\int_\Omega
		\boldsymbol\sigma_\delta^{\mathrm d}:
		\dot{\boldsymbol\varepsilon}_\delta
		\,\mathrm dv
		-
		\int_\Omega
		2(1-\psi)\mathcal W_{\delta,e}^{+}(\mathbf u)\dot\psi
		\,\mathrm dv .
	\end{aligned}
	\label{eq:strain_energy_chain_rule_for_energy_thm_app}
\end{equation}
For the nonlocal crack-surface functional, the self-adjointness of
$\mathcal L_\delta$ gives
\begin{equation}
	\frac{\mathrm d}{\mathrm dt}
	\mathcal A_\delta(\psi,\mathcal L_\delta\psi)
	=
	\int_\Omega
	\frac{1}{\delta}
	\left(
	\psi+\delta^2\mathcal L_\delta\psi
	\right)
	\dot\psi
	\,\mathrm dv .
	\label{eq:crack_surface_derivative_for_energy_thm_app}
\end{equation}
Indeed, from
\begin{equation}
	\mathcal A_\delta(\psi,\mathcal L_\delta\psi)
	=
	\int_\Omega
	\frac{1}{2\delta}\psi^2\,\mathrm dv
	+
	\frac{\delta}{4}
	\int_\Omega\int_\Omega
	\omega_\delta(|\mathbf y-\mathbf x|)
	\big(\psi(\mathbf x)-\psi(\mathbf y)\big)^2
	\,\mathrm dv_{\mathbf x}\mathrm dv_{\mathbf y},
\end{equation}
one obtains
\begin{equation}
	\frac{\mathrm d}{\mathrm dt}
	\mathcal A_\delta(\psi,\mathcal L_\delta\psi)
	=
	\int_\Omega
	\frac{1}{\delta}\psi\dot\psi\,\mathrm dv
	+
	\delta
	\int_\Omega
	\mathcal L_\delta\psi\,\dot\psi
	\,\mathrm dv ,
\end{equation}
which is exactly \eqref{eq:crack_surface_derivative_for_energy_thm_app}.

Now multiply the phase-field equation
\begin{equation}
	\bar{\eta}\dot\psi
	=
	\mathcal Y(\psi,D)
	-
	\left(
	\psi+\delta^2\mathcal L_\delta\psi
	\right)
\end{equation}
by $\frac{G_c}{\delta}\dot\psi$ and integrate over $\Omega$.
Using
\begin{equation}
	\mathcal Y(\psi,D)
	=
	(1-\psi)D,
	\qquad
	D=
	\frac{2\mathcal W_{\delta,e}^{+}(\mathbf u)}{G_c/\delta},
\end{equation}
we obtain
\begin{equation}
	\frac{\bar{\eta}G_c}{\delta}
	\int_\Omega |\dot\psi|^2\,\mathrm dv
	=
	\int_\Omega
	2(1-\psi)\mathcal W_{\delta,e}^{+}(\mathbf u)\dot\psi
	\,\mathrm dv
	-
	G_c
	\frac{\mathrm d}{\mathrm dt}
	\mathcal A_\delta(\psi,\mathcal L_\delta\psi).
	\label{eq:phase_balance_for_energy_thm_app}
\end{equation}

Substituting \eqref{eq:phase_balance_for_energy_thm_app} into
\eqref{eq:strain_energy_chain_rule_for_energy_thm_app} yields
\begin{equation}
	\begin{aligned}
		\frac{\mathrm d}{\mathrm dt}
		\int_\Omega
		\mathcal W_{\delta,de}(\mathbf u,\psi)\,\mathrm dv
		&=
		\int_\Omega
		\boldsymbol\sigma_\delta^{\mathrm d}:
		\dot{\boldsymbol\varepsilon}_\delta
		\,\mathrm dv
		-
		\frac{\bar{\eta}G_c}{\delta}
		\int_\Omega |\dot\psi|^2\,\mathrm dv  \\
		&\quad
		-
		G_c
		\frac{\mathrm d}{\mathrm dt}
		\mathcal A_\delta(\psi,\mathcal L_\delta\psi).
	\end{aligned}
\end{equation}
Rearranging gives
\begin{equation}\label{eq:strain_fracture_balance_for_energy_thm_app}
	\begin{aligned}
		&
		\frac{\mathrm d}{\mathrm dt}
		\int_\Omega
		\mathcal W_{\delta,de}(\mathbf u,\psi)\,\mathrm dv
		+
		G_c
		\frac{\mathrm d}{\mathrm dt}
		\mathcal A_\delta(\psi,\mathcal L_\delta\psi)
		\\
		&\qquad
		=
		\int_\Omega
		\boldsymbol\sigma_\delta^{\mathrm d}:
		\dot{\boldsymbol\varepsilon}_\delta
		\,\mathrm dv
		-
		\frac{\bar{\eta}G_c}{\delta}
		\int_\Omega |\dot\psi|^2\,\mathrm dv .
	\end{aligned}
\end{equation}

Finally, adding \eqref{eq:kinetic_balance_for_energy_thm_app} and
\eqref{eq:strain_fracture_balance_for_energy_thm_app}, and subtracting
\eqref{eq:external_loading_potential_derivative_app}, we obtain
\begin{equation}
	\begin{aligned}
		\frac{\mathrm d}{\mathrm dt}
		\left[
		\int_{\Omega\setminus\Gamma_\delta}
		\frac{\rho}{2}|\dot{\mathbf u}|^2\,\mathrm dv
		+
		\int_\Omega
		\mathcal W_{\delta,de}(\mathbf u,\psi)\,\mathrm dv
		+
		G_c\mathcal A_\delta(\psi,\mathcal L_\delta\psi)
		-
		\ell_u(\mathbf u)
		\right]
		=
		-
		\frac{\bar{\eta}G_c}{\delta}
		\int_\Omega |\dot\psi|^2\,\mathrm dv .
	\end{aligned}
\end{equation}
That is,
\begin{equation}
	\frac{\mathrm d}{\mathrm dt}\mathcal H(t)
	=
	-\eta
	\int_\Omega |\dot\psi|^2\,\mathrm dv
	\le 0,
	\qquad
	\eta:=\frac{\bar{\eta}G_c}{\delta}.
\end{equation}
This proves the auxiliary identity used in the proof of
Theorem~\ref{thm:Hamiltonian_balance}.

 \bibliographystyle{elsarticle-num} 
 \bibliography{cas-refs}

@article{gunzburger2010nonlocal,
  title={A nonlocal vector calculus with application to nonlocal boundary value problems},
  author={Gunzburger, Max and Lehoucq, Richard B},
  journal={Multiscale Modeling \& Simulation},
  volume={8},
  number={5},
  pages={1581--1598},
  year={2010},
  publisher={SIAM}
}

@article{du2013nonlocal,
  title={A nonlocal vector calculus, nonlocal volume-constrained problems, and nonlocal balance laws},
  author={Du, Qiang and Gunzburger, Max and Lehoucq, Richard B and Zhou, Kun},
  journal={Mathematical Models and Methods in Applied Sciences},
  volume={23},
  number={03},
  pages={493--540},
  year={2013},
  publisher={World Scientific}
}

@article{d2021cookbook,
  title={A cookbook for approximating Euclidean balls and for quadrature rules in finite element methods for nonlocal problems},
  author={D’Elia, Marta and Gunzburger, Max and Vollmann, Christian},
  journal={Mathematical Models and Methods in Applied Sciences},
  volume={31},
  number={08},
  pages={1505--1567},
  year={2021},
  publisher={World Scientific}
}

@article{silling2000reformulation,
  title={Reformulation of elasticity theory for discontinuities and long-range forces},
  author={Silling, Stewart A},
  journal={Journal of the Mechanics and Physics of Solids},
  volume={48},
  number={1},
  pages={175--209},
  year={2000},
  publisher={Elsevier}
}

@article{silling2007peridynamic,
  title={Peridynamic states and constitutive modeling},
  author={Silling, Stewart A and Epton, M and Weckner, Olaf and Xu, Jifeng and Askari, E23481501120},
  journal={Journal of elasticity},
  volume={88},
  pages={151--184},
  year={2007},
  publisher={Springer}
}

@article{zhou2010mathematical,
	title={Mathematical and numerical analysis of linear peridynamic models with nonlocal boundary conditions},
	author={Zhou, Kun and Du, Qiang},
	journal={SIAM Journal on Numerical Analysis},
	volume={48},
	number={5},
	pages={1759--1780},
	year={2010},
	publisher={SIAM}
}

@article{Bourdin2000,
	author  = {Bourdin, Blaise and Francfort, Gilles\,A. and Marigo, Jean‑Jacques},
	title   = {Numerical Experiments in Revisited Brittle Fracture},
	journal = {Journal of the Mechanics and Physics of Solids},
	volume  = {48},
	number  = {4},
	pages   = {797--826},
	year    = {2000},
}

@incollection{Wu2020,
	author    = {Wu, Jian‑Ying and Nguyen, Vinh Phu and Nguyen, Chi Thanh and Sutula, Danas and Sinaie, Sina and Bordas, Stéphane P. A.},
	title     = {Phase‑field modeling of fracture},
	booktitle = {Advances in Applied Mechanics},
	editor    = {Bordas, Stéphane P. A. and Balint, Daniel},
	series    = {Advances in Applied Mechanics},
	volume    = {53},
	pages     = {1--183},
	year      = {2020},
	publisher = {Elsevier},
}

@article{Miehe2015a,
	author  = {Miehe, Christian and Sch{\"a}nzel, Lars‑Martin and Ulmer, Henk},
	title   = {Phase‑field modeling of fracture in multi‑physics problems. Part-I. Balance of crack surface and failure criteria for brittle crack propagation in thermo‑elastic solids},
	journal = {Computer Methods in Applied Mechanics and Engineering},
	volume  = {294},
	pages   = {449--485},
	year    = {2015},
}

@article{Miehe2010,
	author  = {Miehe, Christian and Hofacker, Martin and Welschinger, Frank},
	title   = {A phase field model for rate‑independent crack propagation: robust algorithmic implementation based on operator splits},
	journal = {Computer Methods in Applied Mechanics and Engineering},
	volume  = {199},
	number  = {45–48},
	pages   = {2765--2778},
	year    = {2010},
}

@article{song2008comparative,
  title={A comparative study on finite element methods for dynamic fracture},
  author={Song, Jeong-Hoon and Wang, Hongwu and Belytschko, Ted},
  journal={Computational Mechanics},
  volume={42},
  number={2},
  pages={239--250},
  year={2008},
  publisher={Springer}
}

@article{xu1994numerical,
  title={Numerical simulations of fast crack growth in brittle solids},
  author={Xu, X-P and Needleman, Alan},
  journal={Journal of the Mechanics and Physics of Solids},
  volume={42},
  number={9},
  pages={1397--1434},
  year={1994},
  publisher={Elsevier}
}

@article{belytschko2003dynamic,
  title={Dynamic crack propagation based on loss of hyperbolicity and a new discontinuous enrichment},
  author={Belytschko, Ted and Chen, Hao and Xu, Jingxiao and Zi, Goangseup},
  journal={International journal for numerical methods in engineering},
  volume={58},
  number={12},
  pages={1873--1905},
  year={2003},
  publisher={Wiley Online Library}
}

@article{kalthoff2000modes,
  title={Modes of dynamic shear failure in solids},
  author={Kalthoff, Joerg F},
  journal={International Journal of fracture},
  volume={101},
  number={1},
  pages={1--31},
  year={2000},
  publisher={Springer}
}

@book{belytschko2014nonlinear,
  title={Nonlinear finite elements for continua and structures},
  author={Belytschko, Ted and Liu, Wing Kam and Moran, Brian and Elkhodary, Khalil},
  year={2014},
  publisher={John wiley \& sons}
}

@book{lemaitre2012course,
  title={A course on damage mechanics},
  author={Lemaitre, Jean},
  year={2012},
  publisher={Springer science \& business media}
}

@article{wells2001new,
  title={A new method for modelling cohesive cracks using finite elements},
  author={Wells, Garth N and Sluys, LJ1013},
  journal={International Journal for numerical methods in engineering},
  volume={50},
  number={12},
  pages={2667--2682},
  year={2001},
  publisher={Wiley Online Library}
}

@article{moes1999finite,
  title={A finite element method for crack growth without remeshing},
  author={Mo{\"e}s, Nicolas and Dolbow, John and Belytschko, Ted},
  journal={International journal for numerical methods in engineering},
  volume={46},
  number={1},
  pages={131--150},
  year={1999},
  publisher={Wiley Online Library}
}

@article{daux2000arbitrary,
  title={Arbitrary branched and intersecting cracks with the extended finite element method},
  author={Daux, Christophe and Mo{\"e}s, Nicolas and Dolbow, John and Sukumar, Natarajan and Belytschko, Ted},
  journal={International journal for numerical methods in engineering},
  volume={48},
  number={12},
  pages={1741--1760},
  year={2000},
  publisher={Wiley Online Library}
}

@article{rabczuk2004cracking,
  title={Cracking particles: a simplified meshfree method for arbitrary evolving cracks},
  author={Rabczuk, Timon and Belytschko, Ted},
  journal={International journal for numerical methods in engineering},
  volume={61},
  number={13},
  pages={2316--2343},
  year={2004},
  publisher={Wiley Online Library}
}

@article{rabczuk2007three,
  title={A three-dimensional large deformation meshfree method for arbitrary evolving cracks},
  author={Rabczuk, Timon and Belytschko, T23253911128},
  journal={Computer methods in applied mechanics and engineering},
  volume={196},
  number={29-30},
  pages={2777--2799},
  year={2007},
  publisher={Elsevier}
}

@article{geers1998strain,
  title={Strain-based transient-gradient damage model for failure analyses},
  author={Geers, MGD and De Borst, R and Brekelmans, WAM and Peerlings, RHJ0938},
  journal={Computer methods in applied mechanics and engineering},
  volume={160},
  number={1-2},
  pages={133--153},
  year={1998},
  publisher={Elsevier}
}

@article{peerlings2002localisation,
  title={Localisation issues in local and nonlocal continuum approaches to fracture},
  author={Peerlings, RHJ d and De Borst, R and Brekelmans, WAM and Geers, MGD1903091},
  journal={European Journal of Mechanics-A/Solids},
  volume={21},
  number={2},
  pages={175--189},
  year={2002},
  publisher={Elsevier}
}

@article{bavzant2002nonlocal,
  title={Nonlocal integral formulations of plasticity and damage: survey of progress},
  author={Ba{\v{z}}ant, Zden{\v{e}}k P and Jir{\'a}sek, Milan},
  journal={Journal of engineering mechanics},
  volume={128},
  number={11},
  pages={1119--1149},
  year={2002},
  publisher={American Society of Civil Engineers}
}

@article{foulk2010examination,
  title={An examination of stability in cohesive zone modeling},
  author={Foulk III, JW},
  journal={Computer methods in applied mechanics and engineering},
  volume={199},
  number={9-12},
  pages={465--470},
  year={2010},
  publisher={Elsevier}
}

@article{nguyen2014open,
  title={An open source program to generate zero-thickness cohesive interface elements},
  author={Nguyen, Vinh Phu},
  journal={Advances in Engineering Software},
  volume={74},
  pages={27--39},
  year={2014},
  publisher={Elsevier}
}

@article{wu2019computational,
  title={Computational modeling of localized failure in solids: XFEM vs PF-CZM},
  author={Wu, Jian-Ying and Qiu, Jie-Feng and Nguyen, Vinh Phu and Mandal, Tushar Kanti and Zhuang, Luo-Jia},
  journal={Computer Methods in Applied Mechanics and Engineering},
  volume={345},
  pages={618--643},
  year={2019},
  publisher={Elsevier}
}

@article{bourdin2008variational,
  title={The variational approach to fracture},
  author={Bourdin, Blaise and Francfort, Gilles A and Marigo, Jean-Jacques},
  journal={Journal of elasticity},
  volume={91},
  number={1},
  pages={5--148},
  year={2008},
  publisher={Springer}
}

@article{miehe2010phase,
  title={A phase field model for rate-independent crack propagation: Robust algorithmic implementation based on operator splits},
  author={Miehe, Christian and Hofacker, Martina and Welschinger, Fabian},
  journal={Computer Methods in Applied Mechanics and Engineering},
  volume={199},
  number={45-48},
  pages={2765--2778},
  year={2010},
  publisher={Elsevier}
}

@article{bourdin2000numerical,
  title={Numerical experiments in revisited brittle fracture},
  author={Bourdin, Blaise and Francfort, Gilles A and Marigo, Jean-Jacques},
  journal={Journal of the Mechanics and Physics of Solids},
  volume={48},
  number={4},
  pages={797--826},
  year={2000},
  publisher={Elsevier}
}

@article{budarapu2019multiscale,
  title={Multiscale modeling of material failure: Theory and computational methods},
  author={Budarapu, Pattabhi Ramaiah and Zhuang, Xiaoying and Rabczuk, Timon and Bordas, Stephane PA},
  journal={Advances in applied mechanics},
  volume={52},
  pages={1--103},
  year={2019},
  publisher={Elsevier}
}

@article{borden2012phase,
  title={A phase-field description of dynamic brittle fracture},
  author={Borden, Michael J and Verhoosel, Clemens V and Scott, Michael A and Hughes, Thomas JR and Landis, Chad M},
  journal={Computer Methods in Applied Mechanics and Engineering},
  volume={217},
  pages={77--95},
  year={2012},
  publisher={Elsevier}
}

@article{du2013analysis,
	title={Analysis of the volume-constrained peridynamic Navier equation of linear elasticity},
	author={Du, Qiang and Gunzburger, Max and Lehoucq, RB and Zhou, Kun},
	journal={Journal of Elasticity},
	volume={113},
	number={2},
	pages={193--217},
	year={2013},
	publisher={Springer}
}

@article{wu2020phase,
	title={Phase-field modeling of fracture},
	author={Wu, Jian-Ying and Nguyen, Vinh Phu and Nguyen, Chi Thanh and Sutula, Danas and Sinaie, Sina and Bordas, St{\'e}phane PA},
	journal={Advances in applied mechanics},
	volume={53},
	pages={1--183},
	year={2020},
	publisher={Elsevier}
}

@article{sukumar2000extended,
	title={Extended finite element method for three-dimensional crack modelling},
	author={Sukumar, Natarajan and Mo{\"e}s, Nicolas and Moran, Brian and Belytschko, Ted},
	journal={International journal for numerical methods in engineering},
	volume={48},
	number={11},
	pages={1549--1570},
	year={2000},
	publisher={Wiley Online Library}
}

@article{song2009cracking,
	title={Cracking node method for dynamic fracture with finite elements},
	author={Song, Jeong-Hoon and Belytschko, Ted},
	journal={International Journal for Numerical Methods in Engineering},
	volume={77},
	number={3},
	pages={360--385},
	year={2009},
	publisher={Wiley Online Library}
}

@article{ren2025dual,
	title={Dual-horizon peridynamics-based variational damage modeling for complex dynamic fractures},
	author={Ren, Huilong and Zhuang, Xiaoying and Bie, Yehui and Rabczuk, Timon and Zhu, Hehua},
	journal={Theoretical and Applied Fracture Mechanics},
	volume={138},
	pages={104974},
	year={2025},
	publisher={Elsevier}
}





% \end{thebibliography}
\end{document}